\documentclass[12pt]{article}

\usepackage{verbatim}

\usepackage{amsfonts}

\usepackage{amsthm}

\usepackage{amssymb}

\usepackage{amsmath}

\usepackage{mathabx}

\usepackage{enumerate}

\usepackage[all]{xy}

\usepackage{graphicx}

\usepackage{centernot}

\usepackage{mathtools}

\usepackage{stmaryrd}

\usepackage[pagebackref,colorlinks,linkcolor=blue,citecolor=blue,urlcolor=blue,hypertexnames=true]{hyperref}

\usepackage[pagebackref,hypertexnames=true]{hyperref}

\newcommand{\N}{\mathbb{N}}

\newcommand{\R}{\mathbb{R}}

\newcommand{\C}{\mathbb{C}}

\newcommand{\LL}{\mathcal{L}}

\newcommand{\CL}{C_{L,u}}

\newcommand{\CLc}{C_{L,c}}

\newcommand{\IL}{I_{L,u}}

\newcommand{\ILc}{I_{L,c}}

\newcommand{\QL}{Q_L}

\newcommand{\K}{\mathcal{K}}

\newcommand{\KKP}{KK_{\mathcal{P}}}

\newcommand{\Hawaii}{Hawai\kern.05em`\kern.05em\relax i}

\theoremstyle{plain}
\newtheorem{theorem}{Theorem}[section]
\newtheorem{lemma}[theorem]{Lemma}
\newtheorem{corollary}[theorem]{Corollary}
\newtheorem{proposition}[theorem]{Proposition}

\newtheorem{definition-theorem}[theorem]{Definition / Theorem}

\newtheorem*{conjecture*}{Conjecture}
\newtheorem*{theorem*}{Theorem}

\theoremstyle{definition}
\newtheorem{definition}[theorem]{Definition}

\theoremstyle{remark}
\newtheorem{remark}[theorem]{Remark}

\newtheorem*{example*}{Example}  
\newtheorem*{remark*}{Remark}

\title{Controlled KK-theory, and a Milnor exact sequence}

\author{Rufus Willett and Guoliang Yu}

\begin{document}

\maketitle

\begin{abstract}
We introduce controlled $KK$-theory groups associated to a pair $(A,B)$ of separable $C^*$-algebras.  Roughly, these consist of elements of the usual $K$-theory group $K_0(B)$ that approximately commute with elements of $A$.  Our main results show that these groups are related to Kasparov's $KK$-groups by a Milnor exact sequence, in such a way that R\o{}rdam's $KL$-group is identified with an inverse limit of our controlled $KK$-groups.

In the case that the $C^*$-algebras involved satisfy the UCT, our Milnor exact sequence agrees with the Milnor sequence associated to a $KK$-filtration in the sense of Schochet, although our results are independent of the UCT.  Applications to the UCT will be pursued in subsequent work.
\end{abstract}

\tableofcontents

\section{Introduction}

Given two $C^*$-algebras $A$ and $B$, Kasparov associated an abelian group $KK(A,B)$ of generalized morphisms between $A$ and $B$.  The Kasparov $KK$-groups were designed to have applications to index theory and the Novikov conjecture \cite{Kasparov:1988dw}, but now play a fundamental role in many aspects of $C^*$-algebra theory (and elsewhere).  This is particularly true in the Elliott program \cite{Elliott:1995dq} to classify $C^*$-algebras by $K$-theoretic invariants.

Our immediate goal in this paper is to introduce \emph{controlled $KK$-theory groups} and relate them to Kasparov's $KK$-theory groups.  The idea -- which we will pursue in subsequent work -- is that the controlled groups allow more flexibility in computations.  Our groups are analogues of the controlled $K$-theory groups introduced by the second author as part of his work on the Novikov conjecture \cite{Yu:1998wj}, and later developed by him in collaboration with Oyono-Oyono \cite{Oyono-Oyono:2011fk}.  Having said that, our approach in this paper is independent of, and in some sense dual to, these earlier developments: controlled $K$-theory abstracts the approach to the Novikov conjecture through operators of controlled propagation as in \cite{Yu:1998wj}, while the controlled $KK$-theory we introduce here abstracts the `dual' approach to the Novikov conjecture through almost flat bundles (see for example \cite{Connes:1990vo} and \cite[Chapter 11]{Willett:2010ay}).  

Our larger goal is to establish a new sufficient condition for a nuclear $C^*$-algebra to satisfy the UCT of Rosenberg and Schochet \cite{Rosenberg:1987bh}, analogously to recent results on the K\"{u}nneth formula using controlled $K$-theory ideas \cite{Oyono-Oyono:2016qd,Willett:2019aa}.  These applications appear in the companion paper \cite{Willett:2021te}.    

Our goal in this paper is to establish the basic theory, which we hope will be useful in other settings: indeed, since the initial submission of this paper, we also found quite different applications to representation stability in group theory \cite{Willett:2024aa}.

\subsection{Controlled $KK$-theory and the Milnor sequence}

We now discuss a version of our controlled $KK$-theory groups in more detail.

Let $B$ be a separable $C^*$-algebra, let $B\otimes \K$ be its stabilization, and let $M(B\otimes \K)$ be its stable multiplier algebra.  Define $\mathcal{P}(B)$ to consist of all projections in $p\in M_2(M(B\otimes\K))$ such that $p-\begin{psmallmatrix} 1 & 0 \\ 0 & 0 \end{psmallmatrix}$ is in the ideal $M_2(B\otimes \K)$.  Then the formula
\begin{equation}\label{k bij}
 \pi_0(\mathcal{P}(B))\to K_0(B),\quad [p]\mapsto [p]-\begin{bsmallmatrix} 1 & 0 \\ 0 & 0 \end{bsmallmatrix}
\end{equation}
gives a bijection from the set of path components of $\mathcal{P}(B)$ to the usual $K_0$-group of $B$.

Now, assume for simplicity that $A$ is a separable, unital, and nuclear\footnote{The assumptions of unitality and nuclearity are not necessary, but simplify the definitions - see the body of the paper for the general versions.  The fact that the version we give here is equivalent to the general definition is a consequence of Proposition \ref{usa rep gives kl} and Remark \ref{kas rem}.} $C^*$-algebra.  Let $\pi:A\to \mathcal{B}(\ell^2)$ be an infinite amplification of a faithful unital representation\footnote{i.e.\ take a faithful unital representation on a separable Hilbert space, and add it to itself countably many times.  It turns out the choices involved here do not matter in any serious way.}, and use the composition 
$$
A\to \mathcal{B}(\ell^2)=M(\K)\subseteq M(B\otimes \K)
$$
of $\pi$ and the canonical inclusion of $M(\K)$ into $M(B\otimes \K)$ to consider $A$ as a $C^*$-subalgebra of $M(B\otimes \K)$.  Having $A$ act as diagonal matrices, we may also identify $A$ with a $C^*$-subalgebra of $M_2(M(B\otimes \K))$.  For a finite subset $X$ of $A$ and $\epsilon>0$, define  
$$
\mathcal{P}_{\epsilon}(X,B):=\{p\in \mathcal{P}(B)\mid \|[p,a]\|<\epsilon \text{ for all } a\in X\}.
$$
Define the \emph{controlled $KK$-theory group\footnote{It is a group in a natural way, in a way that is compatible with the group structure on $K_0(B)$ via the isomorphism in line \eqref{k bij}.} associated to $X$ and $\epsilon$} to be 
$$
KK_\epsilon(X,B):=\pi_0(\mathcal{P}_{\epsilon}(X,B)).
$$
Thanks to the isomorphism of line \eqref{k bij}, we think of $KK_\epsilon(X,B)$ as `the part\footnote{The word `part' is potentially misleading: there is a map $KK_\epsilon(X,B)\to K_0(B)$ defined by forgetting the commutation condition, but it would not be injective in general.} of $K_0(B)$ that commutes with $X$ up to $\epsilon$'.  This idea -- of considering elements of $K$-theory that asymptotically commute with some representation -- is partly inspired by the $E$-theory of Connes and Higson \cite{Connes:1990if}. 

Now, let $(X_n)$ be a nested sequence of finite subsets of $A$ with dense union, and let $(\epsilon_n)$ be a decreasing  sequence of positive numbers than tend to zero.  As it is easier to commute with $X_n$ up to $\epsilon_n$ that it is to commute with $X_{n+1}$ up to $\epsilon_{n+1}$, we get a sequence of `forget control' homomorphisms
$$
\cdots \to KK_{\epsilon_n}(X_n,B)\to KK_{\epsilon_{n-1}}(X_{n-1},B)\to \cdots \to KK_{\epsilon_1}(X_1,B).
$$
Thus we may build the inverse limit ${\displaystyle \lim_{\leftarrow} KK_{\epsilon_n}(X_n,B) }$ of abelian group theory associated to this sequence; moreover, we may consider this inverse limit as a topological abelian group by giving each $KK_\epsilon(X,B)$ the discrete topology and taking the inverse limit in the category of topological abelian groups.  Replacing $B$ with its suspension $SB$, we may also build the $\lim^1$-group\footnote{The $\lim^1$ functor is the first derived functor of the inverse limit functor.  See for example \cite[Section 3.5]{Weibel:1995ty} or  \cite[Chapter 3]{Schochet:2003vq} for concrete definitions of the inverse limit and $\lim^1$ groups, and for some examples of computations.  See \cite{Jensen:1972wu} for the general case.} ${\displaystyle \lim_{\leftarrow}\!{}^1  KK_{\epsilon_n} (X_n,SB)}$ associated to the corresponding sequence.  We are now ready to state a special case of our main theorem.

\begin{theorem}\label{main}
For any separable $C^*$-algebras $A$ and $B$ with $A$ unital and nuclear\footnote{There is also a very similar version for general separable $A$: see the main body of the paper.}, there is a short exact sequence
$$
\xymatrix{ 0 \ar[r] & {\displaystyle \lim_{\leftarrow}\!{}^1  KK_{\epsilon_n} (X_n,SB)}  \ar[r] & KK(A,B) \ar[r] & {\displaystyle \lim_{\leftarrow} KK_{\epsilon_n}(X_n,B) }  \ar[r] & 0 }.
$$
\end{theorem}

We will explain the idea of the proof below, but first give a more precise version involving R\o{}rdam's $KL$-groups \cite[Section 5]{Rordam:1995aa}, and some comparisons of the results with the previous literature.

\subsection{The topology on $KK$ and Schochet's Milnor sequence}

Recall that $KK(A,B)$ is equipped with a canonical topology, which makes it a (possibly non-Hausdorff) topological group.  This topology can be described in several different ways that turn out to be equivalent, as established by Dadarlat in \cite{Dadarlat:2005aa} (see also \cite{Schochet:2001vd}).  Define\footnote{The original definition of $KL(A,B)$ is due to R\o{}rdam \cite[Section 5]{Rordam:1995aa}, and only makes sense if the pair $(A,B)$ satisfies the UCT.  The definition we are using was suggested by Dadarlat \cite[Section 5]{Dadarlat:2005aa}, and is equivalent to R\o{}rdam's when $A$ satisfies the UCT by \cite[Theorem 4.1]{Dadarlat:2005aa}.} $KL(A,B)$ to be  to be the associated maximal Hausdorff quotient, i.e.\ the quotient $KK(A,B)/\overline{\{0\}}$ of $KK(A,B)$ by the closure of the zero element.

The following theorem relating our controlled $KK$-theory groups to the topology on $KK$ is a more precise version of Theorem \ref{main}, and is what we actually establish in the main body of the paper.

\begin{theorem}\label{main2}
For any separable $C^*$-algebras $A$ and $B$ with $A$ unital and nuclear\footnote{Again, these extra assumptions on $A$ are not necessary, with minor changes to the definitions.}, there are canonical isomorphisms 
$$
\lim_{\leftarrow}\!{}^1  KK_{\epsilon_n} (X_n,SB)\cong \overline{\{0\}}\quad \text{and} \quad \lim_{\leftarrow} KK_{\epsilon_n}(X_n,B) \cong KL(A,B).
$$
Moreover, the second isomorphism above is an isomorphism of topological groups.
\end{theorem}

The short exact sequence in Theorem \ref{main} is an analogue of Schochet's \emph{Milnor exact sequence} \cite{Schochet:1996aa} associated to a \emph{$KK$-filtration}.   A $KK$-filtration consists of a $KK$-equivalence of $A$ with the direct limit of an increasing sequence $(A_n)$ of $C^*$-algebras where each $A_n$ has unitization the continuous functions on some finite CW complex.   Schochet \cite[Theorem 1.5]{Schochet:1996aa} shows that such a filtration exists if and only if $A$ satisfies the UCT.   Schochet \cite[Theorem 1.5]{Schochet:1996aa} then shows that there is an exact sequence
$$
\xymatrix{ 0 \ar[r] & {\displaystyle \lim_{\leftarrow}{}\!^1  KK (A_n,SB) }\ar[r] & KK(A,B)  \ar[r] & {\displaystyle \lim_{\leftarrow} KK(A_n,B) }\ar[r] & 0 }.
$$
It follows from Theorem \ref{main2} and \cite[Proposition 4.1]{Schochet:2002aa} that our Milnor sequence from Theorem \ref{main} agrees with Schochet's when $A$ satisfies the UCT.   Our Milnor sequence can thus be thought of as a generalization of Schochet's sequence that works in the absence of the UCT.

\subsection{Discussion of proofs}

Continuing to assume for simplicity that $A$ is unital and nuclear, let us identify $A$ with a $C^*$-subalgebra of $M(B\otimes \K)$ as in the statement of Theorem \ref{main}.  Then we define $\mathcal{P}(A,B)$ to be the collection of all continuous, bounded, projection-valued functions $p:[1,\infty)\to M_2(M(B\otimes \K))$ such that $[p_t,a]\to 0$ as $t\to\infty$ for all $a\in A$, and so that $p_t-\begin{psmallmatrix} 1 & 0 \\ 0 & 0 \end{psmallmatrix}$ is in $M_2(B\otimes \K)$ for all $t$.  Define $\KKP(A,B)$ to be the quotient of $\mathcal{P}(A,B)$, modulo the equivalence relation one gets by saying $p$ and $q$ are homotopic if they are restrictions to the endpoints of an element of $\mathcal{P}(A,C[0,1]\otimes B)$.

One can then show\footnote{We do not actually show this, only the more general version where $A$ is non-unital and not-necessarily nuclear; nonetheless, this result follows directly from the same methods.} that $KK(A,B)$ is naturally isomorphic to $\KKP(A,B)$.  The first important ingredient in this is the description of $KK(A,B)$ as the $K$-theory of an appropriate \emph{localization algebra}, which was done by Dadarlat, Wu and the first author in \cite[Theorem 4.4]{Dadarlat:2016qc} (inspired by ideas of the second author in the case of commutative $C^*$-algebras \cite{Yu:1997kb}).    The other important (albeit implicit) ingredient we use for the isomorphism $KK(A,B)\cong \KKP(A,B)$ is the fundamental theorem of Kasparov (\cite[Section 6, Theorem 1]{Kasparov:1981vj}, and see also \cite[Theorem 19]{Skandalis:1984aa} and \cite[Section 18.5]{Blackadar:1998yq}) that the equivalence relations on Kasparov cycles induced by operator homotopy, and by homotopy, give rise to the same $KK$-groups.

Having described $KK(A,B)$ using continuous paths of projections, we can now also describe the topology on this group in this language: roughly, a sequence $(p^n)$ converges to $p$ in $\mathcal{P}(A,B)$ if for all $\epsilon>0$ and finite $X\subseteq A$ there is $t_0$ such that for all $t\geq t_0$, $p^n_t$ can be connected to $p_t$ via a homotopy passing through $\mathcal{P}_\epsilon(X,B)$.  This topology on $\mathcal{P}(A,B)$ induces a topology on $\KKP(A,B)$, and we show that this topology agrees with the usual one on $KK(A,B)$ using an abstract characterization of the latter due to Dadarlat \cite[Section 3]{Dadarlat:2005aa} (and based on ideas of Pimsner).

Having got this far, it is not too difficult to see that there is a well-defined map 
\begin{equation}\label{to kl}
\KKP(A,B)\to \lim_{\leftarrow} KK_{\epsilon_n}(X_n,B)
\end{equation}
defined by evaluating a path $(p_t)_{t\in [1,\infty)}$ in $\mathcal{P}(A,B)$ at larger and larger values of $t$, and that there is a well-defined map
\begin{equation}\label{to clo0}
\lim_{\leftarrow}{}{\!}^1 KK_{\epsilon_n}(X_n,SB) \to \KKP(A,B)
\end{equation}
defined by treating an element of $\mathcal{P}_\epsilon(X,SB)$ as a projection-valued function from $[0,1]$ to $\mathcal{P}_\epsilon(X,B)$ that agrees with $\begin{psmallmatrix} 1 & 0 \\ 0 & 0 \end{psmallmatrix}$ at its endpoints, and stringing a countable sequence of these together to get an element $(p_t)_{t\in [1,\infty)}$ in $\mathcal{P}(A,B)$.  Moreover, it is essentially true by definition that the map in line \eqref{to kl} contains the closure of $\{0\}$ in its kernel so descends to a map
$$
\KKP(A,B)/\overline{\{0\}}\to \lim_{\leftarrow} KK_{\epsilon_n}(X_n,B)
$$
and the map in line \eqref{to clo0} takes image if the closure of $\{0\}$ so corestricts to a map
$$
\lim_{\leftarrow}{}{\!}^1 KK_{\epsilon_n}(X_n,SB)\to \overline{\{0\}}.
$$
To establish Theorem \ref{main2}, we show that the maps in the two previous displayed lines are isomorphisms.

\subsection{Notation and conventions}\label{n n c}

We write $\ell^2$ for $\ell^2(\N)$. 

Throughout, the symbols $A$ and $B$ are reserved for separable $C^*$-algebras; the letters $C$, $D$ and others may sometimes refer to non-separable $C^*$-algebras. The unit ball of a $C^*$-algebra $C$ is denoted by $C_1$, its unitization is $C^+$, and its multiplier algebra is $M(C)$.   

Our conventions on Hilbert modules follow those of Lance \cite{Lance:1995ys}.   We will always assume that Hilbert modules are over separable $C^*$-algebras, and are countably generated as discussed on \cite[page 60]{Lance:1995ys}.   If it is not explicitly specified otherwise, all Hilbert modules will be over the $C^*$-algebra called $B$.   For Hilbert $B$-modules $E$ and $F$, we write $\LL(E,F)$ (respectively $\K(E,F)$) for the spaces of adjointable (respectively compact) operators from $E$ to $F$ in the usual sense of Hilbert module theory \cite[pages 8-10]{Lance:1995ys}.  We use the standard  shorthands $\LL(E):=\LL(E,E)$ and $\K(E):=\K(E,E)$.  

In this paper, unless stated otherwise, a \emph{representation of $A$} will refer to a representation of $A$ on a Hilbert \emph{module}, i.e.\ a $*$-homomorphism $\pi:A\to \LL(E)$ for some Hilbert module $E$ (almost always over $B$, as above).  For $a\in A$, we generally just write $a\in \LL(E)$ for $\pi(a)$ to avoid notational clutter (even if $\pi$ is not assumed injective).  We write $E^\infty$ for the (completed) infinite direct sum Hilbert module $\bigoplus_{n=1}^\infty E=\ell^2\otimes E$ \cite[page 6]{Lance:1995ys}, and if $\pi:A\to \LL(E)$ is a representation, we write $\pi^\infty:A\to \LL(E^\infty)$ for the  amplified representation, so in tensor product language $\pi^\infty=1_{\ell^2}\otimes \pi:A\to \LL(\ell^2\otimes E)$.  We say a representation $(\pi,E)$ has \emph{infinite multiplicity} if it isomorphic to $(\sigma^\infty,F^\infty)$ for some other representation $(\sigma,F)$.

The symbol `$\otimes$' always denotes a completed tensor product: either the (external or internal) tensor product of Hilbert modules \cite[Chapter 4]{Lance:1995ys}, or the minimal tensor product of $C^*$-algebras.  

If $E$ is a Banach space and $X$ a locally compact Hausdorff space, we let $C_{b}(X,E)$ (respectively, $C_{ub}(X,E)$, $C_{0}(X,E)$) denote the Banach space of continuous and bounded (respectively uniformly continuous and bounded, continuous and vanishing at infinity) functions from $X$ to $E$.  We write elements of these spaces as $e$ or $(e_x)_{x\in X}$, with $e_x\in E$ denoting the value of $e$ at a point $x\in X$.    We will sometimes say that $e$ is a `...' if it is a pointwise a `...': for example, `$u\in C_{b}([1,\infty),\LL(F_1,F_2))$ is unitary' means `$u_t$ is unitary in $\LL(F_1,F_2)$ for all $t\in [1,\infty)$'; if $E$ is a $C^*$-algebra so $C_b(X,E)$ is also a $C^*$-algebra, then this is consistent with the standard use of `unitary' and so on.  With $u\in C_{b}([1,\infty),\LL(F_1,F_2))$ as above, if $b$ is an element of $\LL(F_1)$ we write $ub$ for the function $t\mapsto u_t b$ in $C_{ub}([1,\infty),\LL(F_1,F_2))$ and similarly for $cu$ with $c\in \LL(F_2)$ and so on.

For $K$-theory, $K_*(A):=K_0(A)\oplus K_1(A)$ denotes the graded $K$-theory group of a $C^*$-algebra, and $KK_*(A,B):=KK_0(A,B)\oplus KK_1(A,B)$ the graded $KK$-theory group.  We will typically just write $KK(A,B)$ instead of $KK_0(A,B)$.

\subsection{Outline of the paper}

Sections \ref{sa sec} and \ref{loc sec} are background.  In Section \ref{sa sec} we recall some facts about `absorbing' representations.  Most of the material in Section \ref{sa sec} is essentially from papers of Kasparov \cite{Kasparov:1980sp}, Thomsen \cite{Thomsen:2000aa}, Dadarlat-Eilers \cite{Dadarlat:2001aa,Dadarlat:2002aa}, and Dadarlat \cite{Dadarlat:2005aa}.  In Section \ref{loc sec} we recall the localization algebra of Dadarlat, Wu and the first author \cite{Dadarlat:2016qc} (inspired by earlier ideas of the second author \cite{Yu:1997kb}), and establish some technical results about this.  

Sections \ref{pp sec} and \ref{top sec} introduce a group $\KKP(A,B)$ that consists of homotopy classes of paths of projections that asymptotically commute with $A$ and relate it to $KK$-theory: the culminating results show that $KK(A,B)$ and $\KKP(A,B)$ are isomorphic as topological groups.  In Section \ref{pp sec} we introduce $\KKP(A,B)$, show that it is a commutative monoid, and then that it is isomorphic to $KK(A,B)$ (whence a group).  In Section \ref{top sec} we introduce a topology on $\KKP(A,B)$.  We then use a characterization of Dadarlat \cite[Section 3]{Dadarlat:2005aa} to identify this with the canonical topology on $KK(A,B)$ that was introduced and studied by Brown-Salinas, Schochet, Pimsner, and Dadarlat in various guises.

Sections \ref{kl sec} and \ref{zero sec} establish Theorem \ref{main2} (and therefore Theorem \ref{main}).  Section \ref{kl sec} identifies the quotient $\KKP(A,B)/\overline{\{0\}}$ with ${\displaystyle \lim_{\leftarrow} KK_\epsilon(X,B)}$ (and therefore identifies $KL(A,B)$ with this inverse limit).  Section \ref{zero sec} identifies the closure of zero in $\KKP(A,B)$ with the appropriate $\lim^1$ group, completing the proof of the main results.

Finally, Appendix \ref{alt pic} gives some alternative pictures of our controlled $KK$-groups that will be useful for our subsequent work.  In particular, we give a slightly simplified picture in the case that $A$ is unital.

\subsection{Acknowledgements}

This paper was written parallel to work of Sarah Browne and Nate Brown on a controlled version of $E$-theory; we are grateful to Browne and Brown for several useful discussions around these subjects.  Many of the results of Section \ref{loc sec} were obtained independently (and earlier) by Jianchao Wu, in many cases with different proofs.  We are grateful to Wu for discussing this with us, and allowing us to include those results here.  We were inspired to learn about $KL$-theory, and connect our controlled $KK$-groups to it, by comments of Marius Dadarlat and Jamie Gabe, and thank them for their suggestions.  We also thank Claude Schochet for some useful suggestions and references.  Finally, we thank the anonymous referees for many useful suggestions.

The authors gratefully acknowledge the support of the US NSF (DMS 1564281, DMS 1700021, DMS 1901522, DMS 2000082, DMS 2247313, DMS 2247968) and the Simons Foundation throughout the writing and revising of this paper.

\section{Strongly absorbing representations}\label{sa sec}

Throughout this section, $A$ and $B$ refer to separable $C^*$-algebras.  All Hilbert modules are countably generated, and all are over $B$ unless explicitly stated otherwise.  All representations of $A$ are on Hilbert $B$-modules unless explicitly stated otherwise.

In this section we establish conventions and terminology regarding representations on Hilbert modules.  The ideas in this section are not original: they come from papers of Kasparov \cite{Kasparov:1980sp}, Thomsen \cite{Thomsen:2000aa}, Dadarlat-Eilers \cite{Dadarlat:2002aa}, and Dadarlat \cite{Dadarlat:2005aa}.   Nonetheless, we need some variants of the material appearing in the literature, so record what we need here for the reader's convenience; we provide proofs where we could not find the precise statement we need in the literature.

The definition of absorbing representation below is essentially\footnote{Thomsen's definition is a little more restrictive: he insists that $B$ be stable, and that the $B$-modules used all be copies $B\otimes \K(\ell^2)$.  Thanks to a combination of Kasparov's stabilization theorem \cite[Theorem 2]{Kasparov:1980sp} and Remark \ref{usa unique} below, our extra generality makes no real difference.} due to Thomsen \cite[Definition 2.6]{Thomsen:2000aa}.

\begin{definition}\label{ab def}
A representation $\pi:A\to \LL(F)$ is \emph{absorbing} (for the pair $(A,B)$) if for any Hilbert $B$-module $E$ and ccp map $\sigma:A \to \LL(E)$, there is a sequence $(v_n)$ of isometries in $\LL(E,F)$ such that:
\begin{enumerate}[(i)]
\item $\sigma(a)-v_n^*\pi(a)v_n\in \K(E)$ for all $a\in A$ and $n\in\N$;
\item $\|\sigma(a)-v_n^*\pi(a)v_n\|\to 0$ as $n\to\infty$ for all $a\in A$.
\end{enumerate}
\end{definition}

We want something slightly stronger.  

\begin{definition}\label{sa rep}
A  representation $\pi:A \to \LL(F)$ is \emph{strongly absorbing} (for the pair $(A,B)$) if $(\pi,F)$ is the infinite amplification $(\sigma^\infty,E^\infty)$ of an absorbing representation $(\sigma,E)$.
\end{definition}

\begin{remark}\label{isoms}
If $(\pi,F)$ is an infinite multiplicity (for example, strongly absorbing) representation then there we can write it as an infinite direct sum of copies of itself.  It follows that there is a sequence $(s_n)_{n=1}^\infty$ of isometries in $\LL(F)$ with mutually orthogonal ranges, that commute with the image of the representation, and are such that the sum $\sum_{n=1}^\infty s_ns_n^*$ converges strictly\footnote{For the strict topology coming from the identification $\LL(F)=M(\K(F))$.  As the partial sums are uniformly bounded, we may equivalently use the topology of pointwise convergence as operators on $F$.} to the identity.
\end{remark}

In \cite[Theorem 2.7]{Thomsen:2000aa}, Thomsen shows that an absorbing representation of $A$ on $\ell^2\otimes B$ always exists.  The following is therefore immediate from the fact that $(\ell^2\otimes B)^\infty\cong \ell^2\otimes B$.

\begin{proposition}\label{abs ex}
There is a strongly absorbing representation of $A$ on $\ell^2\otimes B$. \qed
\end{proposition}

The point of using strongly absorbing representations rather than just absorbing\footnote{We do not know that the lemma fails for absorbing representations, but cannot prove it either.} ones is to get the following lemma.

\begin{lemma}\label{sab lem}
Let $\pi:A\to \LL(F)$ be a strongly absorbing representation, and let $\sigma:A\to \LL(E)$ be a ccp map. Then there is a sequence $(v_n)$ of isometries in $\LL(E,F)$ such that:
\begin{enumerate}[(i)]
\item $\sigma(a)-v_n^*\pi(a)v_n\in \K(E)$ for all $a\in A$ and $n\in\N$;
\item $\|\sigma(a)-v_n^*\pi(a)v_n\|\to 0$ as $n\to\infty$ for all $a\in A$; 
\item $v_n^*v_m=0$ for all $n\neq m$.
\end{enumerate}
\end{lemma}

\begin{proof}
Let $(\pi,F)=(\theta^\infty,G^\infty)$ for some absorbing representation $(\theta,G)$.  Let $(w_n)$ be a sequence of isometries in $\LL(E,G)$ with the properties as in the definition of an absorbing representation for $\sigma$.  For each $n$, let $s_n\in \LL(G,F)$ be the inclusion of $G$ in $F$ as the $n^\text{th}$ summand, and set $v_n:=s_nw_n\in \LL(E,F)$.  Direct checks show that $(v_n)$ has the right properties.
\end{proof}

We will need the following result, which is implicit\footnote{\label{sa footnote}It is also explicit in \cite[Theorem 2.6]{Dadarlat:2016qc}, but with $\pi$ only assumed absorbing, not strongly absorbing.  There seems to be a gap in the proof of that result.  As a result, it seems to be necessary to assume all absorbing modules used in \cite{Dadarlat:2016qc} are actually strongly absorbing.  None of the results of \cite{Dadarlat:2016qc} are further affected if one does this.} in \cite{Dadarlat:2001aa}.  

\begin{proposition}\label{cov isom}
Let $\pi:A\to \LL(E)$ be a strongly absorbing representation.  Then for any ccp map $\sigma:A\to \LL(F)$ there is an isometry $v\in C_{ub}([1,\infty),\LL(F,E))$ such that $v^*\pi(a)v-\sigma(a)\in C_0([1,\infty),\K(F))$.

Moreover, if $\sigma:A\to \LL(F)$ is also a strongly absorbing representation, then there is a unitary $u\in C_{ub}([1,\infty),\LL(F,E))$ such that $u^*\pi(a)u-\sigma(a)\in C_0([1,\infty),\K(F))$.
\end{proposition}

For the proof of Proposition \ref{cov isom} will need two lemmas.  The first is a well-known algebraic trick.

\begin{lemma}\label{v on both sides}
Let $\pi:A\to \LL(E)$ and $\sigma:A\to \LL(F)$ be representations, and $v\in \LL(E,F)$ be an isometry.  If $v \in C_{ub}([1,\infty),\LL(F,E))$ is such that $v^*\pi(a)v-\sigma(a)\in C_0([1,\infty),\K(F))$ for all $a\in A$, then $\pi(a)v-v\sigma(a)$ is an element of $C_0([1,\infty),\K(F,E))$ for all $a\in A$.
\end{lemma}

\begin{proof}
This follows from the fact that 
$$
(\pi(a)v-v\sigma(a))^*(\pi(a)v-v\sigma(a))
$$
equals 
$$
v^*\pi(a^*a)v-\sigma(a^*a) -(v^*\pi(a^*)v-\sigma(a^*))\sigma(a)-\sigma(a^*)(v^*\pi(a)v-\sigma(a))
$$
for all $a\in A$.
\end{proof}

The second lemma we need is \cite[Lemma 2.16]{Dadarlat:2002aa}; we recall the statement for the reader's convenience but refer to the reference for a proof.

\begin{lemma}\label{de2.3}
Let $\pi:A\to \LL(E)$ and $\sigma:A\to \LL(F)$ be representations.  Let $\sigma^\infty:A\to \LL(F^\infty)$ be the infinite amplification of $\sigma$, and let $w\in \LL(F^\infty,F\oplus F^\infty)$ be defined by $(\xi_1,\xi_2,\xi_3...)\mapsto \xi_1\oplus (\xi_2,\xi_3,...)$.  Then for any isometry $v\in \LL(F^\infty,E)$, the operator
$$
u:=(1_F\oplus v)wv^*+1_E-vv^*\in \LL(E,F\oplus E)
$$
is unitary and satisfies 
$$
\|\sigma(a)\oplus \pi(a)-u\pi(a)u^*\|\leq 6\|v\sigma^\infty(a)-\pi(a)v\|+4\|v\sigma^\infty(a^*)-\pi(a^*)v\|.
$$
Moreover, if $v\sigma^\infty(a)-\pi(a)v\in \K(F^\infty,E)$ for all $a\in A$, then $\sigma(a)\oplus \pi(a)-u\pi(a)u^*\in \K(F\oplus E)$ for all $a\in A$. \qed
\end{lemma}

\begin{proof}[Proof of Proposition \ref{cov isom}]
Assume first that $\sigma:A\to \LL(F)$ is ccp.  Let $(v_n)$ be a sequence of isometries in $\LL(F,E)$ as in Lemma \ref{sab lem}.  For each $n\geq 1$ and each $t\in [n,n+1]$, define  
$$
v_{t}:=(n+1-t)^{1/2}v_n+(t-n)^{1/2} v_{n+1}
$$
Then the resulting family $v:=(v_t)_{t\in [1,\infty)}$ is an isometry in $C_{ub}([1,\infty),\LL(F,E))$ such that $v^*\pi(a)v-\sigma(a)\in C_0([1,\infty),\K(F))$ for all $a\in A$; we leave the direct checks involved to the reader.

Assume now that $\sigma:A\to \LL(F)$ is also a strongly absorbing representation.  Using the first part of the proof applied to the infinite amplification $\sigma^\infty:A\to \LL(F^\infty)$, we get $v\in C_{ub}([1,\infty),\LL(F^\infty,E))$ such that $v^*\pi(a)v-\sigma^\infty(a)\in C_0([1,\infty),\K(F^\infty))$ for all $a\in A$.  Lemma \ref{v on both sides} implies that $\pi(a)v-v\sigma^\infty(a)$ is an element of $C_0([1,\infty),\K(F^\infty,E))$ for all $a\in A$.  Building a unitary out of each $v_t$ using the formula in Lemma \ref{de2.3} gives now a unitary $u_E\in C_{ub}([1,\infty),\LL(E,F\oplus E))$ such that $\sigma(a)\oplus \pi(a)-u_E\pi(a)u_E^*\in C_0([1,\infty),\K(F\oplus E))$ for all $a\in A$.  The situation is symmetric, so there is also a  unitary $u_F\in C_{ub}([1,\infty),\LL(F,F\oplus E))$ such that $\sigma(a)\oplus \pi(a)-u_F\sigma(a)u_F^*\in C_0([1,\infty),\K(F\oplus E))$ for all $a\in A$.  Defining $u=u_E^*u_F$, we are done.
\end{proof}

We need one more technical result about strongly absorbing representations.  The statement and proof are essentially\footnote{The main difference is that we drop a unitality assumption on $C$.  This is important for our applications later in the paper, as we will want to apply the result with $C=C_0(0,1)$ in order to treat suspensions of $C^*$-algebras.} the same as a result of Dadarlat \cite[Proposition 3.2]{Dadarlat:2005aa}.  We give give a proof for the reader's convenience.

\begin{proposition}\label{sa tp}
Let $\pi:A\to \LL(B\otimes \ell^2)$ be a strongly absorbing representation of $A$ on the standard Hilbert $B$-module.  Let $C$ be a separable nuclear $C^*$-algebra, and let $C\otimes B\otimes \ell^2$ denote the $C\otimes B$-Hilbert module given by the exterior tensor product.  Then the amplification $1_C\otimes \pi:A\to \LL(C\otimes B\otimes \ell^2)$\footnote{Here and throughout, by ``$1_C$'' we mean either the unit of $C$ if $C$ is unital, or the unit of the unitization $C^+$ of $C$, acting as a multiplier on $C$.} is strongly absorbing for the pair $(A,C\otimes B)$.
\end{proposition}

\begin{proof}
As $1_C\otimes \pi$ is isomorphic to the infinite amplification of itself, it suffices to show that $1_C\otimes \pi$ is absorbing.  Let $(1_C\otimes \pi)^+:A^+\to \LL(C\otimes B\otimes \ell^2)$ be the canonical unital extension of $1_C\otimes \pi$ to the unitization $A^+$ of $A$ (even if $A$ is already unital).  Using Kasparov's stabilization theorem \cite[Theorem 2]{Kasparov:1980sp}, the equivalence of (1) and (2) from \cite[Theorem 2.5]{Thomsen:2000aa}, \cite[Theorem 2.1]{Thomsen:2000aa}, and the canonical identifications $C\otimes B\otimes \mathcal{K}(\ell^2)=\K(C\otimes B\otimes \ell^2)$ and $\LL(C\otimes B\otimes \ell^2)=M(\K(C\otimes B\otimes \ell^2))$, it suffices to show that if $\sigma:A^+\to C\otimes B\otimes\mathcal{K}(\ell^2)$ is any ccp map then there is a sequence $(w_n)$ in $\LL(C\otimes B\otimes \K(\ell^2))$ such that 
$$
\lim_{n\to\infty} \|\sigma(a)-w_n^*(1_C\otimes \pi)^+(a)w_n\|= 0 \quad \text{for all}\quad a\in A^+
$$
and such that 
$$
\lim_{n\to\infty} \|w_n^*b\|=0 \quad \text{for all} \quad b\in C\otimes B\otimes \K(\ell^2).
$$

Let $\delta:C^+\to \mathcal{B}(\ell^2)$ be a unital representation of the unitization of $C$ such that $\delta^{-1}(\K(\ell^2))=\{0\}$.  Let $\iota:C^+\to \LL(C)$ be the canonical multiplication representation.  Kasparov's version of Voiculescu's theorem \cite[Theorem 5]{Kasparov:1980sp} combined with nuclearity of $C^+$ imply that there is a sequence $(v_{n,(0)})_{n=1}^\infty$ of isometries in $\LL(C,C^+\otimes \ell^2)$ such that 
$$
\|\iota(c)-v_{n,(0)}^*\big(1_{C^+}\otimes \delta(c)\big)v_{n,(0)}\|\to 0 \quad \text{as} \quad n\to\infty
$$
for all $c\in C^+$.  Perturbing $v_{n,{(0)}}$ slightly, we may assume that actually $v_{n,(0)}$ has image in $C^+\otimes \ell^2\{1,...,k(n)\}$ for some $k(n)$.

Let $(e_m)$ be an approximate unit for $C$.  We may consider multiplication by $e_m$ as defining an operator in $\LL_C(C^+\otimes H,C\otimes H))$ for any Hilbert space $H$, and therefore the product operators $e_mv_{n,(0)}$ make sense in $\LL(C,C\otimes \ell^2\{1,...,k(n)\})$.  For a suitable choice of $m(n)$ we have that if $v_{n,(1)}:=e_{m(n)}v_{n,(0)}$ then
$$
\|\iota(c)-v_{n,(1)}^*\big(1_C\otimes \delta(c)\big)v_{n,(1)}\|\to 0\quad \text{as} \quad n\to\infty
$$
for all $c\in C$.   Let $\delta_n:C\to \mathcal{K}(\ell^2\{1,...,k(n)\})$ be the compression of $\delta$ to the first $n$ basis vectors.  Note that by choice of $k(n)$ we have $v_{n,(1)}^*\big(1_C\otimes \delta(c)\big)v_{n,(1)}=v_{n,(1)}^*\big(1_C\otimes \delta_n(c)\big)v_{n,(1)}$ for all $n$, and thus that 
$$
\|\iota(c)-v_{n,(1)}^*\big(1_C\otimes \delta_n(c)\big)v_{n,(1)}\|\to 0\quad \text{as} \quad n\to\infty
$$
for all $c\in C$. 

Define
$$
\Delta_n:= \delta_n \otimes 1_{B\otimes \K(\ell^2)} :C\otimes B\otimes \K(\ell^2)\to \mathcal{K}(\ell^2\{1,...,k(n)\})\otimes B\otimes\K(\ell^2).
$$
Define $v_{n,(2)}:=v_{n,(1)}\otimes 1_{B\otimes \K(\ell^2)}$, so 
$$
v_{n,(2)}\in \LL_{C\otimes B\otimes \K(\ell^2)}\big(C\otimes B \otimes \K(\ell^2)\, ,\,C\otimes \ell^2\{1,...,k(n)\}\otimes B\otimes \K(\ell^2) \big).
$$  
Note then that 
$$
\|c-v_{n,(2)}^*\big(1_C\otimes \Delta_n(c)\big)v_{n,(2)}\|\to 0\quad \text{as} \quad n\to\infty
$$
for all $c\in C\otimes B\otimes \K(\ell^2)$ and so in particular
$$
\|\sigma(a)-v_{n,(2)}^*\big(1_C\otimes \Delta_n(\sigma(a))\big)v_{n,(2)}\|\to 0\quad \text{as} \quad n\to\infty
$$
for all $a\in A^+$.

To complete the proof, use an identification $\mathcal{K}(\ell^2\{1,...,k(n)\})\otimes\K(\ell^2)\cong \K(\ell^2)$ to give an isomorphism $\phi: \mathcal{K}(\ell^2\{1,...,k(n)\})\otimes B\otimes \K(\ell^2)\to  B\otimes\K(\ell^2)$.  Note that as $\pi$ is absorbing there is a sequence $(v_{n,(3)})_{n=1}^\infty$ in $\LL(B\otimes \K(\ell^2))$ such that 
$$
\|\phi(\Delta_n(\sigma(a)))-v_{n,(3)}^*\pi(a)v_{n,(3)}\|\to 0\quad \text{as} \quad n\to\infty
$$
for all $a\in A^+$ (compare the equivalence of (1) and (2) from \cite[Theorem 2.5]{Thomsen:2000aa} again).  As $\pi$ is strongly absorbing, we may moreover assume that the $v_{n,(3)}$ satisfy $v_{n,(3)}^*b\to 0$ for all $b\in B\otimes \K(\ell^2)$ by ensuring that for all $m$, $v_{n,(3)}v_{n,(3)}^*$ is orthogonal to any element of $B\otimes \K(\ell^2\{1,...,m\})$ for all large $n$.  It is then not too difficult to check that we can choose $l(n)$ such that if we set 
$$
w_n:=(1_C\otimes v_{l(n),(3)})v_{n,(2)}\in \LL(C\otimes B\otimes \K(\ell^2)),
$$ 
then $(w_n)$ has the right properties.
\end{proof}

\section{Localization algebras}\label{loc sec}

Throughout this section, $A$ and $B$ refer to separable $C^*$-algebras.  All Hilbert modules are countably generated, and all are over $B$ unless explicitly stated otherwise.  All representations of $A$ are on Hilbert $B$-modules unless explicitly stated otherwise.

In this section, we define localization algebras following \cite{Dadarlat:2016qc}, and show that uniform continuity can be replaced with continuity in the definition without changing the $K$-theory.  This result was first observed by Jianchao Wu (with a different proof), and we thank him for permission to include it here.  

The following definition comes from \cite[Section 3]{Dadarlat:2016qc}.  We use slightly different notation to that reference to  differentiate between the continuous and uniformly continuous versions. 

\begin{definition}\label{loc alg}
Let $\pi:A\to \LL(E)$ be a representation.  Define $\CL(\pi)$ to be the $C^*$-algebra of all bounded, uniformly continuous functions $b:[1,\infty)\to \LL(E)$ such that $[b_t,a]\to 0$ for all $a\in A$, and such that $ab_t$ is in $\K(E)$ for all $a\in A$ and all $t\in [1,\infty)$.  We call $\CL(\pi)$ the \emph{localization algebra} of $\pi$. 
\end{definition}

The following is\footnote{As explained in footnote \ref{sa footnote}, the cited result should be stated with the assumption that the representation is strongly absorbing, not just absorbing.} \cite[Theorem 4.4]{Dadarlat:2016qc}.

\begin{theorem}\label{kk iso dww}
Let $\pi:A\to \LL(E)$ be a strongly absorbing representation.  Then there is a canonical isomorphism $KK_*(A,B)\to K_*(\CL(\pi))$.   \qed
\end{theorem}

\begin{definition}\label{loc alg}
Let $\pi:A\to \LL(E)$ be a representation of $A$ on a Hilbert $B$-module.  Define $\CLc(\pi)$ to be the $C^*$-algebra of all bounded, continuous functions $b:[1,\infty)\to \LL(E)$ such that $[b_t,a]\to 0$ for all $a\in A$, and such that $ab_t$ is in $\K(E)$ for all $a\in A$ and all $t\in [1,\infty)$.  
\end{definition}

Clearly there is a canonical inclusion $\CL(\pi)\to \CLc(\pi)$.  Our main goal in this section is to establish the following result.

\begin{theorem}\label{c to uc}
Let $\pi:A\to \LL(F)$ be an infinite multiplicity (in particular, $\pi$ could be strongly absorbing) representation of $A$ on a Hilbert $B$-module.  Then the canonical inclusion $\CL(\pi)\to \CLc(\pi)$ induces an isomorphism on $K$-theory.
\end{theorem}

We will need several preliminary lemmas.   The first of these follows from standard techniques: compare for example \cite[Proposition 4.1.7]{Higson:2000bs}.

\begin{lemma}\label{equi con}
There is a function $\omega:[0,\infty)\to [0,\infty)$ such that $\omega(0)=0$, such that ${\displaystyle \lim_{t\to0} \omega(t)=0}$, and with the following property.  Assume that $D$ is a $C^*$-algebra, and that $p^0,p^1\in C([0,1],D)$ are projections with the property that $\|p^0_t-p^1_t\|<1/2$ for all $t\in [0,1]$.  

Then there is a homotopy $(p^s)_{s\in [0,1]}$ connecting $p^0$ and $p^1$, and with the property that $\|p^s-p^{s'}\|\leq \omega(|s-s'|)$ for all $s,s'\in [0,1]$
\end{lemma}

\begin{proof}
We will work in $C([0,1],D^+)$, where $D^+$ is the unitization of $D$.  Fix $t\in [0,1]$, and define 
$$
x_t:=p^1_tp^0_t+(1-p^1_t)(1-p^0_t).
$$
Then one checks as in the proof of \cite[Proposition 4.1.7]{Higson:2000bs} that $\|1-x_t\|<1/2$.  For $s\in [0,1]$, define  $x^s_t:=s1+(1-s)x_t$, which also satisfies $\|1-x^s_t\|<1/2$.  Hence each $x^s_t$ is invertible, and the norm of its inverse is at most $2$ by the usual Neumann series representation $(x^s_t)^{-1}=\sum_{n=0}^\infty (1-x^s_t)^n$.  Define moreover $u_t^s:=x^s_t((x^s_t)^*x^s_t)^{-1/2}$, which is unitary, and depends continuously on $(s,t)\in [0,1]\times [0,1]$ by the functional calculus (compare \cite[Lemma 1.2.5]{Rordam:2000mz}).  One computes as in the proof of \cite[Proposition 4.1.7]{Higson:2000bs} that 
$$
u_t^1p^0_t(u^1_t)^*=p^1_t
$$
for all $t\in [0,1]$.  Define finally $p^s_t:=u_t^sp^0_t(u^s_t)^*$.  Then $(p^s)_{s\in [0,1]}$ is a homotopy connecting $p^0$ and $p^1$.  The existence of the function $\omega$ follows from continuity of $u_t^s$ as a function of $(s,t)\in [0,1]\times [0,1]$ and a compactness argument.
\end{proof}

For the statement of the next lemma, recall that if $C$ and $D$ are $C^*$-algebras equipped with surjections $\pi_C:C\to Q$ and $\pi_D:D\to Q$ to a third $C^*$-algebra $Q$, then the \emph{pullback} is the $C^*$-algebra 
$$
P:=\{(c,d)\in C\oplus D\mid \pi_C(c)=\pi_D(d)\}.
$$
Such a pullback gives rise to a canonical diagram
\begin{equation}\label{po}
\xymatrix{ P \ar[r] \ar[d] & C \ar[d]^-{\pi_C} \\ D \ar[r]_-{\pi_D} & Q }
\end{equation}
where the arrows out of $P$ are the natural (surjective) coordinate projections.  We call any square isomorphic to one of this form a \emph{pullback square}.

See for example \cite[Proposition 2.7.15]{Willett:2010ay} for a proof of the next result.

\begin{lemma}\label{mv lem}
Given a pullback diagram as in line \eqref{po} above, there is a six-term exact sequence 
$$
\xymatrix{ K_0(P) \ar[r] & K_0(C)\oplus K_0(D) \ar[r] & K_0(Q) \ar[d] \\ K_1(Q) \ar[u] & K_1(C)\oplus K_1(D) \ar[l] & K_1(P) \ar[l] }
$$
of $K$-theory groups.  The diagram is natural for maps between pullback squares. \qed
\end{lemma}

Before the proof of Theorem \ref{c to uc}, we record  two more well-known $K$-theoretic lemmas.  See for example \cite[Proposition 2.7.5 and Lemma 2.7.6]{Willett:2010ay} for proofs\footnote{The statement of \cite[Proposition 2.7.5]{Willett:2010ay} has a typo: $vv^*$ should be $v^*v$ where it appears there.}.

\begin{lemma}\label{orth sum}
If $\alpha,\beta:C\to D$ are $*$-homomorphisms with orthogonal images, then $\alpha+\beta:C\to D$ is also a $*$-homomorphism, and $(\alpha+\beta)_*=\alpha_*+\beta_*$ as maps on $K$-theory. \qed
\end{lemma}

\begin{lemma}\label{isom lem}
Let $\alpha,\beta:C\to D$ be $*$-homomorphisms, and assume that there is a partial isometry $v$ in the multiplier algebra of $D$ such that $\alpha(c)v^*v=\alpha(c)$ for all $c\in C$, and so that $v\alpha(c)v^*=\beta(c)$ for all $c\in C$.  Then $\alpha$ and $\beta$ induce the same maps on $K$-theory. \qed
\end{lemma}

\begin{proof}[Proof of Theorem \ref{c to uc}]
Let $E:=\bigsqcup_{n\geq 1} [2n,2n+1]$ and $O:=\bigsqcup_{n\geq 1} [2n-1,2n]$, equipped with the restriction of the metric from $[1,\infty)$ .  Let $\CL(\pi;E)$ denote the collection of all bounded, uniformly continuous functions $b:E\to \LL(F)$ such that $ab_t\in \K(F)$ for all $a\in A$, and such that $[a,b_t]\to 0$ as $t\to\infty$.  Define $\CL(\pi;O)$ and $\CL(\pi;E\cap O)$ similarly, and define $\CLc(\pi;E)$, $\CLc(\pi;O)$, and $\CLc(\pi;O\cap E)$ analogously, but with `uniformly continuous' replaced by `continuous'.  Then we have a commutative diagram of pullback squares
$$
\xymatrix{ \CL(\pi) \ar[rr]\ar[dd] \ar[dr] & &  \CL(\pi;E)\ar[dr] \ar'[d][dd] & \\
& \CLc(\pi) \ar[rr] \ar[dd] & & \CLc(\pi;E) \ar[dd] \\
\CL(\pi;O)  \ar'[r][rr] \ar[dr] & & \CL(\pi;E\cap O) \ar[dr]& \\
& \CLc(\pi;O) \ar[rr] & & \CLc(\pi;E\cap O) }
$$
where the diagonal arrows are the canonical inclusions, and all the other arrows are the obvious restriction maps.  Using Lemma \ref{mv lem} and the five lemma, it thus suffices to show that the maps $\CL(\pi;G)\to \CLc(\pi;G)$ induce isomorphisms on $K$-theory for $G\in \{E,O,E\cap O\}$.  For $E\cap O$, which just equals $\N\cap [1,\infty)$, this is clear: the map is the identity on the level of $C^*$-algebras as there is no difference between continuity and uniform continuity in this case.  The cases of $E$ and $O$ are essentially the same, so we just focus on $E$.  

Let now $E_\N:=E\cap 2\N=\{2,4,6,...\}$ be the set of positive even numbers.  Then we have a surjective $*$-homomorphism $\CL(\pi;E)\to \CL(\pi;E_\N)$ defined by restriction, and similarly for $\CLc$; write $\CL^0(\pi;E)$ and $\CLc^0(\pi;E)$ for the respective kernels.  Then we have a commutative diagram
$$
\xymatrix{ 0\ar[r] & \CL^0(\pi;E) \ar[r]  \ar[d] & \CL(\pi;E) \ar[r]\ar[d] & \CL(\pi;E_\N) \ar[d] \ar[r] & 0 \\
0\ar[r] & \CLc^0(\pi;E) \ar[r]  & \CLc(\pi;E) \ar[r] & \CLc(\pi;E_\N) \ar[r] & 0 }
$$
of short exact sequences where the vertical maps are the canonical inclusions.  The right hand vertical map is the identity as there is no difference between continuity and uniform continuity for maps out of $E_\N$.  Hence by the five lemma and the usual long exact sequence in $K$-theory, it suffices to show that the left hand vertical map induces an isomorphism on $K$-theory.  For $r\in [0,1]$ let us define a $*$-homomorphism $h_r:\CL^0(\pi;E)\to \CL^0(\pi;E)$ by the following prescription.  For $b\in \CL^0(\pi;E)$ and $t\in [2n,2n+1]$, we set 
$$
(h_rb)_t:=b_{2n+r(t-2n)}
$$
(in words, $h_r$ contracts $[2n,2n+1]$ to $\{2n\}$ as $r$ varies from $1$ to $0$).  Using uniform continuity, $(h_r)_{r\in [0,1]}$ is a null-homotopy of $\CL^0(\pi;E)$, and therefore $K_*(\CL^0(\pi;E))=0$.  It thus suffices to show that $K_*(\CLc^0(\pi;E))=0$, which we spend the rest of the proof doing.

We will focus on the case of $K_0$ (which is in any case all we use in this paper); the case of $K_1$ is similar.  Take then  an arbitrary element $x\in K_0(\CLc^0(\pi;E))$, which we may represent by a formal difference $x=[p]-[1_k]$ where $p$ is a projection in the $m\times m$ matrices $M_m(\CLc^0(\pi;E)^+)$ over the unitization $\CLc^0(\pi;E)^+$ of $\CLc(\pi;E)$ for some $m$, and $1_k\in M_m(\C)\subseteq M_m(\CLc^0(\pi;E)^+)$ is the scalar matrix with $1$s in the first $k$ diagonal entries and $0$s elsewhere for some $k\leq m$.  Without loss of generality, we may think of $p$ as a continuous projection-valued function 
$$
p:E\to M_m(\LL(F))
$$
such that $a(p-1_k)\in M_m(\K(F))$ for all $a\in A$ (here we use the amplification of the representation of $A$ to a representation on $M_m(\LL(F))$ to make sense of this), such that $[a,p_t]\to 0$ for all $a\in A$, and such that $p_{2n}=1_k$ for all $n\in \N$.  

Now, for each $n$, the restriction $p|_{[2n,2n+1]}$ is uniformly continuous, whence there is some $r_n\in (0,1)$ such that if $t,s\in [2n,2n+1]$ satisfy $|t-s|\leq 1-r_n$, then $\|p_t-p_s\|<1/2$.  For each $l\in \N\cup\{0\}$, define $p^{(l)}:E\to M_m(\LL(F))$ to be the function whose restriction to $[2n,2n+1]$ is defined by 
$$
p^{(l)}_t:=p_{2n+(t-2n)(r_n)^l}.
$$  
Fix a sequence $(s_l)_{l=0}^\infty$ of isometries as in Remark \ref{isoms} and consider the formal difference 
$$
x_\infty:=\Bigg[\sum_{l=0}^\infty s_lp^{(l)}s_l^*\Bigg] - \Bigg[\sum_{l=0}^\infty s_l1_ks_l^*\Bigg]
$$
where the sum converges strictly in $M_m(\LL(F))\cong \LL(F^{\oplus m})$, pointwise in $t$ (we are abusing notation slightly: we should really have replaced $s_l$ by $1_{M_m(\C)}\otimes s_l$).  As $r_n<1$ and as $p_{2n}=1_k$ for each $n$, we see that for any $t$, $p^{(l)}_t-1_k\to 0$ as $l\to\infty$; it follows from this and the fact that each $s_l$ commutes with the representation of $A$ that $x_\infty$ gives a well-defined element of $K_0(\CLc^0(\pi;E))$.  

Now, let us consider the element $x_\infty+x$ of $K_0(\CLc^0(\pi;E))$. We claim this equals $x_\infty$.  As $K_0$ is a group, this forces $x=0$, and thus $K_0(\CLc^0(\pi;E))=0$ as required.  Indeed, first note that conjugating by the isometry 
$$
s:=\sum_{l=0}^\infty s_{l+1}s_{l}^*
$$
in the multiplier algebra of $\CLc^0(\pi;E)$ and applying Lemma \ref{isom lem} shows that 
$$
x_\infty=\Bigg[\sum_{l=1}^\infty s_lp^{(l-1)}s_l^*\Bigg] - \Bigg[\sum_{l=1}^\infty s_l1_ks_l^*\Bigg].
$$
The choice of the sequence $(r_n)$ and Lemma \ref{equi con}\footnote{For the case of $K_1$ one needs an analogue of Lemma \ref{equi con} that works for unitaries: this follows from an analogous functional calculus argument using that if $u,v$ are unitaries with $\|u-v\|<1/2$ then there is self-adjoint $h$ with $\|h\|\leq \pi$, and $uv^*=e^{ih}$.} guarantees the existence of a homotopy between $p^{(l-1)}$ and $p^{(l)}$ for each $l\geq 1$, and moreover that these homotopies can be assumed equicontinuous as $l$ varies (this includes the fact that the restrictions to the various intervals $[2n,2n+1]$ are equicontinuous  as $n$ varies).  It follows that 
\begin{equation}\label{xinf reform}
x_\infty=\Bigg[\sum_{l=1}^\infty s_lp^{(l)}s_l^*\Bigg] - \Bigg[\sum_{l=1}^\infty s_l1_ks_l^*\Bigg]
\end{equation}
On the other hand, applying Lemma \ref{isom lem} again, we have that 
$$
x=[s_0ps_0^*]-[s_01_ks_0^*].
$$
Hence combining this with line \eqref{xinf reform} above and also Lemma \ref{orth sum}
\begin{align*}
x+x_\infty & =[s_0ps_0^*]-[s_01_ks_0^*]+\Bigg[\sum_{l=1}^\infty s_lp^{(l)}s_l^*\Bigg] - \Bigg[\sum_{l=1}^\infty s_l1_ks_l^*\Bigg] \\ & =\Bigg[\sum_{l=0}^\infty s_lp^{(l)}s_l^*\Bigg] - \Bigg[\sum_{l=0}^\infty s_l1_ks_l^*\Bigg] \\ & =x_\infty
\end{align*}
and we are done.
\end{proof}

We finish this section with some technical results that we will need later.   The first goal is to show that $K_*(\CLc(\pi))$ only really depends on information `at $t=\infty$' in some sense.  This is made precise in Corollary \ref{quot iso} below, but we need some more notation first.

\begin{definition}\label{id quot}
Let $\pi:A\to \LL(E)$ be a representation of $A$.  Define $\ILc(\pi)$ to be the ideal in $\CLc(\pi)$ consisting of all functions $b$ such that $ab\in C_0([1,\infty),\K(E))$ for all $a\in A$.  Define $\QL(\pi):=\CLc(\pi)/\ILc(\pi)$ to be the corresponding quotient. 
\end{definition}

\begin{lemma}\label{eilenberg}
Let $\pi:A\to \LL(E)$ be an infinite multiplicity representation of $A$.  Then $\ILc(\pi)$ has trivial $K$-theory.
\end{lemma}

\begin{proof}
Set $\IL(\pi):=\CL(\pi)\cap \ILc(\pi)$.  The same argument in the proof of Theorem \ref{c to uc} shows that the inclusion $\IL(\pi)\to \ILc(\pi)$ induces an isomorphism on $K$-theory.  It thus suffices to prove that $K_*(\IL(\pi))=0$, which we now do.

Let $(s_n)_{n=0}^\infty$ be a sequence of isometries in $\LL(E)$ that commute with $A$, and that have orthogonal ranges as in Remark \ref{isoms}.  We regard each $s_n$ as an isometry in the multiplier algebra of $\IL(\pi)$ by having it act pointwise in $t$.  Define 
$$
\iota:\IL(\pi)\to \IL(\pi),\quad b\mapsto s_0bs_0^*,
$$
which is a $*$-homomorphism that induces the identity map on $K$-theory by Lemma \ref{isom lem}.  On the other hand, for each $s\geq 0$, define a $*$-endomorphism $\alpha_s$ of $\IL(\pi)$ by the formula $\alpha_s(b)_t:=b_{t+s}$.  Note that for any norm-bounded sequence $(b_n)$ in $\LL(E)$, the sum 
$$
\sum_{n=1}^\infty s_nb_ns_n^*
$$
converges in the strict topology of $\LL(E)=M(\K(E))$.  Therefore we get a $*$-homomorphism
$$
\alpha:\IL(\pi)\to \IL(\pi), \quad \alpha(b):=\sum_{n=1}^\infty s_n\alpha_n(b) s_n^*
$$
(the image is in $\IL(\pi)$ as $ab_t\to 0$ as $t\to\infty$ for all $a\in A$, which implies that for each fixed $t$ and any $a\in A$, $a\alpha_n(b_t)\to 0$ as $n\to\infty$).  Now, the maps $\alpha$ and $\iota$ have orthogonal ranges, whence by Lemma \ref{orth sum}, $\alpha+\iota$ is also a $*$-homomorphism, and we have that as maps on $K$-theory, $\alpha_*+\iota_*=(\alpha_*+\iota_*)$.  Define $s:=\sum_{n=0}^\infty s_{n+1}s_n^*$ (convergence in the strict topology), which we think of as a multiplier of $\IL(\pi)$.  Applying Lemma \ref{isom lem} again, we see that $\iota+\alpha$ induces the same map on $K$-theory as the map $b\mapsto s(\iota(b)+\alpha(b))s^*$, which is the map
$$
\IL(\pi)\to \IL(\pi),\quad b\mapsto \sum_{n=1}^\infty s_n \alpha_{n-1}(b)s_n^*.
$$
On the other hand, using that elements of $\IL(\pi)$ are uniformly continuous, we get a homotopy 
$$
b\mapsto \sum_{n=1}^\infty s_n \alpha_{n-1+r}(b)s_n^*,\quad r\in [0,1]
$$
between this map and $\alpha$.  In other words, we now have that $\alpha_*+\iota_*=\alpha_*$ as maps on $K$-theory.  This forces $\iota_*$ to be the zero map on $K_*(\IL(\pi))$.  However, we also observed already that $\iota_*$ is the identity map, so $K_*(\IL(\pi))$ is indeed zero. 
\end{proof}

The following corollary is immediate from the six-term exact sequence in $K$-theory.

\begin{corollary}\label{quot iso}
Let $\pi:A\to \LL(E)$ be an infinite multiplicity representation of $A$ on a Hilbert $B$-module.  Then the  canonical quotient map $\CLc(\pi)\to \QL(\pi)$ induces an isomorphism on $K$-theory. \qed
\end{corollary}

We will need one more definition and lemma about the structure of $\CLc(\pi)$.

\begin{definition}\label{clpi}
Let $\pi:A\to \LL(E)$ be a representation of $A$ on a Hilbert $B$-module.   Define 
$$
\CLc(\pi;\K):=C_{b}([1,\infty),\K(E))\cap \CLc(\pi),
$$
which is an ideal in $\CLc(\pi)$.
\end{definition}

\begin{lemma}\label{cl sum}
Let $\pi:A\to \LL(E)$ be a representation of $A$ on a Hilbert $B$-module.  With notation as in Definitions \ref{id quot} and \ref{clpi}, we have 
$$
\CLc(\pi)=\CLc(\pi;\K)+\ILc(\pi).
$$
In particular, the restriction of the quotient map $\CLc(\pi)\to \QL(\pi)$ to $\CLc(\pi;\K)$ is surjective.
\end{lemma}

\begin{proof}
Let $(h_n)$ be a sequential approximate unit for $A$, and define $h\in C_{ub}([1,\infty),A)$ by setting $h_t:=(n+1-t)h_n+(t-n)h_{n+1}$ for $t\in [n,n+1]$.  Then a direct check using that $[a,h]\in C_0([1,\infty),A)$ for any $a\in A$ shows that $h$ defines a multiplier of $\CLc(\pi)$.  Moreover, for any $b\in \CLc(\pi)$, $b=(1-h)b+hb$, and one checks directly that $(1-h)b$ is in $\ILc(\pi)$ and that $hb$ is in $\CLc(\pi;\K)$.  This gives the result on the sum, and the result on the quotient follows immediately. 
\end{proof}

Our final goal in this section is to check that the isomorphisms from Theorem \ref{kk iso dww} and Theorem \ref{c to uc} are compatible with a special case of functoriality for $KK$-theory.  

For this, let $C$ be a separable $C^*$-algebra, and let $\phi:B\to C$ be a $*$-homomorphism.  Let $E$ be a Hilbert $B$-module, and let $E\otimes_\phi C$ be the internal tensor product defined using $\phi$, which is a Hilbert $C$-module.  As discussed on \cite[page 42]{Lance:1995ys} there is a canonical $*$-homomorphism 
$$
\Phi:\LL(E)\to \LL(E\otimes_\phi C), \quad a\mapsto a\otimes 1_C.
$$
Abusing notation slightly, we also write $\Phi$ for the $*$-homomorphism 
$$
C_b([1,\infty),\LL(E))\to C_b([1,\infty),\LL(E\otimes_\phi C))
$$
defined by applying $\Phi$ pointwise.  Let $\pi_B:A\to \LL(E)$ and $\pi_C:A\to \LL(F)$ be representations of $A$ on Hilbert $B$- and $C$-modules respectively.  

\begin{definition}\label{cov isom def}
With notation as above, a \emph{covering isometry} for $\phi$ (with respect to $\pi_B$ and $\pi_C$) is any isometry $v\in C_{b}([1,\infty),\LL(E\otimes_\phi C,F))$ such that 
$$
v^*\pi_C(a)v-(\Phi\circ \pi_B)(a)\in C_0([1,\infty),\K(E\otimes_\phi C))
$$
for all $a\in A$.
\end{definition}

\begin{lemma}\label{cov isom $*$-hom}
With notation as above, if $v$ is a covering isometry for $\phi$, then the formula 
$$
\phi^v:\CLc(\pi_B)\to \CLc(\pi_C),\quad \phi^v(b):=v\Phi(b)v^*
$$
gives a well-defined $*$-homomorphism.  Moreover, the induced map 
$$
\phi^v_*:K_*(\CLc(\pi_B))\to K_*(\CLc(\pi_C))
$$  
on $K$-theory does not depend on the choice of $v$.  Finally, if $\pi_C$ is strongly absorbing, then a covering isometry for $\phi$ always exists, and can be taken to belong to $C_{ub}([1,\infty),\LL(E\otimes_\phi C,F))$ (i.e.\ to be uniformly continuous, not just continuous).
\end{lemma}

\begin{proof}
Let $v$ be a covering isometry for $\phi$.  For notational simplicity, write $\sigma:=\Phi\circ \pi_B$.  Using Lemma \ref{v on both sides} we have that 
$$
\pi_C(a)v-v\sigma(a)\in C_0([1,\infty),\K(E\otimes _\phi C,F))
$$ 
for all $a\in A$.  Note that for $a\in A$ and $b\in \CLc(\pi_B)$
$$
\pi_C(a)\phi^v(b)=(\pi_C(a)v-v\sigma(a))\Phi(b)v^*+v\Phi(\pi_B(a)b)v^*;
$$
using that $\Phi$ takes $\K(E)$ to $\K(E\otimes_\phi C)$ (see \cite[Proposition 4.7]{Lance:1995ys}), this shows that $\pi_C(a)\phi^v(b)\in C_{b}([1,\infty),\K(F))$ .  Similarly,  
\begin{align*}
[\pi_C(a),\phi^v(b)] =\big(\pi_C(a)v-v\sigma(a)\big)\Phi(b)v^*   & +v\Phi([\pi(a),b])v^* \\ & +v\Phi(b)\big(v^*\pi_C(a)-\sigma(a)v^*\big),
\end{align*}
whence $[\pi_C(a),\phi^v(b)]\in C_0([1,\infty),\K(F))$.  It follows that $\phi^v$ is indeed a well-defined $*$-homomorphism $\CLc(\pi_B)\to \CLc(\pi_C)$. 

Let now $v,w$ be possibly different covering isometries for $\phi$.   Using similar computations to the above, one checks that $wv^*$ is an element of the multiplier algebra of $\CLc(\pi_C)$ that conjugates the $*$-homomorphisms $\phi^v$ and $\phi^w$ to each other.  The fact that $\phi^v_*=\phi^w_*$ as maps $K_*(\CLc(\pi_B))\to K_*(\CLc(\pi_C))$ follows from this and Lemma \ref{isom lem}.  

Finally, if $\pi_C$ is strongly absorbing, then covering isometries exist, and can be assumed uniformly continuous, by Proposition \ref{cov isom}.
\end{proof}

\begin{definition}\label{cl func}
Let $\pi_B:A\to \LL(E)$ and $\pi_C:A\to \LL(F)$ be representations of $A$ on a Hilbert $B$-module and Hilbert $C$-module respectively, with $\pi_C$ strongly absorbing.  Let $\phi:B\to C$ be any $*$-homomorphism.  Then Lemma \ref{cov isom $*$-hom} gives a well-defined homomorphism $K_*(\CLc(\pi_B))\to K_*(\CLc(\pi_C))$, which we denote $\phi_*$.   
\end{definition}

On the other hand, for a $*$-homomorphism $\phi:B\to C$, let us write $\phi_*:KK(A,B)\to KK(A,C)$ for the usual functorially induced map on $KK$-theory.  The following lemma gives compatibility between these two maps.

\begin{lemma}\label{func prop}
With notation as above, assume that both $\pi_B:A\to \LL(E)$ and $\pi_C:A\to \LL(F)$ are strongly absorbing, and let $KK(A,B)\to K_0(\CLc(\pi_B))$ and $KK(A,C)\to K_0(\CLc(\pi_C))$ be the isomorphisms from Theorem \ref{kk iso dww} and Theorem \ref{c to uc}.  Then the diagram 
$$
\xymatrix{ KK(A,B) \ar[r] \ar[d]^-{\phi_*} & K_0(\CLc(\pi_B)) \ar[d]^-{\phi_*} \\
KK(A,C) \ar[r]  & K_0(\CLc(\pi_C))  }
$$
commutes.
\end{lemma}

\begin{proof}
The proof is unfortunately long as there is a lot to check, but the checks are fairly routine.  We recall first the precise form of the isomorphism $KK_*(A,B) \to  K_*(\CL(\pi_B))$ of Theorems \ref{kk iso dww} and \ref{c to uc}.  It is a composition of the following maps (see also \cite[Definition 3.1]{Dadarlat:2016qc} for the various algebras involved).
\begin{enumerate}[(i)]
\item The Paschke duality isomorphism $P:KK(A,B) \to K_1(\mathcal{D}(\pi_B)/\mathcal{C}(\pi_B))$ of \cite[Theorem 3.2]{Thomsen:2000aa}, where $\mathcal{D}(\pi_B):=\{b\in \LL(E)\mid [b,a]\in \K(E) \text{ for all } a\in A\}$, and $\mathcal{C}(\pi_B):=\{b\in \mathcal{D}(\pi_B)\mid ab\in \K(E) \text{ for all } a\in A\}$.
\item The map on $K$-theory 
$$
\iota_*:K_1(\mathcal{D}(\pi_B)/\mathcal{C}(\pi_B))\to K_1(\mathcal{D}_T(\pi_B)/\mathcal{C}_T(\pi_B))
$$
induced by the constant inclusion $\iota:\mathcal{D}(\pi_B)\to \mathcal{D}_T(\pi_B)$, where $\mathcal{D}_T(\pi_B):=C_{ub}([1,\infty),\mathcal{D}(\pi_B))$ and $\mathcal{C}_T(\pi_B):=C_{ub}([1,\infty),\mathcal{C}(\pi_B))$.
\item The map on $K$-theory 
$$
\eta_*^{-1}:K_1(\mathcal{D}_T(\pi_B)/\mathcal{C}_T(\pi_B))\to K_1(\mathcal{D}_L(\pi_B)/\CL(\pi_B))
$$
which is induced by the inverse (it turns out to be an isomorphism of $C^*$-algbras) of the map $\eta:\mathcal{D}_T(\pi_B)/\mathcal{C}_T(\pi_B)\to \mathcal{D}_{L}(\pi_B)/\CL(\pi_B)$ induced by the inclusion $\mathcal{D}_{L}(\pi_B)\to \mathcal{D}_{T}(\pi_B)$, where $\mathcal{D}_{L}(\pi_B):=\{b\in \mathcal{D}_T(\pi_B)\mid [a,b_t]\to 0 \text{ as }t\to\infty \text{ for all } a\in A\}$.
\item The usual $K$-theory boundary map 
$$
\partial: K_1(\mathcal{D}_L(\pi_B)/\CL(\pi_B))\to K_0(\CL(\pi_B)).
$$
\item The isomorphism $\kappa_*:K_0(\CL(\pi_B))\to K_0(\CLc(\pi_B))$ of Theorem \ref{c to uc} induced by the canonical inclusion. 
\end{enumerate}
Now, if $v$ is a uniformly continuous covering isometry for $\phi$, then one sees from analogous arguments to those given in the proof of Lemma \ref{cov isom $*$-hom} that the formula  
$$
\phi^v(b)_t:=v_t\Phi(b_t)v_t^*
$$
from Lemma \ref{cov isom $*$-hom} also defines $*$-homomorphisms
$$
\phi^v:\left\{\begin{array}{l}\mathcal{D}_L(\pi_B)/\CL(\pi_B) \to \mathcal{D}_L(\pi_C)/\CL(\pi_C)  \\
 \mathcal{D}_T(\pi_B)/\mathcal{C}_T(\pi_B) \to \mathcal{D}_T(\pi_C)/ \mathcal{C}_T(\pi_C)\\\end{array}\right\}.
$$
Moreover, the formula 
$$
\phi^{v_1}(b):=v_1\Phi(b)v_1^*
$$
defines a $*$-homomorphism $\mathcal{D}(\pi_B)/\mathcal{C}(\pi_B)\to \mathcal{D}(\pi_C)/\mathcal{C}(\pi_B)$.  Putting all this together, we get a diagram 
 \begin{equation}\label{com rectangle}
\xymatrix{ K_1(\mathcal{D}(\pi_B)/\mathcal{C}(\pi_B)) \ar[d]^-{\iota_*} \ar[r]^-{\phi^{v_1}_*} & K_1(\mathcal{D}(\pi_C)/\mathcal{C}(\pi_C)) \ar[d]^-{\iota_*} \\
K_1(\mathcal{D}_T(\pi_B)/\mathcal{C}_T(\pi_B)) \ar[r]^-{\phi^{v}_*} \ar[d]^-{\eta_*^{-1}} & K_1(\mathcal{D}_T(\pi_C)/\mathcal{C}_T(\pi_C)) \ar[d]^-{\eta_*^{-1}} \\
K_1(\mathcal{D}_L(\pi_B)/\CL(\pi_B)) \ar[r]^-{\phi^{v}_*} \ar[d]^-\partial & K_1(\mathcal{D}_L(\pi_C)/\CL(\pi_C)) \ar[d]^-\partial \\
K_0(\CLc(\pi_B)) \ar[r]^-{\phi^{v}_*} \ar[d]^-{\kappa_*} & K_0(\CLc(\pi_C)) \ar[d]^-{\kappa_*} \\
K_0(\CLc(\pi_B)) \ar[r]^-{\phi^v_*} &  K_0(\CLc(\pi_C)) }
\end{equation}
\noindent{}We claim that this commutes.  Indeed, the first square commutes as $\iota_*$ is an isomorphism on $K$-theory (\cite[Proposition 4.3 (b)]{Dadarlat:2016qc}), whence its inverse on the level of $K$-theory is the map induced  by the evaluation-at-one homomorphism $e:\mathcal{D}_T(\pi_B)\to \mathcal{D}(\pi_B)$, and the diagram
$$
\xymatrix{ K_1(\mathcal{D}(\pi_B)/\mathcal{C}(\pi_B)) \ar[r]^-{\phi^{v_1}_*} & K_1(\mathcal{D}(\pi_C)/\mathcal{C}(\pi_C))  \\
K_1(\mathcal{D}_T(\pi_B)/\mathcal{C}_T(\pi_B)) \ar[r]^-{\phi^{v}_*} \ar[u]^{e_*} & K_1(\mathcal{D}_T(\pi_C)/\mathcal{C}_T(\pi_C)) \ar[u]^-{e_*}  }
$$
commutes on the level of $*$-homomorphisms.  The second square in line \eqref{com rectangle} commutes as the  diagram 
$$
\xymatrix{ K_1(\mathcal{D}_T(\pi_B)/\mathcal{C}_T(\pi_B)) \ar[r]^-{\phi^{v}_*}  & K_1(\mathcal{D}_T(\pi_C)/\mathcal{C}_T(\pi_C)) \\
K_1(\mathcal{D}_L(\pi_B)/\CL(\pi_B)) \ar[r]^-{\phi^{v}_*} \ar[u]^-{\eta_*} & K_1(\mathcal{D}_L(\pi_C)/\CL(\pi_C)) \ar[u]^-{\eta_*} }
$$
commutes on the level of $*$-homomorphisms.  The third square commutes by naturality of the boundary map in $K$-theory.  Finally, the fourth square commutes as it commutes on the level of $*$-homomorphisms.

Now, the diagram in the statement in the lemma `factors' as the rectangle from line \eqref{com rectangle}, augmented on the top with the diagram below
$$
\xymatrix{ KK(A,B) \ar[r]^-{\phi_*} \ar[d]^-P &  KK(A,C) \ar[d]^-P\\ 
K_1(\mathcal{D}(\pi_B)/\mathcal{C}(\pi_B)) \ar[r]^-{\phi^{v_1}_*}   &  K_1(\mathcal{D}(\pi_C)/\mathcal{C}(\pi_C))}
$$
involving the Paschke duality isomorphism $P$.  To complete the proof, it suffices to show that this commutes.  We will actually work with the diagram
\begin{equation}\label{paschke square}
\xymatrix{ KK(A,B) \ar[r]^-{\phi_*}  &  KK(A,C) \\ 
K_1(\mathcal{D}(\pi_B)/\mathcal{C}(\pi_B)) \ar[u]^-{P^{-1}} \ar[r]^-{\phi^{v_1}_*}   &  K_1(\mathcal{D}(\pi_C)/\mathcal{C}(\pi_C)) \ar[u]^-{P^{-1}}}
\end{equation}
involving the inverse Paschke duality isomorphism, as this is a little simpler.

Let then $[u]\in K_1(\mathcal{D}(\pi_B)/\mathcal{C}(\pi_B))$ be a class.  The existence of a sequence of isometries as in Remark \ref{isoms} implies that we may assume that $u$ is in $\mathcal{D}(\pi_B)/\mathcal{C}(\pi_B)$ (not `just' in some matrix algebra over this $C^*$-algebra).  Abusing notation slightly, we also write $u\in \mathcal{D}(\pi_B)$ for some choice of lift of this unitary.  Then the triple $(\pi_B,E,u)$ is an even Kasparov cycle for $(A,B)$ (in the ungraded picture of the even $KK$-group) and the inverse Paschke duality map is defined by 
$$
P^{-1}:K_1(\mathcal{D}(\pi_B)/\mathcal{C}(\pi_B))\to KK(A,B),\quad [u]\mapsto [\pi_B,E,u] 
$$
(see \cite[Section 3]{Thomsen:2000aa} and \cite[Remarque 2.8]{Skandalis:1988rr}).

Now, the `up-right' composition 
$$
\xymatrix{ KK(A,B) \ar[r]^-{\phi_*}  &  KK(A,C) \\ 
K_1(\mathcal{D}(\pi_B)/\mathcal{C}(\pi_B)) \ar[u]^-{P^{-1}}   &  }
$$
composition from line \eqref{paschke square} takes $[u$ first to $[\pi_B,E,u]$, and then to $[\pi_B\otimes 1_C,E\otimes_\phi C,u\otimes 1_C]$.  On the other hand, the `right-up' composition 
$$
\xymatrix{  &  KK(A,C) \\ 
K_1(\mathcal{D}(\pi_B)/\mathcal{C}(\pi_B)) \ar[r]^-{\phi^{v_1}_*}   &  K_1(\mathcal{D}(\pi_C)/\mathcal{C}(\pi_C)) \ar[u]^-{P^{-1}}}
$$
from line \eqref{paschke square} takes $[u]$ first to $[v_1(u\otimes 1_C)v_1^*+(1-v_1v_1^*)]$ (compare \cite[Exercise 8.5]{Rordam:2000mz} to see where ``$1-v_1v_1^*$'' is coming from), and then to $[v_1(\pi_B\otimes 1_C)v_1^*,F,v_1(u\otimes 1_C)v_1^*+(1-v_1v_1^*)]$.  Our task is therefore to show that 
\begin{equation}\label{k1 desid}
[\pi_B\otimes 1_C,E\otimes_\phi C,u\otimes 1_C]=[v_1(\pi_B\otimes 1_C)v_1^*,F,v_1(u\otimes 1_C)v_1^*+(1-v_1v_1^*)].
\end{equation}
in $KK(A,C)$.

Indeed, we note first that $(\pi_B\otimes 1_C,E\otimes_\phi C,u\otimes 1_C)$ is unitarily equivalent to $(p\pi_C p,pF,v_1(u\otimes 1_C)v_1^*)$, where $p=v_1v_1^*$ vai the unitary isomorphism $v_1:E\otimes _\phi C\to pF$ (note that $p$ commutes with $\pi_C(A)$ as the compression of $\pi_C$ by $p$ agrees with $v_1(pi_B\otimes 1_C)v_1^*$, and the latter is a $*$-homomorphism.  On the other hand, the Kasparov module $((1-p)\pi_C(1-p),(1-p)F,1-p)$ is degenerate, so represents zero in $KK(A,C)$ we thus have that 
\begin{align*}
[\pi_B\otimes 1_C,E\otimes_\phi C,u\otimes 1_C] & =[p\pi_C p,pF,v_1(u\otimes 1_C)v_1^*] \\ & =[p\pi_C p,pF,v_1(u\otimes 1_C)v_1^*] \\ &\quad \quad \oplus [(1-p)\pi_C(1-p),(1-p)F,1-p].
\end{align*}
As $p=v_1v_1^*$ commutes with $\pi_C$, the right hand side agrees with the right hand side of line \eqref{k1 desid} above, and we are done.
\end{proof}

\section{Paths of projections}\label{pp sec}

Throughout this section, $A$ and $B$ refer to separable $C^*$-algebras.  All Hilbert modules are countably generated, and all are over $B$ unless explicitly stated otherwise.  All representations of $A$ are on Hilbert $B$-modules unless explicitly stated otherwise.

Our goal in this section is to introduce a new model of $KK$-theory based on paths of projections.   We will need some more terminology about representations.

\begin{definition}\label{apt}
Let $(\pi,E)$ be a representation of $A$ on a Hilbert $B$-module.
\begin{enumerate}[(i)]
\item $(\pi,E)$ is \emph{graded} if it comes with a fixed decomposition $(\pi,E)=(\pi_0\oplus \pi_1,E_0\oplus E_1)$ as a direct sum of two subrepresentations.  
\item If $\pi:A\to \LL(E)$ is a graded representation, the \emph{neutral projection} is the projection $e\in \LL(E)$ onto the first summand in $E=E_0\oplus E_1$.  
\item A graded representation $(\pi,E)$ is \emph{balanced}\footnote{Compare \cite[Definition 8.3.10]{Higson:2000bs}.} if it is graded and if $(\pi_0,E_0)=(\pi_1,E_1)$ in the given decomposition.
\item A graded representation $(\pi,E)$ is \emph{infinite multiplicity} (respectively, \emph{strongly absorbing}) if $(\pi_0,E_0)$ has infinite multiplicity in the sense of Subsection \ref{n n c} (respectively, is strongly absorbing in the sense of Definition \ref{sa rep}).
\end{enumerate}
\end{definition}

Note that a graded representation $(\pi,E)$ is balanced and infinite multiplicity if and only if 
\begin{equation}\label{hb}
(\pi,E)=(1_{\C^2\otimes \ell^2\otimes \ell^2}\otimes \sigma,\C^2\otimes \ell^2\otimes \ell^2\otimes F)
\end{equation}
for some representation $\sigma:A\to \LL(F)$: here, a tensor factor of $\ell^2$ comes from the infinite multiplicity assumption, and we use an identification $\ell^2=\ell^2\otimes\ell^2$ to split off an extra tensorial factor of $\ell^2$.  We record some useful observations arising from this as a lemma. 

\begin{lemma}\label{reps}
Let $\pi:A\to \LL(E)$ be a graded, balanced, infinite multiplicity representation.  Arising from a decomposition as in line \eqref{hb}, there are canonical unital inclusions 
\begin{equation}\label{can incl}
M_2(\C)\subseteq  \LL(E) \quad \text{and}\quad \mathcal{B}(\ell^2)\subseteq  \LL(E)
\end{equation}
as the the $C^*$-subalgebras 
$$
M_2(\C) \otimes 1_{\ell^2\otimes \ell^2\otimes F} \quad \text{and}\quad  1_{\C^2}\otimes \mathcal{B}(\ell^2)\otimes 1_{\ell^2\otimes F}
$$
respectively.
These inclusions have the following properties:
\begin{enumerate}[(i)]
\item The neutral projection corresponds to the element $\begin{psmallmatrix} 1 & 0 \\ 0 & 0 \end{psmallmatrix}\in M_2(\C)$.
\item The subalgebras $\mathcal{B}(\ell^2)$ and $M_2(\C)$ of $\LL(E)$ commute with each other, and with $A$.  
\item The compositions 
$$
\mathcal{B}(\ell^2)\to \LL(E)\to \LL(E)/\K(E) \quad \text{and} \quad M_2(\C)\to \LL(E)\to \LL(E)/\K(E)
$$
of the inclusions in line \eqref{can incl} with the quotient map to the Calkin algebra are injective.    \qed
\end{enumerate}
\end{lemma}

The following is the key definition of this section.

\begin{definition}\label{proj path}
Let $\pi:A\to \LL(E)$ be a graded representation, and define $\mathcal{P}^\pi(A,B)$ to be the set of self-adjoint contractions $p\in C_{b}([1,\infty),\LL(E))$ such that:
\begin{enumerate}[(i)]
\item $p-e\in C_{b}([1,\infty),\K(E))$\footnote{To make sense of this, we follow our usual conventions and identify $e$ with a constant function in $C_b([1,\infty),\LL(E))$, and similarly for elements of $A$ for the later parts of the definition.};
\item for all $a\in A$, $[a,p]\in C_0([1,\infty),\LL(E))$;
\item for all $a\in A$, $a(p^2-p)\in C_0([1,\infty),\K(E))$.
\end{enumerate}
\end{definition}

%We will sometimes drop the superscript ``\,$^\pi$\,'' and just write ``$\mathcal{P}(A,B)$'' when it seems unlikely to cause confusion.

Our next goal is to define an equivalence relation on $\mathcal{P}^\pi(A,B)$ such that the equivalence classes give a realization of $KK(A,B)$.  For this (and other purposes later), it will be convenient to introduce a parameter space $Y$.  Let then $C=C_0(Y)$ be a separable commutative $C^*$-algebra: for our applications, $Y$ will be one of the intervals $[0,1]$ or $(0,1)$, or the one-point compactification $\overline{\N}$ of the natural numbers.  Let $(\pi,E)$ be a representation of $A$ on a Hilbert $B$-module, and let $C\otimes E$ denote the tensor product Hilbert $C\otimes B$-module.  Let $1\otimes \pi:A\to \LL(C\otimes E)$ be the amplification of $\pi$.   If $\pi$ is graded then $1\otimes \pi$ inherits a grading in a natural way, and so if we are in the graded case we may consider $\mathcal{P}^{1\otimes \pi}(A,C\otimes B)$.

The following lemma characterizes elements of $\mathcal{P}^{1\otimes \pi}(A,C\otimes B)$ in terms of doubly parametrized families $(p_t^y)_{t\in [1,\infty),y\in Y}$.

\begin{lemma}\label{loc func lem 0}
Let $(\pi,E)$ be a graded representation of $A$ on a Hilbert $B$-module.  With notation as above, there is a natural identification between elements $p$ of $\mathcal{P}^{1\otimes \pi}(A,C\otimes B)$ and doubly parametrized families of self-adjoint contractions $(p_{t}^y)_{t\in [1,\infty),y\in Y}$ that define a function 
$$
p:[1,\infty)\to C_b(Y,\LL(E)),\quad t\mapsto (y\mapsto p^y_t)
$$
with the following properties:
\begin{enumerate}[(i)]
\item the function $p-e$ is in $C_{b}([1,\infty),C_0(Y, \K(E)))$;
\item $[p,a]\in C_0([1,\infty),C_b(Y,\LL(E)))$  for all $a\in A$;
\item $a(p^2-p)\in C_0([1,\infty),C_0(Y,\K(E)))$ for all $a\in A$.  
\end{enumerate}
\end{lemma}

\begin{proof}
An element of $\mathcal{P}^{1\otimes \pi}(A,C\otimes B)$ is a function $p:[1,\infty)\to \LL(C\otimes E)$ satisfying the conditions of Definition \ref{proj path}.  Using the canonical identifications
$$
\K(C\otimes E)=C\otimes \K(E)=C_0(Y,\K(E))
$$
(for the first, see for example \cite[pages 37 and 10]{Lance:1995ys}) and the fact that $p-e\in C_b([1,\infty),\K(E))$, we identify $p$ with a function $p:[1,\infty)\to C_b(Y,\LL(E))$ (with image in the subset $C_0(Y,\K(E))+\{e\}\subseteq C_b(Y,\LL(E))$).  The remaining checks are direct.
\end{proof}

\begin{definition}\label{pp hom}
Let $(\pi,E)$ be a graded representation of $A$ on a Hilbert $B$-module.  Elements $p^0$ and $p^1$ of $\mathcal{P}^\pi(A,B)$ are \emph{homotopic} if (with notation as in  Lemma \ref{loc func lem 0}) there is an element 
$$
p=(p_t^s)_{t\in [1,\infty),s\in [0,1]}\in \mathcal{P}^{1\otimes \pi}(A,C[0,1]\otimes B)
$$
that agrees with $p^0$ and $p^1$ at the endpoints $s\in \{0,1\}$.  We write $p^0\sim p^1$ if $p^0$ and $p^1$ are homotopic, and write $\KKP^\pi(A,B)$ for the quotient set $\mathcal{P}^\pi(A,B)/\sim$.
\end{definition}

We will need the following elementary fact a few times, so record it here for ease of reference.

\begin{lemma}\label{start irrelevant}
Assume that $p$ and $q$ are elements of $\mathcal{P}^\pi(A,B)$ such that $p_t-q_t\to0$ as $t\to\infty$.  Then $p\sim q$.
\end{lemma}

\begin{proof}
A straight line homotopy $(sp+(1-s)q)_{s\in [0,1]}$ works: we leave the direct checks involved to the reader.
\end{proof}

In order to define a semigroup structure on $\KKP^\pi(A,B)$, we assume $\pi$ is graded, balanced, and infinite multiplicity as in Definition \ref{apt}, and fix a tensorial decomposition as in line \eqref{hb} (which will remain fixed for the rest of the section).  Fix also two isometries $s_1$ and $s_2$ in $\mathcal{B}(\ell^2)$ that satisfy the Cuntz relation $s_1s_1^*+s_2s_2^*=1$.  Using the canonical (unital) inclusion $\mathcal{B}(\ell^2)\subseteq \LL(E)$ from line \eqref{can incl} in Lemma \ref{reps}, we think of these isometries as adjointable operators on $E$ that commute with $A\subseteq \LL(E)$ and with the neutral projection $e\in \LL(E)$.  

\begin{lemma}\label{group lem}
Let $\pi:A\to \LL(E)$ be a graded, balanced, and infinite multiplicity representation.  Then with notation as above, the operation defined by 
$$
[p]+ [q] :=[s_1ps_1^*+s_2qs_2^*]
$$
makes $\KKP^\pi(A,B)$ into an abelian semigroup.  The operation does not depend on the choice of $s_1,s_2$ within $\mathcal{B}(\ell^2)$.
\end{lemma}

\begin{proof}
As the unitary group of $\mathcal{B}(\ell^2)$ is connected (in the norm topology), conjugation by a unitary in $\mathcal{B}(\ell^2)$ induces the trivial map on $\KKP^\pi(A,B)$.  Hence conjugating by the unitaries $s_1s_2^*+s_2s_1^*$ and $s_1s_1s_1^*+s_1s_2s_1^*s_2^*+s_2s_2^*s_2^*$ show that the operation is commutative and associative.  On the other hand, if $t_1,t_2\in \mathcal{B}(\ell^2)$ also satisfy the Cuntz relation, then conjugating by the unitary $s_1t_1^*+s_2t_2^*$ shows that the pairs $(s_1,s_2)$ and $(t_1,t_2)$ induce the same operation on $\KKP^\pi(A,B)$.    
\end{proof}

Our next goal is to show that the semigroup $\KKP^\pi(A,B)$ is a monoid.  We first state a well-known lemma about paths of projections in a $C^*$-algebra. It follows from the arguments of \cite[Proposition 4.1.7 and Corollary 4.1.8]{Higson:2000bs}, for example.

\begin{lemma}\label{pp lem}
Let $I$ be either $[a,b]$ or $[a,\infty)$ for some $a,b\in \R$, and let $(p_t)_{t\in I}$ be a continuous path of projections in a unital $C^*$-algebra $D$.  Then there is a continuous path of unitaries $(u_t)_{t\in I}$ in $D$ such that $u_a=1$, and such that $p_t=u_tp_au_t^*$ for all $t$. \qed
\end{lemma}

\begin{lemma}\label{id lem}
Let $\pi:A\to \LL(E)$ be a graded, balanced, and infinite multiplicity representation of $A$.  Let $p$ be an element of $\mathcal{P}^\pi(A,B)$, and let $v$ be an isometry in the canonical copy of $\mathcal{B}(\ell^2)\subseteq \LL$ from line \eqref{can incl} from Lemma \ref{reps}.  Then the element
$$
q:=vpv^*+(1-vv^*)e\in C_{b}([1,\infty),\LL(E))
$$
is in $\mathcal{P}^\pi(A,B)$ and satisfies $p\sim q$.
\end{lemma}

\begin{proof}
For each $n\geq 1$, a compactness argument gives a finite rank projection 
$$
e_n\in \mathcal{K}(\ell^2)\subseteq \mathcal{B}(\ell^2)\subseteq \LL(E)
$$
(where the last inclusion is that from line \eqref{can incl} from Lemma \ref{reps}) such that 
$$
\|(1-e_n)(p_t-e)\|<\frac{1}{n}
$$ 
for all $t\in [1,n+1]$.  Choose now a projection $r_1\geq e_1$ such that $r_1-e_1$ and $1-r_1$ both have infinite rank.  Given $r_n$, define $r_{n+1}$ to be the max of $r_n$ and $e_{n+1}$.  In this way we get an increasing sequence $r_1\leq r_2\leq \cdots$ of projections in $\mathcal{B}(\ell^2)$ such that $r_n\geq e_m$ for all $n$ and all $m\leq n$, and such that $r_n-e_m$ and $1-r_n$ both have infinite rank for all $n$ and all $m\leq n$.   For each $n$, $(1-e_n)r_n$ and $(1-e_n)r_{n+1}$ are projections with infinite-dimensional kernel and image as operators on $(1-e_n)\ell^2$, and are thus connected by a continuous path of projections $(r^0_t)_{t\in [n,n+1]}$ in $\mathcal{B}((1-e_n)\ell^2)$.  Set $r_t:=e_n+r^0_{t}$ for $t\in [n,n+1]$.  In this way we get a continuous path of projections $r=(r_t)_{t\in [1,\infty)}$ in $\mathcal{B}(\ell^2)\subseteq \LL(E)$ such that if $\lfloor t\rfloor$ is the floor function of $t$ then 
\begin{equation}\label{rtpt norm}
\|(1-r_t)(p_t-e)\|\leq \|(1-e_{\lfloor t \rfloor})(p_t-e)\|< \frac{1}{\lfloor t\rfloor},
\end{equation}
and such that $r_t$ and $1-r_t$ have infinite rank as operators on $\ell^2$ for each $t$.  

Note now that as $r_t$ commutes with $e$, line \eqref{rtpt norm} implies in particular that $\|[r_t,p_t]\|<2/\lfloor t \rfloor $.  Define $p'\in C_{b}([1,\infty),\LL(E))$ by 
$$
p'_t:=r_tp_tr_t +(1-r_t)e.
$$
As $r_tp_tr_t-r_te$ is in $\K(E)$ for all $t$, we see that $p_t'-e$ is in $\K(E)$ for all $t$.   Moreover, 
$$
\|p_t'-p_t\|\leq \|[r_t,p_t]\|+\|(1-r_t)(p_t-e)\|\to 0 \quad \text{as}\quad  t\to\infty,
$$
and so $p':=(p'_t)$ defines an element of $\mathcal{P}^\pi(A,B)$ such that $p'\sim p$ by Lemma \ref{start irrelevant}.

Now, let $v\in \mathcal{B}(\ell^2)\subseteq \LL(E)$ be an isometry as in the statement of the lemma.  Lemma \ref{pp lem} gives a continuous path $(u^r_t)_{t\in [1,\infty)}$ of unitaries in $\mathcal{B}(\ell^2)$ such that $r_t=u^r_tr_1(u^r_t)^*$ for all $t$.   Similarly, we get a continuous path of unitaries $(u^v_t)_{t\in [1,\infty)}$ such that $u^v_t(1-vv^*+v(1-r_1)v^*)(u^v_t)^*=1-vv^*+v(1-r_t)v^*$ for all $t$.  Choose any partial isometry $w\in \mathcal{B}(\ell^2)$ such that $ww^*=r_1$ and $w^*w=1-vv^*+v(1-r_1)v^*$ (such exists as $r_1$ and $1-vv^*+v(1-r_1)v^*$ are both infinite rank), and define $w_t:=u^r_tw(u^v_t)^*$.  Then $(w_t)_{t\in [1,\infty)}$ is a continuous path of partial isometries in $\mathcal{B}(\ell^2)$ such that $w_tw_t^*=1-r_t$ and $w_t^*w_t=1-vv^*+v(1-r_t)v^*$.   Define 
$$
u_t:=vr_t+w_t^*\in \mathcal{B}(\ell^2)\subseteq \LL(E).
$$
Then $u=(u_t)_{t\in [1,\infty)}$ is a continuous path of unitaries such that $up'u^* = vp'v^*+(1-vv^*)e$.  Let $(h^s:\mathcal{U}(\ell^2)\to \mathcal{U}(\ell^2))_{s\in [0,1]}$ be a norm-continuous contraction of the unitary group of $\ell^2$ to the identity element (such exists by Kuiper's theorem: see for example  \cite[Theorem on page 433]{Cuntz:1987ly}) and note that the path $(h^s(u)p'h^s(u^*))_{s\in [0,1]}$ shows that $p'\sim vp'v^*+(1-vv^*)e$.  In conclusion, we have that 
$$
p\sim p'\sim vp'v^*+(1-vv^*)e\sim vpv^*+(1-vv^*)e
$$
and are done.
\end{proof}

\begin{corollary}\label{monoid}
Let $(\pi,E)$ be a graded, balanced, and infinite multiplicity representation of $A$.  Then for any $p\in \mathcal{P}^\pi(A,B)$, we have $s_1ps_1^*+s_2es_2^*\sim p$.  

In particular, the semigroup $\KKP^\pi(A,B)$ is a commutative monoid with identity given by the class $[e]$ of the neutral projection.
\end{corollary}

\begin{proof}
Apply Lemma \ref{id lem} with $v=s_1$, whence $1-vv^*=s_2s_2^*$, and use that $s_2$ commutes with $e$.
\end{proof}

Our next goal, which is the main point of this section, is to show that if $\pi$ is in graded, balanced, and strongly absorbing then $\KKP^\pi(A,B)\cong KK(A,B)$ (and therefore in particular that $\KKP^\pi(A,B)$ is a group).  We need some preliminaries.

Let $(\pi,E)$ be a graded, balanced, and infinite multiplicity representation of $A$, and fix the decomposition of line \eqref{hb} and the Cuntz isometries of Lemma \ref{group lem}.   Let $\QL(\pi)$ and $\CLc(\pi;\K)$ be as in Definitions \ref{id quot} and \ref{clpi} respectively, so Lemma \ref{cl sum} gives us a surjection $\rho:\CLc(\pi;\K)\to \QL(\pi)$.  This induces a $*$-homomorphism $\overline{\rho}:M(\CLc(\pi;\K))\to M(\QL(\pi))$ on multiplier algebras, which is uniquely determined by the condition that $\overline{\rho}(m)\cdot \rho(b)=\rho(mb)$ for all $m\in M(\CLc(\pi;\K))$ and $b\in \CLc(\pi;\K)$ (see \cite[Chapter 2]{Lance:1995ys} for this).  We define 
\begin{equation}\label{mdef}
M:=\overline{\rho}(M(\CLc(\pi;\K))),
\end{equation}
which is a unital $C^*$-subalgebra\footnote{It could be all of $M(\QL(\pi))$, although this does not seem to be obvious: note that the noncommutative Tietze extension theorem \cite[Proposition 6.8]{Lance:1995ys} is not available here as $\CLc(\pi;\K)$ is not $\sigma$-unital.} of $M(\QL(\pi))$ containing $\QL(\pi)$ as an ideal.

\begin{lemma}\label{km0}
With notation as in line \eqref{mdef} above, $M$ has trivial $K$-theory.  
\end{lemma}

\begin{proof}
The unital inclusion $\mathcal{B}(\ell^2)\subseteq \LL(E)$ of Lemma \ref{reps} induces a unital inclusion $\mathcal{B}(\ell^2)\subseteq M(\CLc(\pi;\K))$ by having $\mathcal{B}(\ell^2))$ act pointwise in the variable $t$ (this uses that $\mathcal{B}(\ell^2)$ commutes with $A$).  This in turn descends to a unital inclusion $\mathcal{B}(\ell^2)\subseteq M$.  Let $(s_n)_{n=0}^\infty$ be a sequence of isometries in $\mathcal{B}(\ell^2)\subseteq M$ with orthogonal ranges.  

Consider the maps 
$$
\iota_0:M(\CLc(\pi;\K))\to M(\CLc(\pi;\K)), \quad b\mapsto s_0 bs_0^*,
$$
and 
$$
\alpha_0:M(\CLc(\pi;\K))\to M(\CLc(\pi;\K)),\quad b\mapsto \sum_{n=1}^\infty s_nbs_n^*
$$
(the sum converges in the strict topology of $\LL(E)$, pointwise in $t$).  With $\ILc(\pi)$ as in Definition \ref{id quot}, the kernel of the map $\overline{\rho}:M(\CLc(\pi;\K))\to M$ is 
$$
\{m\in M(\CLc(\pi;\K))\mid mb\in \ILc(\pi) \text{ for all } b\in \CLc(\pi;\K)\},
$$ 
whence $\iota_0$ and $\alpha_0$ descend to well-defined $*$-homomorphisms $\iota,\alpha:M\to M$.  

As $\alpha$ and $\iota$ have orthogonal ranges, Lemma \ref{orth sum} implies that $\alpha+\iota$ is a $*$-homomorphism and that as maps on $K$-theory, $\alpha_*+\iota_*=(\alpha+\iota)_*$.  Moreover, conjugating by the isometry $s:=\sum_{n=0}^\infty s_ns_{n+1}^*\in \mathcal{B}(\ell^2)\subseteq M$ (the sum converges in the strong topology of $\mathcal{B}(\ell^2)$) and applying Lemma \ref{isom lem} implies that $(\alpha+\iota)_*=\alpha_*$ as maps on $K$-theory.  We thus have 
$$
\alpha_*+\iota_*=(\alpha+\iota)_*=\alpha_*,
$$ 
whence $\iota_*=0$.  However, $\iota_*$ is an isomorphism by Lemma \ref{isom lem} again, whence $K_*(M)$ is zero as required.
\end{proof}

We need one more preliminary definition and lemma before we get to the isomorphism $\KKP^\pi(A,B)\cong KK(A,B)$.

\begin{definition}\label{double}
For an ideal $I$ in a $C^*$-algebra $N$, the \emph{double} of $I$ along $N$ is the $C^*$-algebra defined by
$$
D_N(I):=\{(a,b)\in N\oplus N\mid a-b\in I\}.
$$
\end{definition}

Note that $D_N(I)$ fits into a short exact sequence
\begin{equation}\label{double ses}
\xymatrix{ 0 \ar[r] & I \ar[r] & D_N(I) \ar[r] & N \ar[r] & 0 }
\end{equation}
with the maps $I\to D_N(I)$ and $D_N(I)\to N$ given by $a\mapsto (a,0)$ and $(a,b)\mapsto b$ respectively. 

\begin{lemma}\label{dm iso}
Assume that $I$ is an ideal in a unital $C^*$-algebra $N$, let $D_N(I)$ be the double from Definition \ref{double}, and assume that $K_*(N)=0$.  Then $D_N(I)$ has the following properties:
\begin{enumerate}[(i)]
\item \label{dm1} The inclusion $I\to D_N(I)$ from line \eqref{double ses} induces an isomorphism on $K$-theory;
\item \label{dm2} any class in $K_0(D_N(I))$ of the form $[p,p]$ for some projection $p\in M_n(N)$ is zero;
\item \label{dm3} for any $[p,q]\in K_0(D_N(I))$, we have $-[p,q]=[q,p]$;
\item \label{dm4} any element in $K_0(D_N(I))$ can be written as $[p,q]$ for a projection $(p,q)$ in some matrix algebra $M_n(D_N(I))$.
\end{enumerate}
\end{lemma}

\begin{proof}
Part \eqref{dm1} follows from the six-term exact sequence in $K$-theory and the assumption that $K_*(N)=0$.  Part \eqref{dm2} follows as any such class is in the image of the map induced on $K$-theory by the $*$-homomorphism
$$
N\to D_N(I),\quad a\mapsto (a,a)
$$
and is thus zero as $K_*(N)=0$.  For part \eqref{dm3}, say $[p,q]\in K_0(D_N(I))$ with $p,q\in M_n(N)$.  Then $[p,q]+[q,p]=[p\oplus q,q\oplus p]$.  As $p-q\in M_n(I)$, the formula 
$$
[0,\pi/2]\owns s\mapsto \Bigg(\begin{pmatrix} p & 0 \\ 0 &  q\end{pmatrix} , \begin{pmatrix} \cos(s) & \sin(s) \\ -\sin(s) & \cos(s)\end{pmatrix} \begin{pmatrix} q & 0 \\ 0 & p \end{pmatrix} \begin{pmatrix} \cos(s) & -\sin(s) \\ \sin(s) & \cos(s)\end{pmatrix} \Bigg)
$$
defines a homotopy between $(p\oplus q,q\oplus p)$ and $(p\oplus q,p\oplus q)$ passing through projections in $M_{2n}(D_N(I))$.  The latter defines the zero class in $K_0$ by part \eqref{dm2}, which gives part \eqref{dm3}.  Part \eqref{dm4} follows directly from part \eqref{dm3}.
\end{proof}

Here is the main result of this section.

\begin{theorem}\label{kk iso}
Let $\pi:A\to \LL(E)$ be a graded, balanced, and strongly absorbing representation on a Hilbert $B$-module $E$.  Let $M$ be as in line \eqref{mdef}, $\QL(\pi)$ as in Definition \ref{id quot} and $D_M(\QL(\pi))$ be as in Definition \ref{double}.  Then the formula 
\begin{equation}\label{kkp to dmq}
\KKP^\pi(A,B) \to K_0(D_M(\QL(\pi))),\quad [p]\mapsto [p,e]
\end{equation}
defines an isomorphism of commutative monoids.  In particular $\KKP^\pi(A,B)$ is an abelian group.

Moreover, there is a canonical isomorphism $\KKP^\pi(A,B)\cong KK(A,B)$. 
\end{theorem}

\begin{proof}
We first have to show that the map in line \eqref{kkp to dmq} is well-defined.  It is not difficult to see that if $p\in \mathcal{P}^\pi(A,B)$, then $(p,e)$ is a projection in $D_M(\QL(\pi))$.  For well-definedness, we need to show that if $p^0\sim p^1$ in $\mathcal{P}^\pi(A,B)$, then the projections $(p^0,e)$ and $(p^1,e)$ in $D_M(\QL(\pi))$ define the same $K$-theory class.  Let then $(p^s)_{s\in [0,1]}$ be a homotopy implementing the equivalence between $p^0$ and $p^1$.  Let 
$$
1\otimes \pi:A\to \LL(C[0,1]\otimes E)
$$
be the amplification of $\pi$ to the $C[0,1]\otimes B$-module $C[0,1]\otimes E$, and let $\CLc(1\otimes \pi)$ be the associated localization algebra.  Note that $p:=(p^s)_{s\in [0,1]}$ defines an element of the multiplier algebra $M(\CLc(1\otimes \pi))$ such that $p-e$ is in $\CLc(1\otimes \pi)$, and so that $[p,e]$ is a well-defined class in $D_{M_C}(\QL(1\otimes \pi))$, where $M_C$ is defined analogously to $M$, but starting with $1\otimes \pi$.  

As $\pi$ is graded, balanced, and strongly absorbing we can write $E=E_0\oplus E_1$ and Remark \ref{usa unique} implies that $E_0$ is isomorphic as a Hilbert $B$-module to $\ell^2\otimes B$.  Hence we may apply Proposition \ref{sa tp} to conclude that $1\otimes \pi$ is (graded, balanced, and) strongly absorbing, and thus Theorems \ref{kk iso dww} and \ref{c to uc} give an isomorphism 
$$
KK(A,C[0,1]\otimes B) \stackrel{\cong}{\to} K_0(\CLc(1\otimes \pi)).
$$
Let $\epsilon^0,\epsilon^1:C[0,1]\otimes B\to B$ be given by evaluation at the endpoints.   Lemma \ref{func prop} then gives a commutative diagram 
$$
\xymatrix{ KK(A,C[0,1]\otimes B) \ar[d]^-{\epsilon^i_*} \ar[r]^-\cong  & K_0(\CLc(\iota\otimes \pi)) \ar[d]^-{\epsilon^i_*} \\  KK(A,B) \ar[r]^-\cong  & K_0(\CLc(\pi))}
$$
for $i\in \{0,1\}$.   Homotopy invariance of $KK$-theory gives that the maps $\epsilon^0_*,\epsilon^1_*:KK(A,C[0,1]\otimes B)\to KK(A,B)$ are the same, whence the maps $\epsilon^0_*,\epsilon^1_*:K_0(\CLc(1\otimes \pi))\to K_0(\CLc(\pi))$ are too.  On the other hand each $\epsilon^i$ induces maps $\epsilon^i:\QL(1\otimes \pi)\to \QL(\pi)$ and $\epsilon^i:\CLc(1\otimes \pi;\K)\to \CLc(\pi;\K)$, and therefore induces a map $D_{M_C}(\QL(1\otimes \pi))\to D_M(\QL(1\otimes \pi))$.  All this gives rise to a commutative diagram
$$
\xymatrix{ \CLc(1\otimes \pi) \ar[d]^-{\epsilon^i}  \ar[r]  & \QL(1\otimes \pi) \ar[d]^-{\epsilon^i}  \ar[r] & D_{M_C}(\QL(1\otimes \pi))\ar[d]^-{\epsilon^i} \\ \CLc(\pi)  \ar[r] & \QL(\pi)  \ar[r]  & D_{M}(\QL(\pi)) }
$$
where the first pair of horizontal maps are the canonical quotients, and the second pair are the inclusions $a\mapsto (a,0)$.    The horizontal maps induce isomorphisms on $K$-theory by Corollary \ref{quot iso} (first pair), and Lemmas \ref{km0} and \ref{dm iso} (second pair).  Hence the maps $\epsilon^0_*,\epsilon^1_*:K_0(D_{M_C}(\QL(1\otimes \pi)))\to K_0(D_M(\QL(\pi)))$ are the same.  We thus see that
$$
[p^0,e]=\epsilon^0_*[p,e]=\epsilon^1_*[p,e]=[p^1,e],
$$
which is the statement needed for well-definedness of the map in line \eqref{kkp to dmq}.

We now show that the map in line \eqref{kkp to dmq} statement is a homomorphism.  Indeed, for $p,q\in \mathcal{P}^\pi(A,B)$, the element $[s_1ps_1^*+s_2qs_2^*]$ of $\KKP^\pi(A,B)$ gets sent to 
$$
[s_1ps_1^*+s_2qs_2^*,e]=[s_1ps_1^*+s_2qs_2^*,s_1e s_1^*+s_2es_2^*],
$$
where we have used that $s_1s_1^*+s_2s_2^*=1$ and that $s_1,s_2$ commute with $e$.  As $s_1xs_1^*$ is orthogonal to $s_1ys_2^*$ for any $x,y$ we have that  
$$
[s_1ps_1^*+s_2qs_2^*,s_1e s_1^*+s_2es_2^*]=[s_1ps_1^*,s_1es_1^*]+[s_2qs_2^*,s_2es_2^*]
$$
and as conjugation by $s_1$ and $s_2$ has no effect on $K$-theory by Lemma \ref{isom lem}, this equals 
$$
[p,e]+[q,e],
$$
which is the sum of the images of $[p]$ and $[q]$.

We now show that the map in line \eqref{kkp to dmq} is surjective.  Using Lemma \ref{dm iso}, an arbitrary element of $K_0(D_M(\QL(\pi)))$ can be represented as a class $[p,q]$ with $p,q$ projections in $M_n(M)$ for some $n$, and with $p-q\in M_n(\QL(\pi))$.  We have that $[1-q,1-q]=0$ by Lemma \ref{dm iso}, and thus $[p,q]=[p\oplus 1-q,q\oplus 1-q]$.  The matrix $u=\begin{psmallmatrix} q & 1-q \\ 1-q & q \end{psmallmatrix}$ is a unitary in $M_{2n}(M)$, whence conjugating by $(u,u)$ we see that 
$$
[p,q]=[p\oplus 1-q,q\oplus 1-q]=[u(p\oplus q)u^*,u(q\oplus 1-q)u^*]=[u(p\oplus q)u^*,1_n\oplus 0_n],
$$
where $1_n$ and $0_n$ are the unit and zero in $M_n(M)$.  Choose now $2n$ isometries $v_1,...,v_{2n}$ in $\mathcal{B}(\C^2\otimes \ell^2)\subseteq \LL$ such that $\sum_{i=1}^n v_iv_i^*=e$ and $\sum_{i=1}^{2n} v_iv_i^*=1_{2n}$.  The matrix
$$
v:=\begin{pmatrix} v_1 & v_2 & \cdots & v_{2n} \\ 0 & 0 & \cdots & 0 \\ \vdots & \vdots & \ddots & \vdots \\ 0 & 0 & \cdots & 0 \end{pmatrix} \in M_{2n}(M)
$$
is an isometry, whence conjugation by $(v,v)$ induces the trivial map on $K_0(D_M(\QL(\pi)))$ by Lemma \ref{isom lem}.  Hence 
$$
[p,q]=[vu(p\oplus q)u^*v^*,v(1_n\oplus 0)v^*]=[r\oplus 0_{2n-1},e\oplus 0_{2n-1}],
$$
where $r\in M$ is a projection such that $a:=r-e$ is in $\QL(\pi)$.  We may lift $a$ to a self-adjoint element $b\in \CLc(\pi;\K)$ by Lemma \ref{cl sum}.  Consider the self-adjoint element $(b+e,e)\in D_{M(\CLc(\pi;\K))}(\CLc(\pi;\K))$, which maps to $(r,e)\in D_M(\QL(\pi))$ under the $*$-homomorphism
$$
D_{M(\CLc(\pi;\K))}(\CLc(\pi;\K))\to D_M(\QL(\pi))
$$
induced by the quotient map $\CLc(\pi;\K)\to \QL(\pi)$ of Lemma \ref{cl sum}.  Note that if $f:\R\to [-1,1]$ is the function defined by 
$$
f(t):=\left\{\begin{array}{ll} 1 & t> 1 \\ t & -1\leq t\leq 1 \\ -1 & t<-1 \end{array}\right.
$$
then in $D_{M(\CLc(\pi;\K))}(\CLc(\pi;\K))$
$$
f(b,e)=(f(b+e),f(e))=(f(b+e),e),
$$
and this element still maps to $(r,e)$ by naturality of the functional calculus.  Set that $c=f(b+e)$.  Then one checks that $c$ is an element of $\mathcal{P}^\pi(A,B)$ such that $[c,e]=[r,e]=[p,q]$, so we are done with surjectivity.  

To see injectivity of the map in line \eqref{kkp to dmq}, assume that $[p]\in \KKP^\pi(A,B)$ is such that $[p,e]$ is zero in $K_0(D_M(\QL(\pi)))$.  In particular, $[p,e]=[e,e]$ by Lemma \ref{dm iso}, and therefore there is a projection $(q_1,q_2)\in M_{n}(D_M(\QL(\pi)))$ and a homotopy $p_{(1)}=(p_{(1)}^s)_{s\in [0,1]}$ between $(p\oplus q_1,e\oplus q_2)$ and $(e\oplus q_1,e\oplus q_2)$ in $M_{n+1}(D_M(\QL(\pi)))$.  We will manipulate this homotopy to build a homotopy between $p$ and $e$ in $\mathcal{P}^\pi(A,B)$.
\begin{itemize}
\item Replacing $p_{(1)}$ by $p_{(2)}:=p_{(1)}\oplus (q_2,q_1)$, we get a homotopy between $(p\oplus q_1\oplus q_2,e\oplus q_2\oplus q_1)$ and $(e\oplus q_1\oplus q_2,e\oplus q_2\oplus q_1)$.  
\item As $q_1-q_2\in M_n(\QL(\pi))$, we get a homotopy 
$$
s\mapsto \Big(p\oplus q_1\oplus q_2,e\oplus \begin{psmallmatrix} \cos(s) & \sin(s) \\ -\sin(s) & \cos(s)\end{psmallmatrix}\begin{psmallmatrix} q_2 & 0 \\ 0 & q_1 \end{psmallmatrix}\begin{psmallmatrix} \cos(s) & \sin(s) \\ -\sin(s) & \cos(s)\end{psmallmatrix}\Big)
$$
between $(p\oplus q_1\oplus q_2,e\oplus q_2\oplus q_1)$ and $(p\oplus q_1\oplus q_2,e\oplus q_1\oplus q_2)$, and similarly between $(e\oplus q_1\oplus q_2,e\oplus q_2\oplus q_1)$ and $(e\oplus q_1\oplus q_2,e\oplus q_1\oplus q_2)$.  Concatenating these with the homotopy $p_{(2)}$ gives a homotopy $(p_{(3)}^s)_{s\in [0,1]}$ between $(p\oplus q_1\oplus q_2,e\oplus q_1\oplus q_2)$ and $(e\oplus q_1\oplus q_2,e\oplus q_1\oplus q_2)$.
\item Setting $r=q_1\oplus q_2$ and replacing $p_{(3)}$ with 
$$
p_{(4)}^s:=p_{(3)}\oplus \big( (1-r),(1-r)\big)
$$ 
gives a homotopy between $(p\oplus r\oplus (1-r),e\oplus r\oplus 1-r\oplus 0_{4n})$ and $(e\oplus r\oplus (1-r),e\oplus r\oplus (1-r))$.
\item Set $u=\begin{psmallmatrix} r & 1-r \\ 1-r & r\end{psmallmatrix}$, which is a unitary in $M_{4n}(M)$.  Moreover, $u$ is self-adjoint, so connected to the identity via some path $(u^s)_{s\in [0,1]}$ of unitaries.   Then 
$$
(1\oplus u^{s} ,1\oplus u^{s})\big(p\oplus r\oplus (1-r),e\oplus r\oplus (1-r)\big)(1\oplus u^{s} ,1\oplus u^{s})^*
$$
defines a homotopy between $\big(p\oplus r\oplus (1-r),e\oplus r\oplus (1-r)\big)$ and $(p\oplus 1_{2n}\oplus 0_{2n},e\oplus 1_{2n}\oplus 0_{2n})$.  Similarly, we get a homotopy between $\big(e\oplus r\oplus (1-r),e\oplus r\oplus (1-r)\big)$ and $(e\oplus 1_{2n}\oplus 0_{2n},e\oplus 1_{2n}\oplus 0_{2n})$.  Concatenating these with $p_{(4)}$ gives a homotopy $p_{(5)}$ between $(p\oplus 1_{2n}\oplus 0_{2n},e\oplus 1_{2n}\oplus 0_{2n})$ and $(e\oplus 1_{2n}\oplus 0_{2n},e\oplus 1_{2n}\oplus 0_{2n})$.
\item Write $p_{(5)}^s=(p_{0}^{s},p_{1}^{s})$ for paths of projections $(p^{s}_0)_{s\in [0,1]}$ and $(p^{s}_1)_{s\in [0,1]}$ in $M_{4n+1}(M)$. Then Lemma \ref{pp lem} gives a continuous path of unitaries $(v^s)_{s\in [0,1]}$ in $M_{4n+1}(M)$ with $v^0=1$, and $p^{s}_1=v_s(e\oplus 1_{2n}\oplus 0_{2n})v_s^*$ for all $s\in [0,1]$.  Note in particular that $v_1(e\oplus 1_{2n}\oplus 0_{2n})v_1^*=(e\oplus 1_{2n}\oplus 0_{2n})$, even though we may not have that $v^1=1$.  Define then
$$
p_{(6)}^s:=(v_s,v_s)^*p_{(5)}^s (v_s,v_s),
$$
which gives a new homotopy between $(p\oplus 1_{2n}\oplus 0_{2n},e\oplus 1_{2n}\oplus 0_{2n})$ and $(e\oplus 1_{2n}\oplus 0_{2n},e\oplus 1_{2n}\oplus 0_{2n})$ with the additional property of being constant in the second variable.
\item Let $M_{1\times 4n}(M)$ be the $1\times 4n$ row matrices, and choose an isometry $w\in M_{1\times 4n}(\mathcal{B}(\ell^2))\subseteq M$ be such that $w(1_{2n}\oplus 0_{2n})w^*=s_2es_2^*$.  Define 
$$
t:=\begin{pmatrix} s_1 & w \end{pmatrix}\in M_{1\times 4n+1}(\mathcal{B}(\ell^2))\subseteq M_{1\times 4n+1}(M),
$$
which is an isometry, and define $p_{(7)}^s:= tp_{(6)}^st^*$.  Then this is a homotopy in $D_M(\QL(\pi))$ between $(s_1ps_1^*+s_2es_2^*,s_1es_1^*+s_2es_2^*)$ and $(s_1es_1^*+s_2es_2^*,s_1es_1^*+s_2es_2^*)$ that is constant in the second variable.  
\end{itemize}
Now, restricting the homotopy $p_{(7)}$ to the first variable gives a homotopy of projections in $M$, say $(p^s)_{s\in [0,1]}$, between $s_1ps_1^*+s_2es_2^*$ and $e$, and such that $p^s-e$ is in $\QL(\pi)$ for all $s$.  The function 
$$
[0,1]\to D_M(\QL(\pi)),\quad s\mapsto (p^s,e)
$$
defines an idempotent, say $q$, in $C[0,1]\otimes D_M(\QL(\pi))$.  As the natural $*$-homomorphism 
$$
C[0,1]\otimes D_{M(\CLc(\pi;\K))}(\CLc(\pi;\K))\to C[0,1]\otimes D_M(\QL(\pi))
$$
is surjective (this follows from Lemma \ref{cl sum}), $q$ lifts to a self-adjoint contraction of the form 
$$
(a,e)\in C[0,1]\otimes D_{M(\CLc(\pi;\K))}(\CLc(\pi;\K))
$$
analogously to the argument at the end of the surjectivity part.  The element $a$ defines a homotopy in $\mathcal{P}^\pi(A,B)$ between $s_1ps_1^*+s_2es_2^*$ and $e$.  On the other hand, $s_1ps_1^*+s_2es_2^*\sim p$ by Corollary \ref{monoid}, whence we have 
$$
p\sim s_1ps_1^*+s_2es_2^*\sim e.
$$
Corollary \ref{monoid} then shows that $[p]=0$, and so we have injectivity.

To complete the proof of Theorem \ref{kk iso}, we need to show that $\KKP^\pi(A,B)\cong KK(A,B)$.  This follows by combining: the isomorphism $\KKP^\pi(A,B)\cong K_0(D_M(\QL(\pi)))$ established above; the isomorphism $K_*(D_M(\QL(\pi)))\cong K_*(\QL(\pi))$ of Lemma \ref{dm iso}; the isomorphism $K_*(\QL(\pi))\cong K_*(\CLc(\pi))$ of Corollary \ref{quot iso}; and the isomorphism $K_0(\CLc(\pi))\cong KK(A,B)$ of Theorems \ref{kk iso dww} and \ref{c to uc}.
\end{proof}

Finally in this section we establish a technical lemma about functoriality that we will need later.

\begin{lemma}\label{loc func lem}
Let $\pi:A\to \LL(E)$ be a graded, balanced, and strongly absorbing representation on a Hilbert $B$-module, and let $C=C_0(Y)$ be a separable and commutative $C^*$-algebra.  For $y\in Y$, let $e^y:C_0(Y)\to \C$ be the $*$-homomorphism defined by evaluation at $y$.  Let $\phi_B:KK(A,B) \to \KKP^\pi(A,B)$ be the isomorphism of Theorem \ref{kk iso}.  Then if $p$ is an element of $\mathcal{P}^{1\otimes \pi}(A,C\otimes B)$ with corresponding family $(p_{t}^y)_{t\in [1,\infty),y\in Y}$ as in Lemma \ref{loc func lem 0}, we have that 
$$
e^y_*(\phi_{C\otimes B}^{-1}[p])=\phi_B^{-1} [p^y].
$$
\end{lemma}

\begin{proof}
The map
$$
\mathcal{P}^{1\otimes \pi}(A,C\otimes B)\to \mathcal{P}^\pi(A,B),\quad p\mapsto p^y
$$
induces a homomorphism 
$$
e^y_*:\KKP^{1\otimes \pi}(A,C\otimes B)\to \KKP^\pi(A,B).
$$
Moreover, with notation as in the first paragraph of the proof of Theorem \ref{kk iso}, $e^y$ induces $*$-homomorphisms 
$$
e^y:\QL(1\otimes \pi)\to \QL(\pi) \quad\text{and}\quad e^y:D_{M_C}(\QL(1\otimes \pi))\to D_M(\QL(\pi)).
$$
Consider now the diagram 
$$ 
\xymatrix{ KK(A,C\otimes B) \ar[d] \ar[r]^-{e^y_*} & KK(A,B) \ar[d] \\
K_0(\CLc(1\otimes \pi)) \ar[d] \ar[r]^-{e^y_*}  & K_0(\CLc( \pi)) \ar[d] \\
K_0(\QL(1\otimes \pi)) \ar[d]  \ar[r]^-{e^y_*} & K_0(\QL(\pi)) \ar[d] \\
K_0(D_{M_C}(\QL(1\otimes \pi))) \ar[r]^-{e^y_*}  \ar[d] & K_0(D_M(\QL(\pi))) \ar[d] \\
\KKP^{1\otimes \pi}(A,C\otimes B)\ar[r]^-{e^y_*} &  \KKP^\pi(A,B) }
$$
\noindent{}where: the first pairs of vertical arrows are the isomorphisms of Theorems \ref{kk iso dww} and \ref{c to uc}; the second pair of vertical arrows are induced by the canonical quotient map; the third pair of vertical arrows are induced by the inclusion $a\mapsto (a,0)$; and the  last pair of vertical arrows are the isomorphisms of Theorem \ref{kk iso}.  The first square commutes by Lemma \ref{func prop} (using also Proposition \ref{sa tp} to see that the representation $1\otimes \pi$ is strongly absorbing).  It is straightforward to see that the remaining squares commute: we leave this to the reader.  As the isomorphisms $\phi_{C\otimes B}$ and $\phi_B$ are by definition the compositions of the all the vertical arrows on the left and right respectively, the result follows.
\end{proof}

\section{The topology on $KK$}\label{top sec}

Throughout this section, $A$ and $B$ refer to separable $C^*$-algebras.  All Hilbert modules are countably generated, and all are over $B$ unless explicitly stated otherwise.  All representations of $A$ are on Hilbert $B$-modules unless explicitly stated otherwise.

Our goal in this section is to recall the canonical topology on $KK(A,B)$, and describe it in terms of the isomorphism $KK(A,B)\cong \KKP^\pi(A,B)$ of Theorem \ref{kk iso}.

We need a quantitative version of Definition \ref{proj path}; this will also be important to us later when we define our controlled $KK$-theory groups.  See Definition \ref{apt} for graded representations and the neutral projection $e$ used in the next definitions.

\begin{definition}\label{alm com}
Let $A$ and $B$ be separable $C^*$-algebras, and let $\pi:A\to \LL(E)$ be a graded representation on a Hilbert $B$-module.  Let $X$ be a finite subset of the unit ball $A_1$ of $A$, and let $\epsilon>0$.  Define $\mathcal{P}^\pi_{\epsilon}(X,B)$ to be the set of self-adjoint contractions $p$ in $\LL(E)$ satisfying the following conditions:
\begin{enumerate}[(i)]
\item $p-e$ is in $\K(E)$;
\item $\|[p,a]\|<\epsilon$ for all $a\in X$;
\item $\|a(p^2-p)\|<\epsilon$ for all $a\in X$.
\end{enumerate}
\end{definition}

For the next definition, see Definition \ref{proj path} for the notation $\mathcal{P}^\pi(A,B)$.

\begin{definition}\label{tause}
Let $A$ and $B$ be separable $C^*$-algebras, and let $\pi:A\to \LL(E)$ be a graded representation on a Hilbert $B$-module.  For a finite subset $X$ of $A_1$ and $\epsilon>0$, define a function $\tau_{X,\epsilon}:\mathcal{P}^\pi(A,B)\to [1,\infty)$ by 
$$
\tau_{X,\epsilon}(p):=\inf\{t_0\in [1,\infty)\mid p_t\in \mathcal{P}^\pi_\epsilon(X,B) \text{ for all } t\geq t_0\}.
$$
For each $p\in \mathcal{P}^\pi(A,B)$, define $U(p;X,\epsilon)$ to be the subset of $\mathcal{P}^\pi(A,B)$ consisting of all $q$ such that there exists $t\geq \max\{\tau_{X,\epsilon}(p),\tau_{X,\epsilon}(q)\}$ and a norm continuous path $(p^s)_{s\in [0,1]}$ in $\LL(E)$ such that each $p^s$ is in $\mathcal{P}_\epsilon^\pi(X,B)$, and with endpoints $p^0=p_t$ and $p^1=q_t$.
\end{definition}

For the next lemma, recall the homotopy equivalence relation $\sim$ on $\mathcal{P}^\pi(A,B)$ from Definition \ref{pp hom}.

\begin{lemma}\label{quot lem}
Let $\pi:A\to \LL(E)$ be a graded representation of $A$ on a graded Hilbert $B$-module.  Let $p\in \mathcal{P}^\pi(A,B)$, $X$ be a finite subset of $A_1$, and $\epsilon>0$.  Then:
\begin{enumerate}[(i)]
\item \label{ql1} if $p'\sim p$, then $U(p;X,\epsilon)=U(p';X,\epsilon)$;
\item \label{ql2} if $q\in U(p;X,\epsilon)$ and $q\sim q'$, then $q'\in U(p;X,\epsilon)$.
\end{enumerate}
\end{lemma}

\begin{proof}
Part \eqref{ql2} follows from part \eqref{ql1} on noting that $q$ is in $U(p;X,\epsilon)$ if and only if $p$ is in $U(q;X,\epsilon)$.  It thus suffices to prove \eqref{ql1}.

Assume then that $p\sim p'$, so there is a homotopy $(p^s)_{s\in [0,1]}$ in $\mathcal{P}^{1\otimes \pi}(A,C[0,1]\otimes B)$ between $p$ and $p'$.  The definition of a homotopy gives $t_p\geq \max\{\tau_{X,\epsilon}(p),\tau_{X,\epsilon}(p')\}$ such that $p_{t_p}^s$ is in $\mathcal{P}^\pi_\epsilon(X,B)$ for all $s\in [0,1]$.  Let $q$ be an element of $U(p;X,\epsilon)$, and let $t_q\geq \{\tau_{X,\epsilon}(q),\tau_{X,\epsilon}(p)\}$ be such that there is a homotopy $(q^s)_{s\in [0,1]}$ connecting $p_{t_q}$ and $q_{t_q}$.  Write $I$ for whichever of the intervals $[t_p,t_q]$ or $[t_q,t_p]$ makes sense.  Then concatenating the homotopies $(p^s_{t_p})_{s\in [0,1]}$, $(p_t)_{t \in I}$ and $(q^s)_{s\in [0,1]}$ shows that $q$ is in $U(p';X,\epsilon)$.  Hence $U(p;X,\epsilon)\subseteq U(p';X,\epsilon)$.  The opposite inclusion follows by symmetry.
\end{proof}

\begin{definition}\label{top}
Let $\pi:A\to \LL(E)$ be a graded representation on a Hilbert $B$-module.  For a finite subset $X$ of $A_1$, $\epsilon>0$, and $[p]\in \KKP^\pi(A,B)$, define the \emph{$X$-$\epsilon$ neighbourhood} of $[p]$ to be 
$$
V([p];X,\epsilon):=\{[q]\in \KKP^\pi(A,B) \mid q\in U(p;X,\epsilon)\}.
$$
(note that $V([p];X,\epsilon)$ does not depend on the representative of the class $[p]$ by Lemma \ref{quot lem}).  The \emph{asymptotic topology} on $\KKP^\pi(A,B)$ is the topology generated by the subsets $V([p];X,\epsilon)$ of $\KKP^\pi(A,B)$ as $X$ ranges over finite subsets of $A_1$, $\epsilon$ over $(0,\infty)$, and $p$ over $\mathcal{P}^\pi(A,B)$.
\end{definition}

\begin{lemma}\label{fc lem}
Let $\pi:A\to \LL(E)$ be a graded representation on a Hilbert $B$-module.  For any $[p]\in \KKP^\pi(A,B)$, the collection of sets $V([p];X,\epsilon)$ as $X$ ranges over finite subsets of $A_1$ and $\epsilon$ over $(0,\infty)$ form a neighbourhood base of $[p]$.  Moreover, the asymptotic topology is first countable.
\end{lemma}

\begin{proof}
Using Lemma \ref{quot lem}, the asymptotic topology on $\KKP^\pi(A,B)$ is the quotient topology induced by the canonical surjection $\mathcal{P}^\pi(A,B)\to \KKP^\pi(A,B)$, where $\mathcal{P}^\pi(A,B)$ is equipped with the topology generated by the sets $U(p;X,\epsilon)$; moreover, the quotient map is open.  It thus suffices to show that the family 
$$
\{U(p;X,\epsilon)\mid X\subseteq A_1 \text{ finite, } \epsilon>0\}$$
form a neighborhood basis of $p\in \mathcal{P}^\pi(A,B)$, and that this topology on $\mathcal{P}^\pi(A,B)$ is first countable.

For the neighbourhood base claim, we must show that whenever $q_1,...,q_n$, $X_1,...,X_n$ and $\epsilon_1,...,\epsilon_n$ are such that $p\in \bigcap_{i=1}^n U(q_i;X_i,\epsilon_i)$, then there exist $X$, $\epsilon$ with 
$$
U(p;X,\epsilon)\subseteq \bigcap_{i=1}^n U(q_i;X_i,\epsilon_i).
$$
As whenever $Y\supseteq X$ and $\delta\leq \epsilon$, we have that $U(p;Y,\delta)\subseteq U(p;X,\epsilon)$, it suffices to prove this for $n=1$.  Assume then we are given $q\in \mathcal{P}^\pi(A,B)$, a finite subset $X\subseteq A_1$, and $\epsilon>0$ such that $p\in U(q;X,\epsilon)$.  We claim that $U(p;X,\epsilon)\subseteq U(q;X,\epsilon)$, which will suffice to complete the neighbourhood base part of the proof.  Indeed, say $r$ is in $U(p;X,\epsilon)$.  Then there exists $t_r\geq \max\{\tau_{X,\epsilon}(p),\tau_{X,\epsilon}(r)\}$ and a homotopy $(r^s)_{s\in [0,1]}$ passing through $\mathcal{P}^\pi_\epsilon(X,B)$ connecting $p_{t_r}$ and $r_{t_r}$.  Similarly, there exists $t_q\geq \max\{\tau_{X,\epsilon}(p),\tau_{X,\epsilon}(q)\}$ and a homotopy $(q^s)_{s\in [0,1]}$ passing through $\mathcal{P}^\pi_\epsilon(X,B)$ connecting $q_{t_q}$ and $p_{t_q}$.  Let $I$ be the closed interval bounded by $t_r$ and $t_q$.  Then concatenating the three paths $(q^s)_{s\in [0,1]}$, $(p_{t})_{t\in I}$, and $(r^s)_{s\in [0,1]}$ shows that $r$ is in $U(q;X,\epsilon)$, so we are done.

We now show first countability.  As $A$ is separable, there exists a nested sequence $X_1\subseteq X_2\subseteq $ of finite subsets of the unit ball $A_1$ with dense union.  Fix a point $p\in \mathcal{P}^\pi(A,B)$.  We claim that the sets $U(p;X_n,1/n)$ form a neighbourhood basis at $p$.  Indeed, given what we have already proved, it suffices to show that for any finite $X\subseteq A_1$ and any $\epsilon>0$ there exists $n$ with $U(p;X_n,1/n)\subseteq U(p;X,\epsilon)$.  Let $n$ be so large so that for all $a\in X$ there is $a'\in X_n$ with $\|a-a'\|<\epsilon/2$, and also so that $1/n<\epsilon/2$.  From the choice of $n$, it follows that $\mathcal{P}_{1/n}^\pi(X_n,B)\subseteq \mathcal{P}^\pi_\epsilon(X,B)$, from which the inclusion $U(p;X_n,1/n)\subseteq U(p;X,\epsilon)$ follows.
\end{proof}

We now recall the canonical topology on $KK(A,B)$, which has been introduced and studied in different pictures by several authors: see for example the discussion in \cite{Dadarlat:2005aa} for background and references.  Dadarlat\footnote{Dadarlat attributes some of the idea here to unpublished work of Pimsner.} showed in \cite[Lemma 3.1]{Dadarlat:2005aa} that this topology is characterized by the following property (and used this to show that the various different descriptions that had previously appeared in the literature agree).

\begin{proposition}\label{kk top} 
Let $A$ and $B$ be separable $C^*$-algebras.  Let $\overline{\N}=\N\cup\{\infty\}$ be the one point compactification of the natural numbers, and for each $n\in \overline{\N}$, let $e^n:C(\overline{\N},B)\to B$ be the $*$-homomorphism defined by evaluation at $n$.  Then the canonical topology on $KK(A,B)$ is characterized by the following conditions.
\begin{enumerate}[(i)]
\item It is first countable.
\item \label{kk top con} A sequence $(x_n)$ in $KK(A,B)$ converges to $x_\infty$ in $KK(A,B)$ if and only if there is an element $x\in KK(A,C(\overline{\N},B))$ such that $e^n_*(x)=x_n$ for all $n\in \overline{\N}$. \qed
\end{enumerate}
\end{proposition}

\begin{theorem}\label{tops same}
Let $\pi:A\to \LL(E)$ be a graded, balanced, and strongly absorbing representation.  Then the isomorphism of Theorem \ref{kk iso} is a homeomorphism between the asymptotic topology on $\KKP^\pi(A,B)$ and the canonical topology on $KK(A,B)$. 
\end{theorem}

We need an ancillary lemma.

\begin{lemma}\label{close proj}
Let $\pi:A\to \LL(E)$ be a graded representation on a Hilbert $B$-module.  For any $\epsilon>0$ and any finite $X\subseteq A_1$, if $p,q\in \mathcal{P}^\pi_{\epsilon/2}(X,B)$ satisfy $\|p-q\|<\epsilon/6$, then there exists a homotopy $(p^s)_{s\in [0,1]}$ connecting $p$ and $q$ and passing through $\mathcal{P}^\pi_{\epsilon}(X,B)$.
\end{lemma}

\begin{proof}
A straight line homotopy between $p$ and $q$ works. We leave the direct checks to the reader.
\end{proof}

\begin{proof}[Proof of Theorem \ref{tops same}]
We have already seen that the asymptotic topology is first countable in Lemma \ref{fc lem}.  Hence by Proposition \ref{kk top}, it suffices to show that sequential convergence in the asymptotic topology is characterized by condition \eqref{kk top con} from Proposition \ref{kk top}.

Assume first that $([p^n])_{n\in \overline{\N}}$ is a collection of elements of $\KKP^\pi(A,B)$.  Let $1\otimes \pi$ be the amplification of $\pi$ to the Hilbert $C(\overline{\N})\otimes B$ module $C(\overline{\N})\otimes E$, and let $q\in \mathcal{P}^{1\otimes \pi}(A,C(\overline{\N},B))$ be such that for all $n\in \overline{\N}$ we have $e^n_*[q]=[p^n]$.  We want to show that the sequence $([p^n])_{n\in \N}$ converges to $[p^\infty]$ in the asymptotic topology.  For this, it follows from Lemmas \ref{quot lem} and \ref{fc lem} that it suffices to fix a finite subset $X$ of $A_1$ and $\epsilon>0$, and show that $p^n$ is in $U(p^\infty;X,\epsilon)$ for all suitably large $n$.

Recall the function $\tau_{X,\epsilon}$ of Definition \ref{tause}.  As $q$ is an element of $\mathcal{P}^{1\otimes \pi}(A,C(\overline{\N},B))$, the number $\tau:=\sup_{n\in \overline{\N}}\tau_{X,\epsilon/2}(q^n)$ is finite.  As $q$ is in $\mathcal{P}^{1\otimes \pi}(A,C(\overline{\N},B))$ we also see from Lemma \ref{loc func lem 0} that there exists $N$ such that $\|q^n_\tau-q^\infty_\tau\|<\epsilon/6$ for all $n\geq N$.  We claim that $p^n$ is in $U(p^\infty;X,\epsilon)$ for all $n\geq N$, which will complete the first half of the proof.

Using Lemma \ref{loc func lem 0}, we may identify $q$ with a collection $(q^n)_{n\in \overline{\N}}$ of elements of $\mathcal{P}^\pi(A,B)$ (satisfying certain conditions).  Now let $n\geq N$ and consider the following homotopies.
\begin{enumerate}[(i)]
\item As $e^\infty_*[q]=[p^\infty]$, Theorem \ref{kk iso} implies that $q^\infty\sim p^\infty$, and thus there is $t_\infty\geq \max\{\tau,\tau_{X,\epsilon}(p^\infty)\}$ and a homotopy passing through $\mathcal{P}^\pi_\epsilon(X,B)$ and connecting $p^\infty_{t_\infty}$ and $q^\infty_{t_\infty}$.  
\item Similarly to (i), there is $t_n\geq \max\{\tau,\tau_{X,\epsilon}(p^n)\}$ and a homotopy passing through $\mathcal{P}^\pi_\epsilon(X,B)$ and connecting $p^n_{t_n}$ and $q^n_{t_n}$.  
\item As $\|q^n_\tau-q^\infty_\tau\|<\epsilon/6$ for all $n\geq N$ and as $\tau=\sup_{n\in \overline{\N}}\tau_{X,\epsilon/2}(q^n)$, Lemma \ref{close proj} gives a homotopy passing through $\mathcal{P}^\pi_\epsilon(X,B)$ and connecting $q^\infty_\tau$ and $q^n_\tau$.  
\item The path $(q^n_t)_{t\in [\tau,t_n]}$ is a homotopy passing through $\mathcal{P}^\pi_\epsilon(X,B)$ that connects $q^n_\tau$ and $q^n_{t_n}$.
\item The path $(q^\infty_t)_{t\in [\tau,t_\infty]}$ is a homotopy passing through $\mathcal{P}^\pi_{\epsilon}(X,B)$ that connects $q^\infty_\tau$ and $q^\infty_{t_\infty}$.  
\end{enumerate}
Now let $t_{\max}=\max\{t_n,t_\infty\}$.  Concatenating the five homotopies above with the homotopies $(p_t^n)_{t\in [t_n,t_{\max}]}$ and $(p^\infty_t)_{t\in [t_\infty,t_{\max}]}$ (which pass through $P^\pi_\epsilon(X,B)$) shows that $p^n$ is in $U(p^\infty;X,\epsilon)$ for $n\geq N$.

For the converse, fix a sequence $X_1\subseteq X_2\subseteq \cdots$ of nested finite subsets of $A_1$ with dense union.  Let us assume that $([p^n])_{n\in \N}$ is a sequence in $\KKP^\pi(A,B)$ that converges to $[p^\infty]$ in the asymptotic topology.  We want to construct an element $q\in \mathcal{P}^{1\otimes \pi}(A,C(\overline{\N},B))$ such that $e^n_*[q]=[p^n]$ for each $n$ in $\overline{\N}$.  We will define new representatives of the classes $[p^n]$ as follows.  For each $m$, let $N_m\in \N$ be the smallest natural number such that $p^n$ is in $U(p^\infty;X_m,1/m)$ for all $n\geq N_m$; as $[p^n]$ converges to $[p^\infty]$ in the asymptotic topology, such an $N_m$ exists, and the sequence $N_1,N_2,...$ is non-decreasing.  

Choose a sequence $t_1\leq t_2\leq \cdots$ in $[1,\infty)$ that tends to infinity in the following way.  For $n< N_1$, let $t_n=1$.  Otherwise, let $m$ be as large as possible so that $n\geq N_m$.  Let $t_n=\max\{\tau_{X_m,1/m}(p^n),\tau_{X_m,1/m}(p^\infty),t_1,...,t_{n-1}\}+1$, and note the choice of $N_m$ implies that $p^n\in U(p^\infty;X_m,1/m)$, so there exists a homotopy between $p^n_{t_n}$ and $p^\infty_{t_n}$ parametrized as usual by $[0,1]$ that passes through $\mathcal{P}^\pi_{1/m}(X_m,B)$.  Approximating this homotopy by a piecewise-linear homotopy, we may assume that it is Lipschitz, and still passing through $\mathcal{P}^\pi_{1/m}(X_m,B)$.  Moreover, by lengthening the interval parametrizing the homotopy, we may assume that it is $1$-Lipschitz.  In conclusion, for some suitably large $r_n\in [1,\infty)$, we may assume that we have a $1$-Lipschitz homotopy $(p^{n,t})_{t\in [t_n,t_n+r_n]}$ between $p^n_{t_n}$ and $p^\infty_{t_n}$.   Define for each $n\in \N$
$$
q^n:=\left\{\begin{array}{ll} p^\infty_t & t\in [1,t_n] \\ p^{n,t} & t\in [t_n,t_n+r_n] \\ p^n_t & t\geq t_n+r_n \end{array}\right.,
$$
and note that $[q^n]=[p^n]$ for all $n\in \N$ using Lemma \ref{start irrelevant}.  Define $q^\infty=p^\infty$.  Using the characterization of Lemma \ref{loc func lem 0}, one checks directly that $q=(q^n)_{n\in \overline{\N}}$ defines an element of $\mathcal{P}^{1\otimes \pi}(A,C(\overline{\N},B))$.  This element satisfies $e^n_*[q]=[p^n]$ by construction, so we are done.
\end{proof}

\section{Controlled $KK$-theory and $KL$-theory}\label{kl sec}

Throughout this section, $A$ and $B$ refer to separable $C^*$-algebras.  All Hilbert modules are countably generated, and all are over $B$ unless explicitly stated otherwise.  All representations of $A$ are on Hilbert $B$-modules unless explicitly stated otherwise.

Recall from the introduction that, following Dadarlat \cite[Section 5]{Dadarlat:2005aa}, we define the \emph{$KL$-group of $(A,B)$} by 
\begin{equation}\label{kl def}
KL(A,B):=KK(A,B)/\overline{\{0\}},
\end{equation}
i.e.\ the quotient of $KK(A,B)$ by the closure of zero for the topology from Proposition \ref{kk top}.  This makes $KL(A,B)$ into a Hausdorff topological group when equipped with the quotient topology.  As already mentioned in the introduction, $KL(A,B)$ was originally introduced by R\o{}rdam \cite[Section 5]{Rordam:1995aa} using a purely algebraic definition that agrees with the intrinsic topological definition above under a UCT assumption on $A$.  

Our goal in this section is to define \emph{controlled $K$-homology groups} $KK_\epsilon^\pi(X,B)$, arrange these into an inverse system, and show that the inverse limit of these is canonically isomorphic to $KL(A,B)$; the isomorphism moreover holds on the level of topological groups when the inverse limit is taken in the category of topological groups and each $KK_\epsilon(X,B)$ is given the discrete topology.

We start with the controlled $KK$-theory groups.  For the next definition, recall the notion of a graded representation (plus other conditions on representations) from Definition \ref{apt}, and the set $\mathcal{P}^\pi_\epsilon(X,B)$ from Definition \ref{alm com}.

\begin{definition}\label{con kk}
Let $\pi:A\to \LL(E)$ be a graded representation of $A$ on a Hilbert $B$-module.  Let $X\subseteq A_1$ be a finite subset of the unit ball $A_1$ of $A$ and let $\epsilon>0$.  Equip $\mathcal{P}^\pi_\epsilon(X,B)$ with the norm topology it inherits from $\LL(E)$, and define $KK^\pi_\epsilon(X,B):=\pi_0(\mathcal{P}^\pi_\epsilon(X,B))$ to be the associated set of path components.
\end{definition}

Our first goal is to define a group structure on $KK^\pi_\epsilon(X,B)$.  For this, let us assume that the representation $(\pi,E)$ is graded, balanced, and infinite multiplicity (see Definition \ref{apt} for terminology), and fix two isometries $s_1,s_2$ in $\mathcal{B}(\ell^2)$ satisfying the Cuntz relation $s_1s_1^*+s_2s_2^*=1$.  Using the inclusion $\mathcal{B}(\ell^2)\subseteq \LL(E)$ from line \eqref{can incl} in Lemma \ref{reps}, we think of $s_1$ and $s_2$ as isometries in $\LL(E)$ that commute with the subalgebra $\pi(A)$ and the neutral projection $e$.  We define an operation on $KK^\pi_\epsilon(X,B)$ by 
$$
[p]+[q]:=[s_1ps_1^*+s_2qs_2^*].
$$
The following lemma can be proved in exactly the same way as Lemma \ref{group lem}: we leave the direct checks involved to the reader.

\begin{lemma}\label{group lem 2}
With notation as above, the set $KK^\pi_\epsilon(X,B)$ is a commutative semigroup.  The group operation does not depend on the choice of $s_1$ and $s_2$. \qed
\end{lemma}

In order to show that $KK^\pi_\epsilon(X,B)$ is a monoid, we need an analogue of Lemma \ref{id lem}.

\begin{lemma}\label{id lem 2}
Let $\pi:A\to\LL(E)$ be a graded, balanced, and infinite multiplicity representation of $A$ on a Hilbert $B$-module.  Let $X$ be a finite subset of $A_1$, let $\epsilon>0$, let $p$ be an element of $\mathcal{P}^\pi_\epsilon(X,B)$, and let $v$ be an isometry in the canonical copy $\mathcal{B}(\ell^2)\subseteq \LL(E)$ from line \eqref{can incl} from Lemma \ref{reps}.  Then the formula 
$$
vpv^*+(1-vv^*)e
$$
defines an element of $\mathcal{P}^\pi_\epsilon(X,B)$ in the same path component as $p$.
\end{lemma}

\begin{proof}
The fact that $vpv^*+(1-vv^*)e$ is an element of $\mathcal{P}^\pi_\epsilon(X,B)$ follows from the fact that $v$ commutes with $A$.  We fix $\delta\in (0,1)$, to be determined in the course of the proof by $X$, $p$, and $\epsilon$.   As $p-e$ is in $\K(E)$, there exists an infinite rank projection $r\in \mathcal{B}(\ell^2)$ such that $1-r$ also has infinite rank, and such that 
\begin{equation}\label{rtpt norm 2}
\|(1-r)(p-e)\|<\delta.
\end{equation}
Note that as $r$ commutes with $e$, line \eqref{rtpt norm 2} implies that 
\begin{equation}\label{r p com}
\|[r,p]\|<2\delta.
\end{equation}
As $r$ is a projection and commutes with elements of $A$, and as $p$ is a contraction, this implies that for any $a\in X$, 
\begin{equation}\label{rpr proj}
\|a((rpr)^2-rpr)\|\leq \|r[p,r]pr\|+\|ra(p^2-p)r\|<2\delta+\max_{a\in X} \|a(p^2-p)\|.
\end{equation}

Define now $q:=rpr +(1-r)e$, which is a self-adjoint contraction.  Note that $q-e=rpr-re=r(p-e)r$, 
so $q-e$ is in $\K(E)$.  We have $q^2-q=(rpr)^2-rpr$, and so line \eqref{rpr proj} implies that for all $a\in X$,
$$
\|a(q^2-q)\|<2\delta+\max_{a\in X} \|a(p^2-p)\|.
$$ 
Moreover, 
$$
\|q-p\|=\|rpr-rp+(1-r)e-(1-r)p\|\leq  \|[r,p]\|+\|(1-r)(p-e)\|<3\delta
$$
by lines \eqref{rtpt norm} and \eqref{r p com}.  Hence as long as $\delta$ is so suitably small (depending on $\epsilon$ and ${\displaystyle \epsilon-\max_{a\in X} \|a(p^2-p)\|}$), we see that $q$ is in $\mathcal{P}^\pi_\epsilon(X,B)$. Moreover, for suitably small $\delta$, we have that the path 
$$
[0,1]\mapsto \LL(E), \quad s\mapsto sp+(1-s)q
$$
is norm continuous and passes through $\mathcal{P}^\pi_\epsilon(X,B)$, and so shows that $p\sim q$.  Hence it suffices to prove the result with $p$ replaced by $q$.  

Now, let $v\in \mathcal{B}(\ell^2)\subseteq \LL(E)$ be an isometry as in the original statement.  Choose a partial isometry $w\in \mathcal{B}(\ell^2)$ such that $ww^*=1-r$ and $w^*w=1-vv^*+v(1-r)v^*$; such exists as the operators appearing on the right hand sides of these equations are infinite rank projections.  Define 
$$
u:=vr+w^*\in \mathcal{B}(\ell^2)\subseteq \LL(E).
$$
Then one checks that $u$ is unitary, and moreover that $uqu^* = vqv^*+(1-vv^*)e$.  Let $(u_s)_{s\in [0,1]}$ be any norm continuous path of unitaries in $\mathcal{B}(\ell^2)$ connecting $u$ to the identity.  Then if we write ``$r\sim s$'' to mean that $r,s\in \mathcal{P}^\pi_\epsilon(X,B)$ are in the same path component, the homotopy $(u_squ_s^*)_{s\in [0,1]}$ shows that $q\sim vqv^*+(1-vv^*)e$.  In conclusion, we have that 
$$
p\sim q\sim vqv^*+(1-vv^*)e\sim vpv^*+(1-vv^*)e,
$$
where the last `$\sim$' follows from the homotopy $\big(v(sp+(1-s)q)v^*+(1-vv^*)e\big)_{s\in [0,1]}$.
\end{proof}

\begin{corollary}\label{kkeps mon}
Let $\pi:A\to\LL(E)$ be a graded, balanced, and infinite multiplicity representation.  Let $X$ be a finite subset of $A_1$, let $\epsilon>0$.  Then the commutative semigroup $KK^\pi_\epsilon(X,B)$ is a commutative monoid with identity element $[e]$.  
\end{corollary}

\begin{proof}
This follows from Lemma \ref{group lem 2}, and Lemma \ref{id lem 2} with $v=s_1$ (whence $1-vv^*=s_2s_2^*$).  
\end{proof}

\begin{proposition}\label{kkeps gp}
Let $\pi:A\to\LL(E)$ be a graded, balanced, and infinite multiplicity representation.  Let $X$ be a finite subset of $A_1$ and let $\epsilon>0$.  Then the monoid $KK_\epsilon^\pi(X,B)$ is a group.
\end{proposition}

\begin{proof}
For simplicity of notation, if $p,q\in \mathcal{P}_\epsilon^\pi(X,B)$, write $p\sim q$ if $p,q$ are in the same path component, i.e.\ if they represent the same element of $KK^\pi_\epsilon(X,B)$.

Let $p$ be an element of $\mathcal{P}^\pi_\epsilon(X,B)$.  According to Lemma \ref{group lem 2} and Corollary \ref{kkeps mon} it suffices to show that $p$ has an inverse, i.e.\ to find an element $q\in \mathcal{P}_\epsilon(X,B)$ such that 
\begin{equation}\label{q inverse}
s_1qs_1^*+s_2ps_2^*\sim e.
\end{equation}
Let $M_2(\C)$ be unitally included in $\LL(E)$ as in line \eqref{can incl} from Lemma \ref{reps}, and let $u$ be the element $\begin{psmallmatrix} 0 & 1 \\ 1 & 0 \end{psmallmatrix}\in M_2(\C)$, so $ueu=1-e$.   The self-adjoint contraction 
$$
q:=s_1es_1^*+s_2u(1-p)us_2^*
$$
then defines an element of $\mathcal{P}^\pi_\epsilon(X,B)$, and we claim this satisfies the condition in line \eqref{q inverse} above.

We first define 
$$
v:=s_2s_1^*s_1^*+s_1s_1s_2^*+s_1s_2s_2^*s_1^*,
$$
which is unitary in $\mathcal{B}(\ell^2)\subseteq \LL(E)$.  Computing,
\begin{equation}\label{v stuff}
v(s_1qs_1^*+s_2ps_2^*)v^*=s_1(s_1ps_1^*+s_2u(1-p)us_2^*)s_1^*+s_2es_2^*.
\end{equation}
On the other hand, Lemma \ref{id lem 2} (with $v=s_1$) implies that 
\begin{equation}\label{no e}
s_1(s_1ps_1^*+s_2u(1-p)us_2^*)s_1^*+s_2es_2^*\sim s_1ps_1^*+s_2u(1-p)us_2^*. 
\end{equation}
Moreover, $v$ is connected to the identity in the unitary group of $\mathcal{B}(\ell^2)$; as $\mathcal{B}(\ell^2)$ commutes with $A$ and $e$, this implies that 
\begin{equation}\label{no v}
vrv^*\sim r \quad \text{for any} \quad r\in \mathcal{P}_\epsilon^\pi(X,B).
\end{equation}
Combining lines \eqref{v stuff}, \eqref{no e}, and \eqref{no v} we get that 
$$
s_1qs_1^*+s_2ps_2^*\sim s_1ps_1^*+s_2u(1-p)us_2^*.
$$
Comparing this with line \eqref{q inverse}, it thus suffices to show that 
\begin{equation}\label{desid sim}
s_1ps_1^*+s_2u(1-p)us_2^*\sim e.
\end{equation}
We will show this in two steps:
\begin{enumerate}[(i)]
\item \label{step 1 inv} construct a homotopy $(p_t)_{t\in [0,1]}$ between $s_1ps_1^*+s_2u(1-p)us_2^*$ and $s_1s_1^*$ such that $\|[p_t,a]\|<\epsilon$ and $\|a(p_t^2-p_t)\|<\epsilon$ for all $t$ and all $a\in X$;
\item \label{step 2 inv} show how to `fix' this homotopy so that is also satisfies $p_t-e\in \K(E)$ for all $t\in [0,1]$.
\end{enumerate}

Let us start on Step \eqref{step 1 inv} above.  Connect $u$ to the identity through unitary elements of $M_2(\C)$.  This gives a path, say $(p_t^{(0)})_{t\in [0,1]}$ connecting $s_1ps_1^*+s_2u(1-p)us_2^*$ to $s_1ps_1^*+s_2(1-p)s_2^*$ and that satisfies $\|[p_t^{(0)},a]\|<\epsilon$ and $\|a((p_t^{(0)})^2-p_t^{(0)})\|<\epsilon$ for all $t$ and all $a\in X$.  

At this point, to simplify notation, let us write elements of $\LL(E)$ as $2\times 2$ matrices with respect to the matrix units $e_{ij}:=s_is_j^*$.  With this notation\footnote{In more formal notation, $p_t^{(1)}=s_1(p+\sin^2(t)(1-p))s_1^*+s_1\cos(t)\sin(t)(1-p)s_2^* +s_2\cos(t)\sin(t)(1-p)s_1^*+s_2\cos^2(t)(1-p)s_2^*$.}, consider the path $(p_t^{(1)})_{t\in [0,\pi/2]}$ defined by 
$$
p_t^{(1)}:=\begin{pmatrix} p & 0 \\ 0 & 0 \end{pmatrix} +\begin{pmatrix} \cos(t) & \sin(t) \\ -\sin(t) & \cos(t) \end{pmatrix} \begin{pmatrix} 0 & 0 \\ 0 & 1-p \end{pmatrix}\begin{pmatrix} \cos(t) & -\sin(t) \\ \sin(t) & \cos(t) \end{pmatrix}.
$$
One computes that 
$$
(p_t^{(1)})^2-p_t^{(1)}=\begin{pmatrix} \cos(t)(p^2-p) & 0 \\ 0 & \cos^2(t)(p^2-p) \end{pmatrix},
$$
whence $\|a((p_t^{(1)})^2-p_t^{(1)})\|<\epsilon$ for all $t\in [0,\pi/2]$ and all $a\in X$.  Another computation gives that for any $a\in A$ and $t\in [0,\pi/2]$,
$$
[a,p_t^{(1)}]=[a,p]\begin{pmatrix} \cos^2(t) & -\cos(t)\sin(t) \\ -\cos(t)\sin(t) & -\cos^2(t) \end{pmatrix}.
$$
The norm of the matrix appearing on the right hand side is $|\cos(t)|$, and therefore $\|[a,p_t^{(1)}]\|<\epsilon$ for all $a\in X$ and all $t\in [0,\pi/2]$.  

Concatenating the paths $(p_t^{(0)})_{t\in [0,1]}$ and $(p_t^{(1)})_{t\in [0,\pi/2]}$, and reparametrizing, we get a new path $(p_t)_{t\in [0,1]}$ connecting $s_1ps_1^*+s_2u(1-p)us_2^*$ and $s_1s_1^*$.  This completes Step \eqref{step 1 inv} above.

For Step \eqref{step 2 inv}, let $\varpi:\LL(E)\to \LL(E)/\K(E)$ be the quotient map.  With respect to the decomposition in Lemma \ref{reps}, $\varpi$ is injective on the canonical copy of $\mathcal{B}(\C^2\otimes \ell^2)\subseteq \LL(E)$.  As $p-e\in \K(E)$ and $e$ is a projection, we see that the path $(\varpi(p_t))_{t\in [0,1]}$ passes through projections in $\mathcal{B}(\C^2\otimes \ell^2)$, and it connects $e$ and $s_1s_1^*$.  Hence using Lemma \ref{pp lem}, there exists a continuous path of unitaries $(w_t)_{t\in [0,1]}$ in $\mathcal{B}(\C^2\otimes \ell^2)$ with $w_0=1$ and such that $\varpi(p_t)=w_t\varpi(p_0)w_t^*$ for all $t$.  The path $(w_t^*p_tw_t)_{t\in [0,1]}$ then lies in $\mathcal{P}^\pi_{\epsilon}(X,B)$, and connects $s_1ps_1^*+s_2u(1-p)us_2^*$ and $e$.  This completes Step \eqref{step 2 inv}, and the proof.
\end{proof}

We now have that $KK_\epsilon^\pi(X,B)$ is a group whenever $\pi$ is graded, balanced, and infinite multiplicity.  Under a slightly stronger assumption on the representations involved, the groups $KK^\pi_\epsilon(X,B)$ are independent of the representation $\pi$.

\begin{lemma}\label{rep ind}
Let $\pi:A\to \LL(E)$ and $\sigma:A\to \LL(F)$ be graded, balanced, and strongly absorbing representations in the sense of Definition \ref{apt}.  Then the groups $KK^\pi_\epsilon(X,B)$ and $KK^\sigma_\epsilon(X,B)$ are isomorphic.  
\end{lemma}

\begin{proof}
Analogously to line \eqref{hb} above, we may write 
$$
(\pi,E)=(1_{\C^2\otimes \ell^2}\otimes \pi_0,\C^2\otimes \ell^2\otimes E_0)
$$
where $\pi_0:A\to \LL(E_0)$ is a strongly absorbing representation of $A$, and analogously for $(\sigma,F)$.  Let $e^E=\begin{psmallmatrix} 1 & 0 \\ 0 & 0 \end{psmallmatrix} \otimes 1_{\ell^2\otimes E_0}$ denote the neutral projection for $(\pi,E)$ as in Definition \ref{apt}, and choose a pair of isometries in $s_1^E,s_2^E\in 1_{\C^2}\mathcal{B}(\ell^2)\otimes 1_{F_0}$ satisfying the Cuntz relation and implementing the group operation on $KK^\epsilon(X,B)$ as in Lemma \ref{group lem 2}.  Let $e^F$ be defined similarly, and let $s_1^F$ and $s_2^F$ be chosen compatibly with $s_1^E$ and $s_2^E$.  

Proposition \ref{cov isom} gives us a path of unitaries 
$$
u^0=(u_t^0)_{t\in [1,\infty)}\in C_b([1,\infty),\LL(F_0,E_0))
$$
such that $u_t^*\pi_0(a)u_t-\sigma_0(a)\to 0$ as $t\to\infty$, and such that $u_t^*\pi(a)u_t-\sigma(a)\in \mathcal{K}(F_0)$ for all $t$.  Let $u_t=1_{\C^2\otimes \ell^2}\otimes u^0_t\in \LL(F,E)$.   

Now, let $p\in \mathcal{P}^\pi_\epsilon(X,B)$.  One checks that there exists $t_p\geq 1$ such that for all $t\geq t_p$ suitably large, $u_t^*pu_t\in \mathcal{P}^\sigma_\epsilon(X,B)$.  Provisionally define a group homomorphism by  
\begin{equation}\label{rep change map}
KK^\pi_\epsilon(X,B)\to KK^\sigma_\epsilon(X,B),\quad [p]\mapsto [u_t^*pu_t] \text{ for any } t\geq t_p.
\end{equation}
Note that the class $[u_t^*pu_t]$ does not depend on the choice of $t$: indeed, if $t,s\geq t_p$ and $I$ is the interval between them, then $(u_t^*pu_t)_{t\in I}$ defines a homotopy between $u_t^*pu_t$ and $u_s^*pu_s$ in $\mathcal{P}_\epsilon^\sigma(X,B)$.  Note also that this map does not depend on the choice of representative $p$ of the class $[p]$.  Indeed, if $p^0$ and $p^1$ are two such representatives, then there is a homotopy $\mathbf{p}=(p^s)_{s\in [0,1]}$ between them in $\mathcal{P}_\epsilon^\pi(X,B)$.  A compactness argument gives $t_\mathbf{p}\geq 1$ such that for all $t\geq t_{\mathbf{p}}$ and all $s\in [0,1]$, $u_t^*p^su_t\in \mathcal{P}_\epsilon^\sigma(X,B)$.  Hence for any $t\geq t_{\mathbf{p}}$, $(u_t^*p^su_t)_{s\in [0,1]}$ defines a homotopy between $u_t^*p^0u_t$ and $u_t^*p^1u_t$.  

From the discussion above, the map in line \eqref{rep change map} is well-defined.  The choice of $u_t$ guarantees that $u_t^*s^E_iu_t=s_i^F$ for $i\in \{0,1\}$ and all $t$, and also that $u_te^Eu_t^*=e^F$ for all $t$.  Hence the map in line \eqref{rep change map} is a group homomorphism.  Switching the roles of $u_t$ and $u_t^*$ gives a group homomorphism in the other direction; as the two are easily seen to be mutually inverse, this completes the proof that $KK^\pi_\epsilon(X,B)$ is isomorphic to $KK^\sigma_\epsilon(X,B)$.
\end{proof}

From Lemma \ref{rep ind}, the next definition is reasonable.

\begin{definition}\label{con kk gp}
Let $\pi:A\to\LL(E)$ be a graded, balanced, and strongly absorbing representation on a Hilbert $B$-module.  We call $KK^\pi_\epsilon(X,B)$ the \emph{controlled $KK$-theory group} of $A$ associated to $X$, and $\epsilon$.
\end{definition}

Having established that each $KK^\pi_\epsilon(X,B)$ is a group, we now arrange these groups into an inverse system, and show that the resulting inverse limit agrees with $KL(A,B)$.  

\begin{definition}\label{seps}
Let $\mathcal{X}$ be the set of all pairs $(X,\epsilon)$ where $X$ is a finite subset of $A_1$, and $\epsilon>0$.  We equip $\mathcal{X}$ with the partial order defined by $(X,\epsilon)\leq (Y,\delta)$ if for any graded representation $\pi:A\to \LL(E)$ we have that $\mathcal{P}^\pi_\delta(Y,B)\subseteq \mathcal{P}^\pi_\epsilon(X,B)$.
\end{definition}

\begin{remark}\label{seps rem}
We record some basic properties of the partially ordered set $\mathcal{X}$.
\begin{enumerate}[(i)]
\item \label{naive} Note that $(X,\epsilon)\leq (Y,\delta)$ if $X\subseteq Y$ and $\delta\leq \epsilon$.
\item It follows from this that $\mathcal{X}$ is directed: an upper bound for $(X_1,\epsilon_1)$ and $(X_2,\epsilon_2)$ is given by $(X_1\cup X_2,\min\{\epsilon_1,\epsilon_2\})$.
\item \label{cof seq} Recall that a subset $S$ of a directed set $I$ is \emph{cofinal} if for all $i\in I$ there is $s\in S$ with $s\geq i$.  The partial order in Definition \ref{seps} contains a lot more comparable pairs than those arising from the  `naive ordering' on the set $\mathcal{X}$ defined by `$(X,\epsilon)\leq (Y,\delta)$ if $X\subseteq Y$ and $\delta\leq \epsilon$' as in \eqref{naive} above.  For example, the naive ordering never contains cofinal sequences (even for $A=\C)$, while the ordering from Definition \ref{seps} always does.  To see this, let $(a_n)_{n=1}^\infty$ be a dense sequence in $A_1$, and define $X_n:=\{a_1,...,a_n\}$.  Then the sequence $(X_n,1/n)_{n=1}^\infty$ is cofinal in $\mathcal{X}$ for the ordering from Definition \ref{seps}.
\item \label{cof seq 2} If $A$ is generated as a $C^*$-algebra by a finite set $X\subseteq A_1$, then the sequence $(X,1/n)_{n=1}^\infty$ is cofinal in $\mathcal{X}$.
\end{enumerate}
\end{remark}

As it is not so widely used in $C^*$-algebra theory, and to establish conventions, we recall the definition of an inverse limit in the category of topological abelian groups.  We will need both a purely algebraic notion and a notion that incorporates a topology: for the purely algebraic notion, just omit all the words in parentheses, and include them all for the version incorporating a topology.

\begin{definition}\label{inv lim def}
Let $I$ be a directed set.  An \emph{inverse system} of (topological) abelian groups indexed by $I$ is a collection $(G_i)_{i\in I}$ of (topological) abelian groups, and (continuous) group homomorphisms $\phi_{ij}:G_j\to G_i$ for each $j\geq i$ such that $\phi_{ii}=\text{id}_{G_i}$ for each $i$, and such that $\phi_{ij}\circ \phi_{jk}$ whenever $k\geq j\geq i$.  

Given such an inverse system, its \emph{inverse limit} is defined to be 
$$
\lim_{\leftarrow}G_i:=\Bigg\{(g_i)\in \prod_{i\in I} G_i \mid \phi_{ij}(g_j)=g_i \text{ for all }j\geq i\Bigg\},
$$
equipped with the group operations (and the topology that it inherits from $\prod_{i\in I} G_i$).
\end{definition}

We recall some basic facts about inverse limits; we leave the direct checks involved to the reader.

\begin{remark}\label{gen inv lim rem}
\begin{enumerate}[(i)]
\item \label{inv lim uni} The inverse limit ${\displaystyle \lim_{\leftarrow}G_i}$ of an inverse system $(G_i)$ of (topological) abelian groups has the following universal property.  Restricting to the $i^\text{th}$ coordinate gives a canonical (continuous) homomorphism $\phi_i:{\displaystyle \lim_{\leftarrow}G_i}\to G_i$ for each $i$.  Then for any (topological) abelian group $H$ equipped with a family of (continuous) homomorphisms $\psi_i:H\to G_i$ such that the diagrams
$$
\xymatrix{ & H \ar[dl]_-{\psi_j} \ar[dr]^-{\psi_i} & \\
G_j \ar[rr]^-{\phi_{ij}} & & G_i }
$$
commute for each $j\geq i$, there is a unique (continuous) homomorphism $H \to {\displaystyle \lim_{\leftarrow}G_i}$ making the following diagrams 
$$
\xymatrix{ & H \ar@/_1pc/[ddl]_-{\psi_j} \ar@/^1pc/[ddr]^-{\psi_j} \ar[d] & \\
& {\displaystyle \lim_{\leftarrow}G_i}\ar[dl]_-{\phi_j} \ar[dr]^-{\phi_i} & \\ 
G_j \ar[rr]^-{\phi_{ij}} & & G_i }
$$
commute for all $j\geq i$.
\item \label{cof iso} Any cofinal subset $J$ of a directed set $I$ defines the same inverse limit.  Precisely, the canonical `restriction' map 
$$
\prod_{i\in I}G_i\to \prod_{i\in J} G_i
$$
defined by forgetting all coordinates outside of $J$ restricts to an isomorphism ${\displaystyle \lim_{\leftarrow,i\in I} G_i\to  \lim_{\leftarrow,i\in J} G_i}$, which is also a homeomorphism in the topological setting.
\end{enumerate}
\end{remark}

We now come to the inverse limit we are interested in.  

\begin{definition}\label{good i l}
Assume that $(X,\epsilon)\leq (Y,\delta)$ as elements of the partially ordered set $\mathcal{X}$ from Definition \ref{seps}.  Then for any graded representation $\pi:A\to \LL(E)$ there is a canonical `forget control' map
\begin{equation}\label{forget}
\varphi_{X,\epsilon}^{Y,\delta}:KK^\pi_\delta(Y,B)\to KK^\pi_\epsilon(X,B)
\end{equation}
induced by the inclusion $\mathcal{P}_\delta^\pi(Y,B)\to \mathcal{P}_\epsilon^\pi(X,B)$.  

Assume moreover that $\pi$ is balanced and infinite multiplicity so that each $KK_\epsilon^\pi(X,B)$ is an abelian group, which we equip with the discrete topology.  Then the collection 
$$
(KK^\pi_\epsilon(X,B))_{(X,\epsilon)\in \mathcal{X}}
$$
becomes an inverse system of topological abelian groups and we define ${\displaystyle \lim_{\leftarrow}KK^\pi_\epsilon(X,B)}$ to be its inverse limit as in Definition \ref{inv lim def}.
\end{definition}

Our goal in the remainder of this section is to show that with notation as in line \eqref{kl def} and Definition \ref{good i l}
$$
\lim_{\leftarrow}KK^\pi_\epsilon(X,B)\cong KL(A,B)
$$ 
whenever $\pi$ is graded, balanced, and strongly absorbing as in Definition \ref{apt}.

For the next lemma, recall the notation $\tau_{X,\epsilon}(p)$ from Definition \ref{tause} above and the notation $\KKP^\pi(A,B)$ from Definition \ref{pp hom}.  It is an abelian group under the assumptions in the Lemma by Theorem \ref{kk iso}.

\begin{lemma}\label{comp map}
Let $\pi:A\to \LL(E)$ be a graded, balanced, and strongly absorbing representation of $A$ on a Hilbert $B$-module.  For each $(X,\epsilon)$ in the set $\mathcal{X}$ of Definition \ref{seps} there is a group homomorphism
$$
\psi_{X,\epsilon}:\KKP^\pi(A,B)\to KK^\pi_{\epsilon}(X,B)
$$
defined by sending $[p]$ to the class of $[p_{t}]$, where $t=\tau_{X,\epsilon}(p)+1$.  Moreover, the family of homomorphisms $(\psi_{X,\epsilon})_{(X,\epsilon)\in \mathcal{X}}$ are compatible with the forget control maps in line \eqref{forget} above in the sense that the diagrams
$$
\xymatrix{\KKP^\pi(A,B) \ar[d]^-{\psi_{Y,\delta}} \ar@{=}[r] & \KKP^\pi(A,B) \ar[d]^-{\psi_{X,\epsilon}} \\ KK^\pi_\delta(Y,B) \ar[r]^-{\varphi_{X,\epsilon}^{Y,\delta}} & KK^\pi_\epsilon(X,B) }
$$
commute.  In particular, the maps $\psi_{X,\epsilon}$ determine a group homomorphism 
$$
\psi:\KKP^\pi(A,B)\to \lim_{\leftarrow}KK_\epsilon(X,B).
$$
\end{lemma}

\begin{proof}
To see that the map $\psi_{X,\epsilon}$ is well-defined, let $(p^s)_{s\in [0,1]}$ be a homotopy between $p^0,p^1$ in the set $\mathcal{P}^\pi(A,B)$ of paths of projections as in Definition \ref{pp hom}.  Let $t_0=\tau_{X,\epsilon}(p^0)+1$, $t_1=\tau_{X,\epsilon}(p^1)+1$, and choose $t_2$ such that $t_2\geq \max \{t_0,t_1\}$, and such that $p^s_{t_2}$ is in $\mathcal{P}^\pi_\epsilon(X,B)$ for all $s$ (such a number $t_2$ exists by compactness of $[0,1]$).  Then concatenating the homotopies $(p^0_t)_{t\in [t_0,t_2]}$, $(p^s_{t_2})_{s\in [0,1]}$, and $(p^1_t)_{t\in [t_1,t_2]}$ connects $p^0_{t_0}$ and $p^1_{t_1}$ in $\mathcal{P}^\pi_\epsilon(X,B)$, and we get well-definedness.  

To see that $\psi_{X,\epsilon}$ is a group homomorphism, let $s_1,s_2\in \mathcal{B}(\ell^2)\subseteq \LL(E)$ be a pair of isometries satisfying the Cuntz relation, and used to define the group operations on both $\KKP^\pi(A,B)$ and $KK^\pi_\epsilon(X,B)$, and let $[p],[q]\in \mathcal{P}^\pi(A,B)$.  Then $[p]+[q]$ is the class of $[s_1ps_1^*+s_2qs_2^*]$, and we have that  
$$
\psi_{X,\epsilon}:[s_1ps_1^*+s_2qs_2^*]\mapsto [s_1p_{t_{p+q}}s_1^*+s_1q_{t_{p+q}}s_1^*]
$$ 
where $t_{p+q}:=\tau_{X,\epsilon}(s_1ps_1^*+s_2qs_2^*)+1$.  On the other hand, if we define $t_p:=\tau_{X,\epsilon}(p)+1$ and $t_q:=\tau_{X,\epsilon}(q)$, then 
$$
\psi_{X,\epsilon}[p]+\psi_{X,\epsilon}[q]=[s_1p_{t_p}s_1^*+s_2q_{t_q}s_2^*].
$$
Define $t_{p+q}:=\max\{t_p,t_q\}$, and say without loss of generality that $t_p\geq t_q$.  Then the path $(s_1p_{t_p}s_1^*+s_2q_{t}s_2^*)_{t\in [t_q,t_p]}$ shows that $\psi_{X,\epsilon}([p]+[q])=\psi_{X,\epsilon}[p]+\psi_{X,\epsilon}[q]$ as required.

Compatibility of the maps $\psi_{X,\epsilon}$ with the forget control maps in line \eqref{forget} can be shown via similar arguments; we leave the details to the reader.  The existence of $\psi$ follows from this and the universal property of the inverse limit as in Remark \ref{gen inv lim rem}, part \eqref{inv lim uni}.
\end{proof}

Using the ideas in the previous section, we now get the promised relationship to $KL$.  

\begin{theorem}\label{id with KL}
Let $\pi:A\to \LL(E)$ be a graded, balanced, and strongly absorbing representation on a Hilbert $B$-module.  The homomorphism $\psi$ in Lemma \ref{comp map} is surjective and descends to a homeomorphic isomorphism 
$$
\psi: KL(A,B)\to  \lim_{\leftarrow}KK^\pi_\epsilon(X,B).
$$
\end{theorem}

\begin{proof}
Inspecting Definition \ref{inv lim def}, a neighborhood basis of $0$ in ${\displaystyle \lim_{\leftarrow}KK^\pi_\epsilon(X,B)}$ consists of the sets 
$$
W(X,\epsilon):=\{([p_{Y,\delta}]_{(Y,\delta)\in \mathcal{X}} \mid [p_{X,\epsilon}]=0\}
$$
as $(X,\epsilon)$ varies over the set $\mathcal{X}$ of Definition \ref{seps}.  On the other hand, inspecting Definition \ref{top}, a neighborhood basis of $0$ in $\KKP^\pi(A,B)$ consists precisely of the sets $V(0;X,\epsilon)$ as $(X,\epsilon)$ ranges over $\mathcal{X}$.   Direct checks show that these sets are in bijective correspondence via the map $\psi$.  It follows from this that the map
$$
\psi:\KKP^\pi(A,B)\to \lim_{\leftarrow}KK^\pi_\epsilon(X,B)
$$
of Lemma \ref{comp map} descends to an open and continuous injection 
$$
\psi: \frac{\KKP^\pi(A,B)}{\overline{\{0\}}}\to \lim_{\leftarrow}KK^\pi_\epsilon(X,B).
$$
The left hand side identifies with $KL(A,B)$ by Theorem \ref{kk iso}, Theorem \ref{tops same} and the definition of $KL(A,B)$ as in line \eqref{kl def} above, so we have an open continuous injection 
$$
\psi: KL(A,B)\to \lim_{\leftarrow}KK^\pi_\epsilon(X,B)
$$
To see that $\psi$ is a homeomorphic isomorphism, it remains to show that it is surjective. 

For this, let us choose a cofinal sequence $(X_n,1/n)_{n=1}^\infty$ of $\mathcal{X}$ as in Remark \ref{seps rem} part \eqref{cof seq}, whence there is a canonical isomorphism 
$$
\lim_{\leftarrow}KK^\pi_\epsilon(X,B)=\lim_{\leftarrow} KK^\pi_{1/n}(X_n,B)
$$
as in Remark \ref{gen inv lim rem} part \eqref{cof iso}.  Hence it suffices to prove surjectivity of the induced map
$$
\KKP^\pi(A,B)\to \lim_{\leftarrow} KK^\pi_{1/n}(X_n,B).
$$
As in Definition \ref{inv lim def} above, let $([p^n])_{n=1}^\infty$ be a sequence defining an element of ${\displaystyle \lim_{\leftarrow} KK^\pi_{1/n}(X_n,B)}$ with $p^n\in \mathcal{P}^\pi_{1/n}(X_n,B)$ for each $n$.  As this sequence defines an element of the inverse limit we must have that for each $n$, the forget control map 
$$
KK^\pi_{1/(n+1)}(X_{n+1},B)\to KK^\pi_{1/n}(X_n,B)
$$
sends $[p^{n+1}]$ to $[p^n]$.  This implies that there is a homotopy $(p^{n}_s)_{s\in [0,1]}$ of elements in $\mathcal{P}^\pi_{1/n}(X_n,B)$ with $p^{n}_0=p^n$ and $p^{n}_1=p^{n+1}$.  Define $p:[1,\infty)\to \LL(E)$ by setting $p_t:=p^n_{t-n}$ whenever $t$ is in $[n,n+1]$, and note that $p$ is then an element of $\mathcal{P}^\pi(A,B)$.  

We claim that $\pi[p]=([p^n])_{n=1}^\infty$, which will complete the proof.  Indeed, it suffices to fix $n$ and show that $\psi_{X_n,1/n}[p]=[p^n]$.  We have $\psi_{X_n,1/n}[p]=[p_{t_p}]$, where $t_p:=\tau_{X_n,1/n}(p)$.  By definition of $p$ and of $\tau_{X_n,1/n}$, the interval $I$ with endpoints $n$ and $t_p$ is such that the path $(p_t)_{t\in I}$ lies entirely in $\mathcal{P}^\pi_{1/n}(X_n,B)$.  Hence 
$$
\psi_{X_n,1/n}[p]=[p_{t_p}]=[p_n]=[p^n]
$$
and we are done.
\end{proof}

\section{The closure of zero and ${\displaystyle \lim_{\leftarrow}{}\!^1}$ groups}\label{zero sec}

Throughout this section, $A$ and $B$ refer to separable $C^*$-algebras.  All Hilbert modules are countably generated, and all are over $B$ unless explicitly stated otherwise.  All representations of $A$ are on Hilbert $B$-modules unless explicitly stated otherwise.

Our goal in this section is to concretely identify the closure of zero in the asymptotic topology on $\KKP^\pi(A,B)$ (see Definitions \ref{pp hom} and \ref{top}) in terms of the controlled $KK$-theory groups $KK_\epsilon^\pi(X,B)$ (see Definition \ref{con kk gp}).  The precise statement we are aiming for is that the \emph{${\displaystyle \lim_{\leftarrow}{}\!^1}$ group}
\begin{equation}\label{gen lim 1 good}
 \lim_{\leftarrow}{}\!^1KK_\epsilon^{1\otimes\pi}(X,SB)
\end{equation}
of the inverse system from Definition \ref{good i l} is isomorphic to the closure of zero inside $\KKP^\pi(A,B)$.  This will complete the proof of Theorem \ref{main} from the introduction.

The general definition of ${\displaystyle \lim_{\leftarrow}{}\!^1}$ is as follows: see \cite[pages 1-3]{Jensen:1972wu} for details, and see for example \cite[Section 2.5]{Weibel:1995ty} for general background on the theory of right derived functors being used here.  Let $I$ be a directed set, and let $(G_i)_{i\in I}$ be an inverse system of abelian groups indexed by $I$ in the sense of Definition \ref{inv lim def} (we will not assume the groups are equipped with a topology).  The collection $\text{Ab}^I$ of all such inverse systems can be arranged into an abelian category, and taking inverse limits defines a left exact functor 
$$
\lim_{\leftarrow}:\text{Ab}^I\to \text{Ab}
$$
from $\text{Ab}^I$ to the category $\text{Ab}$ of abelian groups.  One can show that $\text{Ab}^I$ has enough injective objects so that the right derived functors of ${\displaystyle \lim_{\leftarrow}}$ make sense.  Then by definition ${\displaystyle \lim_{\leftarrow}{}\!^1}$ is the first right derived functor of ${\displaystyle \lim_{\leftarrow}}$; in particular ${\displaystyle \lim_{\leftarrow}{}\!^1}$ is a functor taking inverse systems of abelian groups indexed by $I$ to abelian groups.

We gave the general definition of ${\displaystyle \lim_{\leftarrow}{}\!^1}$ above for completeness, but will not need to use it in the proofs below.  There is a more concrete picture that is available when the index set is $\N$ and that is more useful for computations: see for example \cite[Section 3.5]{Weibel:1995ty} for a detailed exposition of the concrete definition below.

\begin{definition}\label{gen lim1}
Let $(G_n)_{n=1}^\infty$ be an inverse system of abelian groups indexed by $\N\setminus \{0\}$ as in Definition \ref{inv lim def} above, and write $\phi:G_{n+1}\to G_{n}$ for the connecting maps.  Define 
\begin{equation}\label{delta def}
\delta:\prod_{n=1}^\infty G_n \to \prod_{n=1}^\infty G_n,\quad (g_n)_{n=1}^\infty\mapsto (g_n-\phi(g_{n+1}))_{n=1}^\infty.
\end{equation}
The \emph{ ${\displaystyle\lim_{\leftarrow}{}\!^1}$ group} of $(G_n)$, denoted ${\displaystyle\lim_{\leftarrow}{}\!^1}G_n$ or  ${\displaystyle\lim_{\leftarrow,\N}{}\!^1}G_n$, is defined to be the cokernel of $\delta$.
\end{definition}

We record some basic facts as a lemma.  For the statement, recall that a subset $J$ of a directed set $I$ is \emph{cofinal} if for all $i\in I$ there is $j\in J$ with $j\geq i$.  

\begin{lemma}\label{cofin nml1}
Let $(G_i)_{i\in I}$ be an inverse system of abelian groups as in Definition \ref{inv lim def}, and let $J\subseteq I$ be a cofinal subset.  Then there is a canonical isomorphism 
$$
\lim_{\leftarrow,I}{}\!^1 G_i\cong \lim_{\leftarrow,J}{}\!^1 G_j.
$$
In particular, if $((X_n,\epsilon_n))_{n=1}^\infty$ is a cofinal subsequence\footnote{Cofinal subsequences of $\mathcal{X}$ always exist by Remark \ref{seps rem}, part \eqref{cof seq}.} of the directed set $\mathcal{X}$ of Definition \ref{seps}, and $\pi:A\to \LL(E)$ is a balanced, graded, infinite multiplicity representation, then there is a canonical isomorphism
$$
 \lim_{\leftarrow,\mathcal{X}}{}\!^1KK_\epsilon^\pi(X,B)\cong  \lim_{\leftarrow,\N}{}\!^1KK_{\epsilon_n}^\pi(X_n,B)
$$
where the left hand side is the ${\displaystyle\lim_{\leftarrow}{}\!^1}$ group over the directed set $\mathcal{X}$ of Definition \ref{seps}, and the right hand side is the ${\displaystyle\lim_{\leftarrow}{}\!^1}$ group over the sequence $((X_n,\epsilon_n))_{n=1}^\infty$, computed using Definition \ref{gen lim1} above.  
\end{lemma}

\begin{proof}
As pointed out on \cite[page 11]{Jensen:1972wu}, the general statement about cofinal subsets is a consequence of \cite[Th\'{e}or\`{e}me 1.9 and Lemme 1.5]{Jensen:1972wu}.  The second part is a special case of the first part, combined with the discussion on \cite[pages 13-14]{Jensen:1972wu} which shows that the general definition of the ${\displaystyle \lim_{\leftarrow}{}\!^1}$ group agrees with the concrete version from Definition \ref{gen lim1} in situations where both make sense.
\end{proof}

We will need an analogue of (a special case of) Lemma \ref{loc func lem 0} in the controlled setting; the proof is very similar, and left to the reader.  For the statement, recall the notation $\mathcal{P}^\pi_\epsilon(X,B)$ from Definition \ref{alm com}, and recall that $SB:=C_0(0,1)\otimes B$ denotes the suspension of $B$.

%\begin{lemma}\label{loc func lem 1}
%Let $\pi:A\to \LL(E)$ be a graded representation on a Hilbert $B$-module.  Let $C=C_0(Y)$ be a separable commutative $C^*$-algebra, and let $C\otimes E$ be equipped with the amplified representation $1\otimes \pi$ of $A$ as in the discussion just above Lemma \ref{loc func lem 0}.  Let $\epsilon>0$, and let $X$ be a finite subset of the unit ball $A_1$ of $A$.  Then there is a natural identification between elements $p$ of $\mathcal{P}^{1\otimes \pi}_\epsilon(X,C\otimes B)$ and parametrized families of self-adjoint contractions $(p^y)_{y\in Y}$ such that the corresponding function $p:Y\to \LL(E)$ has the following properties:
%\begin{enumerate}[(i)]
%\item the function $p-e$ is in $C_0(Y, \K(E))$;
%\item $\|[p,a]\|_{C_b(Y,\LL(E))}<\epsilon$ for all $a\in X$;
%\item $\|a(p^2-p)\|_{C_b(Y,\LL(E))}<\epsilon$ for all $a\in X$.   
%\end{enumerate}
%\end{lemma}

%\begin{proof}
%Analogously to Lemma \ref{loc func lem 0}, the proof rests on the identification $\K(C\otimes E)=C_0(Y,\K(E))$; we leave the details to the reader.
%\end{proof}

\begin{lemma}\label{susp cor}
Let $\pi:A\to \LL(E)$ be a graded representation of $A$ on a Hilbert $B$-module, and let $1\otimes \pi$ be the amplified representation of $A$ on the $SB$-module $C_0(0,1)\otimes E$.  For any finite subset $X$ of $A_1$ and $\epsilon>0$, elements of $\mathcal{P}^{1\otimes \pi}_\epsilon(X,SB)$ can be identified with norm continuous functions
$$
p:[0,1]\to \LL(E), \quad t\mapsto p_t
$$ 
such that:
\begin{enumerate}[(i)]
\item $p_0=p_1=e$;
\item $p_t-e\in \K(E)$ for all $t\in [0,1]$;
\item $\|a(p_t^2-p_t)\|<\epsilon$ and $\|[p_t,a]\|<\epsilon$ for all $a\in X$ and all $t\in [0,1]$.   \qed
\end{enumerate}
\end{lemma}

\begin{proposition}\label{psi lem}
Let $\pi:A\to \LL(E)$ be a graded, balanced, and infinite multiplicity representation, and let $\big((X_n,\epsilon_n)\big)_{n=1}^\infty$ be a cofinal subsequence of the directed set $\mathcal{X}$ of Definition \ref{seps}.  

For each $n$, let $p^n$ be an element of the space $\mathcal{P}_{\epsilon_n}^\pi(X_n,SB)$ of Definition \ref{alm com}.  Use the identification in Lemma \ref{susp cor} to consider each $p^n$ as a function $p^n:[0,1]\to \LL(E)$, and define $p:[1,\infty)\to \LL(E)$ by setting $p_t:=p^n_{t-n}$ whenever $t\in [n,n+1]$.  Then $p$ is in $\mathcal{P}^\pi(A,B)$, and the formula
$$
\psi:\prod_n KK^{1\otimes\pi}_{\epsilon_n}(X_n,SB)\to \KKP^\pi(A,B),\quad ([p^n])_{n=1}^\infty\mapsto [p]
$$
gives a well-defined homomorphism.  

Moreover, the homomorphism $\psi$ takes image in the closure $\overline{\{0\}}$ of the zero element for the asymptotic topology (see Definition \ref{top}), and descends to a well-defined homomorphism 
$$
\psi: \lim_{\leftarrow}{}\!^1 KK_{\epsilon_n}^{1\otimes\pi}(X_n,SB)\to \KKP^\pi(A,B)
$$ 
on the $\lim^1$-group.
\end{proposition}

\begin{proof}
We leave the direct check that $p$ is an element of $\mathcal{P}^\pi(A,B)$ to the reader: for this purpose note that $\epsilon_n\to 0$ as $n\to\infty$, and also that for each $n$, $p^n(0)=p^n(1)=e$.  

To see that $\psi$ is well-defined on the product $\prod_n KK^{1\otimes \pi}_{\epsilon_n}(X_n,SB)$, let $([p^{n,(0)}])_{n=1}^\infty$ and $([p^{n,(1)}])_{n=1}^\infty$ be sequences in $KK_{\epsilon_n}^{1\otimes\pi}(X_n,SB)$ representing the same class in the product $\prod_n KK^{1\otimes\pi}_{\epsilon_n}(X_n,SB)$, and with images $[p^{(0)}]$ and $[p^{(1)}]$ in $\KKP^\pi(A,B)$.  Then for each $n$ there is a homotopy $(p^{n,(s)})_{s\in [0,1]}$ between $p^{n,(0)}$ and $p^{n,(1)}$.  Using the identification in Lemma \ref{loc func lem 0}, define a new function $p:[1,\infty)\to \LL(C[0,1]\otimes E)$ by $p^{(s)}_t:=p^{n,s}_{t-n}$ for $t\in [n,n+1]$.  Then direct checks using the conditions in Lemma \ref{loc func lem 0} show that $(p^{(s)})_{s\in [0,1]}$ is a homotopy between $p^{(0)}$ and $p^{(1)}$, whence $[p^{(0)}]=[p^{(1)}]$ and we have well-definedness.  

The fact that $\psi$ is a homomorphism follows as we may assume the group operations are all defined using the same pair of isometries $(s_1,s_2)$ satisfying the Cuntz relation.

We now show that the image of $\psi$ is contained in $\overline {\{0\}}$.  Using the definition of the asymptotic topology (see Definition \ref{top} above) we must show that if $[p]$ is in the image, then for every finite subset $X\subseteq A_1$ and $\epsilon>0$ there is $t\geq \tau_{X,\epsilon}(p)$ (see Definition \ref{tause}) and a homotopy passing through $\mathcal{P}^\pi_\epsilon(X,B)$ connecting $p_{t}$ to $e$.  Indeed, by construction of $p$, there is a sequence $(t_n)$ tending to infinity such that $p_{t_n}=e$ for all $n$.   Choose then $n$ such that $t_n\geq \tau_{X,\epsilon}(p)$, and set $t=t_n$; the constant homotopy connecting $p_{t}$ to $e$ works.

For the statement that $\psi$ descends to the $\lim^1$ group, we must show that if $([p^n])_{n=1}^\infty$ is a sequence in $\prod_n KK^{1\otimes\pi}_{\epsilon_n}(X_n,SB)$, then the image of $([p^n])$ is the same as that of the sequence $([p^{n+1}])_{n=1}^\infty$.  Indeed, say the image of the former is $p$ and the image of the latter is $q$.  Then by construction we have that $q_t=p_{t-1}$ for all $t\geq 2$.  The path $(p^s)_{s\in [0,1]}$ defined by $p^s_t:=p_{t-s}$ is a homotopy between $p$ and $q$, so we are done.
\end{proof}

Our goal in the remainder of this section is to show that $\psi$ as in Proposition \ref{psi lem} actually defines an isomorphism between ${\displaystyle \lim_{\leftarrow}{}\!^1 KK_{\epsilon_n}^{1\otimes\pi}(X_n,SB)}$ and the closure of $0$ in $\KKP^\pi(A,B)$.  

The next definition and lemma give an alternative description of the group operation on $KK^{1\otimes \pi}_\epsilon(X,SB)$.

\begin{definition}\label{conc def}
Let $\pi:A\to \LL(E)$ be a graded representation of $A$, let $X\subseteq A_1$ be finite, and let $\epsilon>0$.  Let $p,q\in \mathcal{P}^{1\otimes\pi}_\epsilon(X,SB)$ be represented by paths as in Corollary \ref{susp cor}.  Define their \emph{concatenation}, denoted $p\cdot q$, to be the path that follows $p$ then $q$: precisely
$$
(p\cdot q)_t:=\left\{\begin{array}{ll} p_{2t} & 0\leq t\leq 1/2 \\ q_{2t-1} & 1/2<t\leq 1\end{array}\right..
$$
\end{definition}

\begin{lemma}\label{con gp op}
Let $\pi:A\to \LL(E)$ be a graded, balanced, and infinite multiplicity representation of $A$, let $X$ be a finite subset of $A_1$, $\epsilon>0$, and let $SB$ be the suspension of $B$.  Then for any $[p],[q]\in KK^{1\otimes\pi}_\epsilon (X,SB)$ we have that $[p]+[q]=[p\cdot q]$.  Moreover, $-[p]$ is represented by the path $\overline{p}$ that traverses $p$ in the opposite direction.
\end{lemma}

\begin{proof}
Up to homotopy, we may assume that $p_t=e$ for all $t\in [1/3,1]$, and that $q_t=e$ for all $t\in [0,2/3]$.  The sum $[p]+[q]$ is then represented by a function of the form
$$
(s_1ps_1^*+s_2qs_2^*)_t=\left\{\begin{array}{ll} s_1p_ts_1^*+s_2es_2^* & t\in [0,1/3] \\ e & t\in [1/3,2/3] \\ s_1es_1^*+s_2q_ts_2^* & t\in [2/3,1] \end{array}\right..
$$
Let $v=s_1s_2^*+s_2s_1^*$, which is a unitary in $\mathcal{B}(\ell^2)$, which we identity with a unital $C^*$-subalgebra of $\LL(E)$ as in line \eqref{can incl} from Lemma \ref{reps}.  As the unitary group of $\mathcal{B}(\ell^2)$ is connected, there is a path $u=(u_t)_{t\in [0,1]}$ of unitaries in $\mathcal{B}(\ell^2)$ such that $u_t=1$ for $t\leq 1/3$ and $u_t=v$ for $t\geq 2/3$.  We may consider $u$ as an element of the unitary group of $\LL(C_0(0,1)\otimes \ell^2)$, which we identify with a unital $C^*$-subalgebra of $\LL(C_0(0,1)\otimes E)$ compatibly with the inclusion $\mathcal{B}(\ell^2)\subseteq \LL(E)$ from line \eqref{can incl} again.  Using that $u$ commutes with $e$, we have then that  
\begin{equation}\label{conj el}
(u(s_1ps_1^*+s_2qs_2^*)u^*)_t=\left\{\begin{array}{ll} s_1p_ts_1^*+s_2es_2^* & t\in [0,1/3] \\ e & t\in [1/3,2/3] \\ s_1q_ts_1^*+s_2es_2^* & t\in [2/3,1] \end{array}\right..
\end{equation}
On the other hand, the unitary group of $\LL(C_0(0,1)\otimes \ell^2)$ is connected (even contractible) by \cite[Theorem on page 433]{Cuntz:1987ly}, so we may connect $u$ to the identity via some norm continuous path in this unitary group $\LL(C_0(0,1)\otimes \ell^2)$.  As the unitary group of $\LL(C_0(0,1)\otimes \ell^2)$ commutes with both $e$ and $A$, this gives a homotopy showing that $u(s_1ps_1^*+s_2qs_2^*)u^*$ defines the same element of $KK^{1\otimes\pi}_\epsilon (X,SB)$ as $s_1ps_1^*+s_2qs_2^*$.  From the description in line \eqref{conj el}, we have also that $u(s_1ps_1^*+s_2qs_2^*)u^*$ and $s_1(p\cdot q)s_1^*+s_2es_2^*$ define the same element of $KK^{1\otimes\pi}_\epsilon (X,SB)$.  The latter element is homotopic to $p\cdot q$ by Lemma \ref{id lem 2}, so we are done with the proof that $[p]+[q]=[p\cdot q]$.

The fact that $-[p]=[\overline{p}]$ is a consequence of the first part: indeed, $[p]+[\overline{p}]=[p\cdot \overline{p}]$, and $p\cdot \overline{p}$ is easily seen to be homotopic to the constant path $e$, which represents the identity in $KK_\epsilon^{1\otimes\pi}(X,SB)$ by Corollary \ref{kkeps mon}.
\end{proof}

We need one more technical lemma before we can establish the main result.  This is elementary, but a little fiddly: unfortunately, we could not find a more conceptual proof of what we need.  For the proof of the lemma, let us explicitly adopt the convention that $\N$ does not contain zero.

\begin{lemma}\label{abs l1}
Let $(G_n)_{n=1}^\infty$ be an inverse system of abelian groups indexed by $\N$.  For notational simplicity, let us write $\phi:G_{n+1}\to G_n$ for the connecting map, and for $m\geq n$ write $\phi^{m-n}:G_m\to G_n$ for the connecting map\footnote{This is consistent with standard notation for composition of functions, which is how we will use it: in particular, $\phi^1=\phi$, and $\phi^0$ is the identity map.}.  Let $a=(a_n)\in \prod_{n=1}^\infty G_n$, and assume that there exists a sequence $(m_n)_{n=1}^\infty$ in $\N$ and elements $b_n\in G_{m_n}$ with the following properties:
\begin{enumerate}[(i)]
\item for each $n$, $m_n\in \{1,...,n\}$;
\item $m_1\leq m_2\leq \cdots$;
\item $m_n\to\infty$ as $n\to\infty$;
\item for each $n$, there exists $b_n\in G_{m_n}$ such that 
\begin{equation}\label{a and b}
\phi^{n-m_n}(a_n)=b_n-\phi^{m_{n+1}-m_n}(b_{n+1}).
\end{equation}
\end{enumerate}
Then $(a_n)$ is in the image of the map $\delta$ from line \eqref{delta def}.
\end{lemma}

\begin{proof}
Given $g\in G_n$, we abuse notation slightly by also writing $g$ for the element of $\prod G_m$ with $g$ in the $n^\text{th}$ slot, and zero elsewhere.  For an element $a=(a_n)\in \prod G_n$, the \emph{support} of $a$ is the set $\{n\in \N\mid a_n\neq 0\}$.  For integers $k,l\in \N$, we write $[k,l]$ for the set $\{k,k+1,...,l\}$ if $k\leq l$, and with the convention that $[k,l]=\varnothing$ if $k>l$.  We also adopt the convention that a sum indexed by the empty set is zero.  We say that a possibly infinite formal sum 
$$
\sum_{i\in I} a^{(i)}
$$
of elements of $\prod G_n$ is \emph{locally finite} if each $n\in \N$ appears in the support of only finitely many $a^{(i)}$s, and we note that locally finite sums give well-defined elements of $\prod G_n$.

Write $\mu:\N\to \N$ for the map $n\mapsto m_n$.  The properties of $m_n$ guarantee that each $\mu^{-1}(m)$ is a (possibly empty) finite interval of the form $[\alpha(m),\beta(m)]$ with $\alpha(m)\geq m$ as in the schematic below.
\begin{center}
\includegraphics[width=10cm]{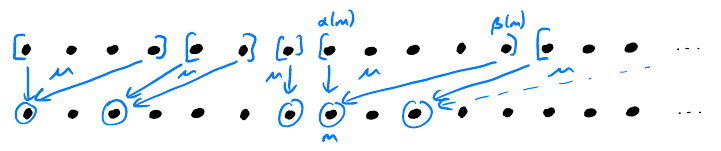}
\end{center}

For each $m\in \N$, define 
$$
c^{(m)}=\sum_{n=\alpha(m)}^{\beta(m)}\phi^{n-m}(a_n)
$$
In words, $c^{(m)}$ takes all the elements $a_n$ with $m_n=m$, moves them all down to be supported at $m$, and sums them.  Note that each $c^{(m)}$ is supported on the singleton $\{m\}$, or has empty support if $\mu^{-1}(m)$ is empty.  Hence the sum $c:=\sum_{m\in \N} c^{(m)}$ is locally finite, so $c$ is a well-defined element of $\prod G_n$.

We first claim that the element $a-c$ is contained in the image of $\delta$.  To see this, for each $m\in \mu(\N)$, define $d^{(m)}$ by the conditions below.  First, if $\mu^{-1}(m)=\varnothing$, define $d^{(m)}=0$.  Second, if $m=\alpha(m)=\beta(m)$, define $d^{(m)}=0$.  Third, in all other cases define
$$
d^{(m)}_n =\left\{\begin{array}{ll} {\displaystyle\sum_{i=n}^{\beta(m)} \phi^{i-n}(a_i)}, & k\in [\alpha(m)+1,\beta(m)] \\
{\displaystyle\sum_{i=\alpha(m)}^{\beta(m)} \phi^{i-n}(a_i)}, & k\in [m+1,\alpha(m)] \\
0, & \text{otherwise}\end{array}\right.
$$
The schematic below pictures $d^{(m)}$ for $\mu^{-1}(m)\neq \varnothing$ and $\beta(m)>\alpha(m)>m$:
\begin{center}
\includegraphics[width=10cm]{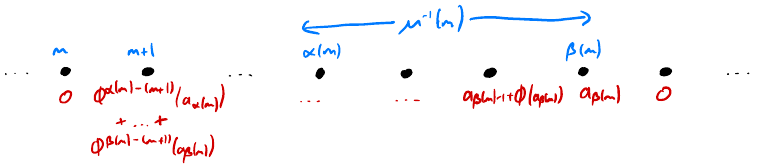}
\end{center}
One checks that for each $m$,
\begin{equation}\label{delta d}
\delta(d^{(m)})=\Bigg(\sum_{k=\alpha(m)}^{\beta(m)}a_k\Bigg)-c^{(m)}.
\end{equation}
As $m_n\to\infty$, the sum $d:=\sum_{m\in \N} d^{(m)}$ is locally finite and makes sense as an element of $\prod G_n$.  The computation in line \eqref{delta d} above shows that one has $\delta(d)=a-c$, so $a-c$ is in the image of $\delta$ as claimed.

To complete the proof, it therefore suffices to show that $c$ is in the image of $\delta$.  Now, for each $m$
\begin{align*}
c^{(m)}=\sum_{n\in \mu^{-1}(m)}\phi^{n-m}(a_n)=\sum_{\{n\mid m_n=m\}} \phi^{n-m_n}(a_n).
\end{align*}
Using the assumption in line \eqref{a and b}, we thus see that
\begin{equation}\label{cm sum}
c^{(m)}=\sum_{n=\alpha(m)}^{\beta(m)}(b_n-\phi^{m_{n+1}-m_n}(b_{n+1})).
\end{equation}
For $n\in [\alpha(m),\beta(m)-1]$, $m_{n+1}=m_n$, and so the sum in line \eqref{cm sum} above telescopes, leaving just
\begin{equation}\label{c m form}
c^{(m)}=b_{\alpha(m)}-\phi^{m_{\beta(m)+1}-m_{\beta(m)}}(b_{\beta(m)+1}).
\end{equation}
Let us write the elements of $\mu(\N)$ as $k_1<k_2<k_3<\cdots$.  For $m\in \N$, let $k(m)$ be the smallest $k_n$ such that $m\leq k_n$, and define $e\in \prod G_m$ by
$$
e_m:=\phi^{k(m)-m}(b_{\alpha(k(m))})
$$
(and $e_m=0$ if $m<k_1$) as in the schematic below
\begin{center}
\includegraphics[width=10cm]{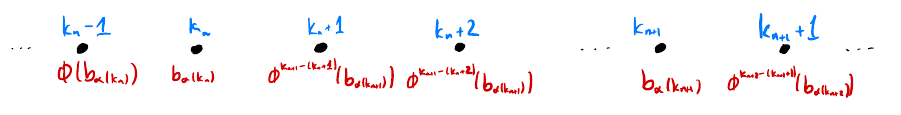}.
\end{center}
Using the formula in line \eqref{c m form} and the fact that $\beta(k_n)+1=\alpha(k_{n+1})$ one checks that $\delta(e)=c$, so we are done.
\end{proof}

We are now ready for the main result of this section.  As already commented above, it completes the proof of Theorem \ref{main}.

\begin{theorem}\label{c0l1}
Let $\pi:A\to \LL(E)$ be a graded, balanced, and strongly absorbing representation, and let $\big((X_n,\epsilon_n)\big)_{n=1}^\infty$ be a cofinal subsequence of the directed set $\mathcal{X}$ of Definition \ref{seps}.   Then the map
$$
\psi:\lim_{\leftarrow}{}\!^1 KK^{1\otimes\pi}_{\epsilon_n}(X_n,SB)\to \KKP^\pi(A,B)
$$
from Proposition \ref{psi lem} is an isomorphism onto the closure of zero in $\KKP^\pi(A,B)$.  In particular, we have an isomorphism 
$$
 \lim_{\leftarrow}{}\!^1KK_\epsilon^\pi(X,B)\cong \overline{\{0\}}
$$
where the ${\displaystyle \lim_{\leftarrow}{}\!^1}$ group on the left is taken over the directed set $\mathcal{X}$ of Definition \ref{seps}.
\end{theorem}

\begin{proof}
The last statement follows from Lemma \ref{cofin nml1} and the statement that $\psi$ is an isomorphism.  We focus then on the statement that $\psi$ is an isomorphism. 

To see that the map $\psi$ is surjective, let $p\in \mathcal{P}^\pi(A,B)$ (see Definition \ref{proj path}) be an element so that $[p]$ is in the closure of zero.  Using the description of neighbourhood bases from Lemma \ref{fc lem}, we may find an increasing sequence $(t_n)$ in $[1,\infty)$ such that $t_n\to\infty$, such that $t_n\geq \tau_{X_n,\epsilon_n}(p)$ for all $n$, and such that for each $n$ there is a homotopy $(q^{n}_s)_{s\in [0,1]}$ such that $q^n_0=p_{t_n}$ and $q^n_1=e$, and such that $q^n_s$ is in $\mathcal{P}^\pi_{\epsilon_n}(X_n,B)$ for all $s$.  For each $n$, build a path $r^n:[0,1]\to \LL(E)$  by concatenating the paths $(q^n_{1-s})_{s\in [0,1]}$, $(p_t)_{t\in [t_n,t_{n+1}]}$, and $(q^{n+1}_s)_{s\in [0,1]}$, and reparametrizing to get the domain equal to $[0,1]$.  Note that the path $(r^n_s)_{s\in [0,1]}$ starts and ends at $e$, and has image contained in $\mathcal{P}^\pi_{\epsilon_n}(X_n,B)$.  One checks directly that $r^n$ lies in $\mathcal{P}^{1\otimes\pi}_{\epsilon_n}(X_n,SB)$ using the conditions in Corollary \ref{susp cor}, and thus we get a class ${\displaystyle ([r^n])\in \lim_{\leftarrow}{}\!^1 KK_{\epsilon_n}^\pi(X_n,SB)}$.  We claim the image of $([r^n])$ in $\KKP^\pi(A,B)$ is $[p]$.

Indeed, up to reparametrizations (which do not affect the resulting class in $\KKP^\pi(A,B)$), the image of $([r^n])$ is represented by concatenating the paths
\begin{align*}
(q^1_{1-s})_{s\in [0,1]}, & ~(p_t)_{t\in [t_1,t_2]},~(q^2_{s})_{s\in [0,1]},~(q^2_{1-s})_{s\in [0,1]},~(p_t)_{t\in [t_2,t_3]}, \\ \quad &(q^3_{s})_{s\in [0,1]},(q^3_{1-s})_{s\in [0,1]},~(p_t)_{t\in [t_3,t_4]},~...
\end{align*}
As each pair $(q^n_s)_{s\in [0,1]},~(q^n_{1-s})_{s\in [0,1]}$ consists of the same path traversed in opposite directions, a homotopy removes all these pairs, so we are left with the concatenation of the paths
$$
(q^1_{1-s})_{s\in [0,1]},~(p_t)_{t\in [t_1,t_2]},~(p_t)_{t\in [t_2,t_3]},~(p_t)_{t\in [t_3,t_4]},~...
$$
or in other words of $(q^1_{1-s})_{s\in [0,1]}$ and $(p_t)_{t\geq t_1}$.  As any element $q\in \mathcal{P}^\pi(A,B)$ is homotopic to the path defined by $t\mapsto q_{t+L}$ for any fixed $L>0$, this path is homotopic to the original $p$ and we  are done with surjectivity.

For injectivity, let $([p^n])_{n=1}^\infty$ be a sequence in $\prod_n KK_{\epsilon_n}^{1\otimes\pi}(X_n,SB)$ that maps to zero in $\KKP^\pi(A,B)$, so there is a homotopy $(p^s)_{s\in [0,1]}$ connecting the resulting image $p$ to $e$.  Here $p$ is the result of concatenating the functions $p^n:[0,1]\to \LL(E)$, so for $t\in [n,n+1]$, $p_t=p^n_{t-n}$.   For each $n$, let $(q^n)_{s\in [0,1]}$ be the path defined by $q^n_s:=p^s_{n}$, which defines an element of $\mathcal{P}^{1\otimes\pi}_{\epsilon}(X,SB)$ for some $\epsilon$ and $X$.  Schematically, we have the following picture:\\
\includegraphics[width=15cm]{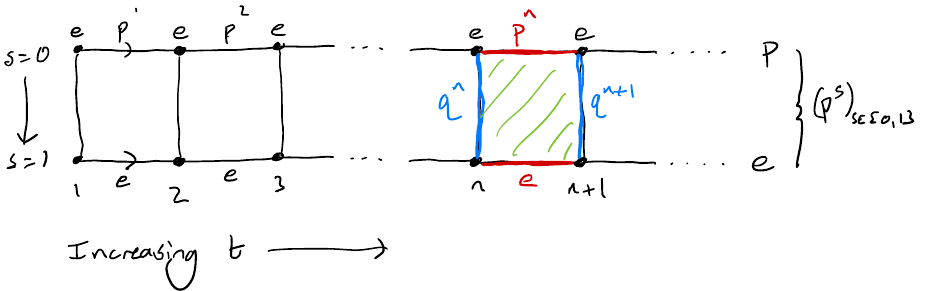}\\
For each $n$, let $m_n\in \{1,...,n\}$ be as large as possible subject to the condition that the elements $(p^s_t)_{s\in [0,1],t\in [n,\infty)}$ are all in $\mathcal{P}^\pi_{\epsilon_{m(n)}}(X_{m_n},B)$.  Note that $m_1\leq m_2\leq \cdots$, and that $m_n\to\infty$ as $n\to\infty$ by definition of a homotopy.  Note moreover that $q^n_s$ is in $\mathcal{P}^\pi_{\epsilon_{m_n}}(X_{m_n},B)$ for all $n$.   

Now, for each $n$, consider the element $-[q^n]+[p^n]+[q^{n+1}]$, which is in $KK^\pi_{\epsilon_{m_n}}(X_{m_n},SB)$ by choice of $m_n$.  This element is represented by the concatenation $\overline{q^n}\cdot p^n\cdot q^{n+1}$ by Lemma \ref{con gp op}, so it forms three sides of the `square' $(p^s_t)_{s\in [0,1],t\in [n,n+1]}$ (pictured as the green square in the diagram above).  The fourth side is the constant function with value $e$, so $-[q^n]+[p^n]+[q^{n+1}]=[e]$ in $KK^\pi_{\epsilon_{m_n}}(X_{m_n},SB)$.  Moreover, $[e]=0$ by Corollary \ref{kkeps mon}, so 
\begin{equation}\label{square rel}
[p^n]=[q^n]-[q^{n+1}] \quad \text{in} \quad KK_{\epsilon_{m_n}}(X_{m_n},SB)
\end{equation}
for all $n$; here we abuse notation slightly, and use the symbol $[p^n]$ for both the original element of $KK_{\epsilon_{n}}(X_{n},SB)$ and its image under the forgetful map 
$$
KK_{\epsilon_{n}}(X_{n},SB)\to KK_{\epsilon_{m_n}}(X_{m_n},SB)
$$
and similarly for the $q$s.  

We are now in the situation of the purely algebraic Lemma \ref{abs l1}, with $([p^n])$ playing the role of $a$, and the $[q^n]$s playing the role of the $b_n$s.  Hence $([p^{(n)}])$ is in the image of the map $\delta$ from line \eqref{delta def}, which is exactly what we wanted to show.
\end{proof}

We conclude this section with a corollary on the general structure of the $KK$ groups; this is not connected to the rest of the paper, but seems of interest in itself.  We thank Claude Schochet for pointing it out to us.

\begin{corollary}
For any separable $C^*$-algebras $A$ and $B$, the closure of $\{0\}$ in $KK(A,B)$ is either $\{0\}$ or uncountable.
\end{corollary}

\begin{proof}
Let $\pi$ be a unitally strongly absorbing representation of $A$.   Standard separability arguments show that for each $\epsilon>0$ and finite $X\subseteq A_1$, the group $KK_\epsilon^{\pi}(X,SB)$ is countable.  It follows from an argument of Gray \cite[page 242]{Gray:1966tq} that a $\lim^1$ group associated to a sequence of countable groups is either zero or uncountable.  Hence ${\displaystyle \lim_{\leftarrow}{}\!^1 KK^{\pi}_\epsilon(X,SB)}$ is either zero or uncountable (here we also use Lemma \ref{cofin nml1} to compute this group using a cofinal sequence in $\mathcal{X}$).  The corollary now follows from Theorem \ref{c0l1}.
\end{proof}

\appendix

\section{Alternative cycles for controlled $KK$-theory}\label{alt pic}

Throughout this appendix, $A$ and $B$ refer to separable $C^*$-algebras.  All Hilbert modules are countably generated, and all are over $B$ unless explicitly stated otherwise.  All representations of $A$ are on Hilbert $B$-modules unless explicitly stated otherwise.

In this appendix, we discuss some technical variants of the groups $KK^\pi_\epsilon(X,B)$ that will be useful in the sequel to this work.

\subsection{Controlled $KK$-groups in the unital case}

In this subsection, we specialize to the unital case and give a picture of the controlled $KK$-groups in this case.  The basic point is that in this case one can use honest projections to define these groups rather than just elements satisfying $\|a(p^2-p)\|<\epsilon$ for suitable $a\in A$ and $\epsilon>0$ as in Definition \ref{alm com}.

\begin{definition}\label{kku}
Let $A$ and $B$ be separable $C^*$-algebras, and let $\pi:A\to \LL(E)$ be a graded representation of $A$ with associated neutral projection $e$ as in Definition \ref{apt}.  Let $X$ be a finite subset of the unit ball $A_1$ of $A$, and let $\epsilon>0$.  Define $\mathcal{P}^{\pi,u}_{\epsilon}(X,B)$\footnote{``$u$'' is for ``unital''.} to be the set of projections in $\LL(E)$ satisfying the following conditions:
\begin{enumerate}[(i)]
\item $p-e$ is in $\K(E)$
\item $\|[p,a]\|<\epsilon$ for all $a\in X$.
\end{enumerate}
Equip $\mathcal{P}^{\pi,u}_{\epsilon}(X,B)$ with the norm topology it inherits from $\LL(E)$, and define $KK^{\pi,u}_\epsilon(X,B):=\pi_0(\mathcal{P}^{\pi,u}_{\epsilon}(X,B))$, i.e.\ $KK^{\pi,u}_\epsilon(X,B)$ is the set of path components of $\mathcal{P}^{\pi,u}_{\epsilon}(X,B)$.
\end{definition}

\begin{definition}\label{kku op}
Let $\pi:A\to \LL(E)$ be a graded, balanced, and infinite multiplicity representation of $A$, and let $KK^{\pi,u}_\epsilon(X,B)$ be as in Definition \ref{kku}.  Let $s_1,s_2\in \mathcal{B}(\ell^2)$ be a pair of isometries satisfying the Cuntz relation $s_1s_1^*+s_2s_2^*=1$, considered as elements of $\LL(E)$ via the inclusion $\mathcal{B}(\ell^2)\subseteq \LL(E)$ of Lemma \ref{reps}.

Define a binary operation on $KK^{\pi,u}_\epsilon(X,B)$ by 
$$
[p]+[q]:=[s_1ps_1^*+s_2qs_2^*]
$$
(it is clear that this definition respects path components, so really does define an operation on $KK^{\pi,u}_\epsilon(X,B)$).
\end{definition}

To show that $KK^{\pi,u}_\epsilon(X,B)$ is a group, we will need an analog of Lemma \ref{id lem 2}. Write ``$\sim$'' for the equivalence relation 

\begin{lemma}\label{id lem 3}
Fix notation as in Definition \ref{kku op}.  Let $e\in \LL(E)$ be the neutral projection, let $p$ be an element of $\mathcal{P}^{\pi,u}_{\epsilon}(X,B)$, and let $v$ be an isometry in the canonical copy of $\mathcal{B}(\ell^2)\subseteq \LL(E)$ from Lemma \ref{reps}.  Then the formula 
$$
vpv^*+(1-vv^*)e
$$
defines an element of $\mathcal{P}^{\pi,u}_{\epsilon}(X,B)$ that is in the same path component as $p$. 
\end{lemma}

\begin{proof}
The proof is essentially the same as that of Lemma \ref{id lem 2}, so we just give a brief sketch, pointing out differences where necessary.  As in the proof of Lemma \ref{id lem 2}, we fix $\delta>0$, and choose an infinite rank projection $r\in \mathcal{B}(\ell^2)$ such that $\|(1-r)(p-e)\|<\delta$ just as in that proof.  Let $\chi:\R\to\{0,1\}$ be the characteristic function of $(1/2,\infty)$, and define $q:=\chi(rpr+(1-r)e)$, which is an element of $\mathcal{P}^{\pi,u}_{\epsilon}(X,B)$ for suitably small $\delta$ by the computations in the proof of Lemma \ref{id lem 2} .  Moreover, for $\delta$ suitably small, the homotopy  
\begin{equation}\label{p q hom}
[0,1]\to\LL(E),\quad s\mapsto \chi(sp+(1-s)q)
\end{equation}
shows that $p$ and $q$ are in the same path component of $\mathcal{P}^{\pi,u}_{\epsilon}(X,B)$ (here we use that there is some $\gamma=\gamma(\delta)$ such that $\gamma\to 0$ as $\delta\to 0$, and such that $\|\chi(sp+(1-s)q)-p\|<\gamma$ for all $s$).  The proof is now finished analogously to that of Lemma \ref{id lem 2} by considering the element $u:=vr+w^*\in \mathcal{B}(\ell^2)\subseteq \LL(E)$ defined just as in that proof, using the element $q$ above in place of the element $q$ from the proof of Lemma \ref{id lem 2}, and using the homotopy in line \eqref{p q hom} where the homotopy $s\mapsto sp+(1-q)$ is used in the proof of Lemma \ref{id lem 2}.
\end{proof}

\begin{lemma}\label{kku group}
Fix notation as in Definition \ref{kku op}.  Then $KK^{\pi,u}_\epsilon(X,B)$ is an abelian group, and does not depend on the choice of Cuntz isometries $s_1$ and $s_2$.
\end{lemma}

\begin{proof}
The fact that $KK^{\pi,u}_\epsilon(X,B)$ is an abelian semigroup with operation not depending on the choice of $s_1,s_2$ proceeds in exactly the same way as Lemma \ref{group lem}.  The fact that it is a monoid with identity element $[e]$ follows directly from Lemma \ref{id lem 3} just as in Corollary \ref{kkeps mon}.  The proof that inverses exist carries over essentially verbatim from the proof of Proposition \ref{kkeps gp} (with slight simplifications, as estimates of the form ``$\|a(p^2-p)\|<\epsilon$'' no longer need to be checked).
\end{proof}

Let now $A$ be a unital $C^*$-algebra, and let $\pi:A\to \LL(E)$ be a representation of $A$.  We write $\pi_1$ for the corestriction of the representation to a representation on $\pi(1_A)\cdot E$.  Note that if $\pi$ is graded, balanced, and infinite multiplicity (see Definition \ref{apt}), then $\pi_1$ is too.  Consider the collection of all pairs $(X,\epsilon)$, where $X$ is a finite subset of the unit ball $A_1$ of $A$ and $\epsilon>0$, made into a directed set as in Definition \ref{seps}.  Our goal in the rest of this section is to show that there are isomorphisms 
$$
KL(A,B)\to \lim_{\leftarrow} KK^{\pi_1,u}_\epsilon(X,B).
$$ 
and 
$$
\lim_{\leftarrow}{}\!^1 KK^{1\otimes \pi_1,u}_\epsilon(X,SB)\to \overline{\{0\}},
$$
where $\overline{\{0\}}$ is the closure of $KK(A,B)$.  The proof proceeds via the construction of certain intertwining maps.

\begin{definition}\label{u to nu}
Fix notation as in Definition \ref{kku op}.   Provisionally define 
$$
\phi : \mathcal{P}^{\pi_1,u}_\epsilon(X,B)\to \mathcal{P}^\pi_\epsilon(X,B), \quad p\mapsto p+(1-1_A)e.
$$
\end{definition}

\begin{lemma}\label{phi hom}
The map $\phi$ from Definition \ref{u to nu} is well-defined, and descends to a homomorphism 
$$
\phi_*:KK^{\pi_1,u}_\epsilon(X,B)\to KK^\pi_{\epsilon}(X,B).
$$
\end{lemma}

\begin{proof}
It is straightforward to see that $\phi$ is a well-defined map that takes homotopies to homotopies and so descends to a well-defined set map $\phi_*:KK^{\pi_1,u}_\epsilon(X,B)\to KK^\pi_{\epsilon}(X,B)$.  Let $s_1,s_2$ be Cuntz isometries inducing the group operation, and define $t_1:=1_As_1$ and $t_2:=1_As_2$, which are a pair of Cuntz isometries in $\LL(1_AE)$ which we may use to define the group operation on $KK^{\pi_1,u}_\epsilon(X,B)$.  We compute that for $p,q\in \mathcal{P}^{\pi_1,u}_\epsilon(X,B)$
$$
t_1pt_1^*+t_2qt_2^*+(1-1_A)e = s_1(p+(1-1_A)e)s_1^*+s_2(q+(1-1_A)e)s_2^*
$$
which implies that $\phi_*([p]+[q])=\phi_*[p]+\phi_*[q]$ as claimed. 
\end{proof}

\begin{definition}\label{nu to u}
Fix notation as in Definition \ref{kku op}.  Assume moreover that $\epsilon<1/8$ and that $X$ contains the unit of $A$.  Let $\chi$ be the characteristic function of $(1/2,\infty)$.  Provisionally define 
$$
\psi : \mathcal{P}^{\pi}_\epsilon(X,B)\to \mathcal{P}^{\pi_1,u}_{5\sqrt{\epsilon}}(X,B), \quad p\mapsto \chi(1_Ap1_A).
$$
\end{definition}

\begin{lemma}\label{psi hom}
The map $\psi$ from Definition \ref{nu to u} is well-defined and descends to a well-defined homomorphism 
$$
\psi_*:KK^{\pi}_\epsilon(X,B)\to KK^{\pi_1,u}_{5\sqrt{\epsilon}}(X,B).
$$
\end{lemma}

\begin{proof}
First, we check that $\psi$ is well-defined, and takes image where we claim.  Let $p$ be an element of $\mathcal{P}^{\pi}_\epsilon(X,B)$.  As we are assuming that $1_A$ is in $X$, we have that 
\begin{equation}\label{[p,1]}
\|[p,1_A]\|<\epsilon.
\end{equation}
Hence
\begin{align*}
\|(1_Ap1_A)^2-(1_Ap1_A)\| & \leq  \|1_Ap1_Ap-1_Ap\|\|1_A\| \\ & \leq \|1_A\|\|[1_A,p]\|\|p\|+\|1_A(p^2-p)\| \\ & < 2\epsilon.
\end{align*}
The polynomial spectral mapping theorem thus implies that the spectrum of $1_Ap1_A$ is contained in the $\sqrt{2\epsilon}$-neighbourhood of $\{0,1\}$ in $\R$.  As $\epsilon<1/8$, we have that $\sqrt{2\epsilon}<1/2$ and so the characteristic function $\chi$ of $(1/2,\infty)$ is continuous on the spectrum of $1_Ap1_A$.  Hence $\chi(1_Ap1_A)$ makes sense by the continuous functional calculus and moreover  
\begin{equation}\label{1p1}
\|1_Ap1_A-\chi(1_Ap1_A)\|<\sqrt{2\epsilon}.
\end{equation} 
Hence we see that for any $a\in X$,
\begin{align}\label{chi com}
\|[\chi(1_Ap1_A),a]\| & \leq \|[\chi(1_Ap1_A)-1_Ap1_A,a]\|+\|[1_Ap1_A,a]\|< 2\sqrt{2\epsilon}+\epsilon.
\end{align}
Putting the discussion so far together, $\chi(1_Ap1_A)$ is a projection in $\LL(1_AE)$ such that $\|[\chi(1_Ap1_A),a]\|<5\sqrt{\epsilon}$ for all $a\in X$.  We have moreover that $1_Ap1_A-1_Ae=1_A(p-e)1_A\in \K(1_AE)$, whence also $\chi(1_Ap1_A)-1_Ae\in \K(1_AE)$.  In conclusion, we see that $\chi(1_Ap1_A)$ defines an element of $\mathcal{P}^{\pi_1,u}_{5\sqrt{\epsilon}}(X,B)$.  We have thus shown that $\psi$ is well-defined.

It is straightforward to check that homotopies pass through the above construction, so that $\psi$ induces a well-defined map of sets 
$$
\psi_*:KK^{\pi}_\epsilon(X,B)\to KK^{\pi_1,u}_{5\sqrt{\epsilon}}(X,B).
$$
Finally, to see that $\psi_*$ is a homomorphism, we fix Cuntz isometries $s_1,s_2$ inducing the group operation in $KK^{\pi}_\epsilon(X,B)$.  As in the proof of Lemma \ref{phi hom}, we may use the Cuntz isometries $t_1:=1_As_1$ and $t_2:=1_As_2$ to define the group operation on $KK^{\pi_1,u}_{5\sqrt{\epsilon}}(X,B)$.  Using naturality of the functional calculus and the fact that $s_1$ and $s_2$ commute with $1_A$, we see that for $p,q\in KK^{\pi}_\epsilon(X,B)$ we have that 
$$
\chi(1_A(s_1ps_1^*+s_2qs_2^*)1_A)=t_1\chi(1_Ap1_A)t_1^*+t_2\chi(1_Aq1_A)t_2^*,
$$
and thus that $\psi_*([p]+[q])=\psi_*[p]+\psi_*[q]$, completing the proof.
\end{proof}

\begin{lemma}\label{compare unital}
Fix notation as in Definition \ref{kku op}.  Assume moreover that $\epsilon<1/8$ and that $X$ contains the unit of $A$.  Consider the diagrams
\begin{equation}\label{phipsi2}
\xymatrix{ KK^\pi_\epsilon(X,B)  \ar@{=}[r] & KK^\pi_{\epsilon}(X,B) \ar[d]^-{\psi_*} \\ 
KK_{\epsilon}^{\pi_1,u}(X,B) \ar[u]^-{\phi_*} \ar[r] & KK^{\pi_1,u}_{5\sqrt{\epsilon}}(X,B)  }
\end{equation}
and 
\begin{equation}\label{phipsi1}
\xymatrix{ KK^\pi_\epsilon(X,B) \ar[dd]^-{\psi_*} \ar[rr] & & KK^\pi_{8\sqrt{\epsilon}}(X,B) \\ 
& KK_{5\sqrt{\epsilon}}^{\pi_1,u}(X,B) \ar[ur] & \\
KK_{5\sqrt{\epsilon}}^{\pi_1,u}(X,B) \ar@{=}[r] & KK^{\pi_1,p}_{5\sqrt{\epsilon}}(X,B) \ar[u]^-{\phi_*} }
\end{equation}
where the unlabeled arrows are the canonical forget control maps, defined as in Definition \ref{good i l}.  These both commute.
\end{lemma}

\begin{proof}
For any $p\in \mathcal{P}^{\pi_1,u}_\epsilon(X,B)$ we have that $\psi(\phi(p))=p$, and so the diagram in line \eqref{phipsi2} clearly commutes.  For the diagram in line \eqref{phipsi1}, we need to show that if $p\in \mathcal{P}_\epsilon(X,B)$, then the classes of $p$ and of $\chi(1_Ap1_A)+(1-1_A)e$ in $KK_{8\sqrt{\epsilon}}(X,B)$ are the same.  For this, we concatenate two homotopies.  First, consider the homotopy 
$$
t\mapsto p_t:=\chi(1_Ap1_A)+(1-1_A)(te+(1-t)p),\quad t\in [0,1].
$$
As $ap_t=a\chi(1_Ap1_A)$ for all $a\in A$ and all $t\in [0,1]$, we see that $a(p_t^2-p_t)=0$.  Moreover, as $A$ commutes with $e$, as $\|[p,a]\|<\epsilon$ for all $a\in X$, and as $\|[\chi(1_Ap1_A),a]\|<5\sqrt{\epsilon}$ for all $a\in X$, we see that $\|[p_t,a]\|<5\sqrt{\epsilon}+\epsilon<6\sqrt{\epsilon}$ for all $a\in X$.  Hence this homotopy passes through $\mathcal{P}^\pi_{6\sqrt{\epsilon}}(X,B)$, and connects $\chi(1_Ap1_A)+(1-1_A)e$ and $\chi(1_Ap1_A)+(1-1_A)p$.  

For the second homotopy, note first that lines \eqref{1p1} and \eqref{[p,1]} imply that 
\begin{equation}\label{1p1 1p}
\|\chi(1_Ap1_A)-1_Ap\|\leq \|\chi(1_Ap1_A)-1_Ap1_A\|+\|1_A[1_A,p]\|< \sqrt{2\epsilon}+\epsilon.
\end{equation}
Consider now the homotopy 
\begin{equation}\label{q hom}
t\mapsto q_t:=(1-t)\chi(1_Ap1_A)+t1_Ap+(1-1_A)p,\quad t\in [0,1]
\end{equation}
Write $r_t:=(1-t)\chi(1_Ap1_A)+t1_Ap$, so we have $\|r_t-\chi(1_Ap1_A)\|<\sqrt{2\epsilon}+\epsilon$ for all $t$ by line \eqref{1p1 1p}.  Hence for any $a\in A$,
\begin{align*}
\|&a(q_t^2-q_t)\| = \|a(r_t^2-r_t)\| \\ & \leq \|r_t(r_t-\chi(1_Ap1_A))\|+\|(\chi(1_Ap1_A)-r_t)\chi(1_Ap1_A)\|+\|r_t-\chi(1_Ap1_A)\| \\ & <3(\sqrt{2\epsilon}+\epsilon).
\end{align*}
Moreover, for any $a\in X$, lines \eqref{chi com} and \eqref{[p,1]} give that
$$
\|[q_t,a]\|\leq \|[\chi(1_Ap1_A),a]\|+\|[1_Ap,a]\|+\|[(1-1_A)p,a]\|\leq \sqrt{5\epsilon}+2\epsilon.
$$
Putting all this together, the homotopy $t\mapsto q_t$ from line \eqref{q hom} passes through $KK_{8\sqrt{\epsilon}}^\pi(X,B)$.   As this homotopy connects $\chi(1_Ap1_A)+(1-1_A)p$ and $p$, this completes the proof.
\end{proof}

We are now in a position to establish the following, which is the main goal of this subsection.

\begin{proposition}\label{kkl kku}
Let $A$ and $B$ be separable $C^*$-algebras with $A$ unital.   Let $\pi:A\to \LL(E)$ be a graded, balanced, and strongly absorbing representation of $A$ on a Hilbert $B$-module.  Then with notation as in Definition \ref{kku} above there are isomorphisms
\begin{equation}\label{kl uni iso}
KL(A,B)\to \lim_{\leftarrow} KK^{\pi_1,u}_\epsilon(X,B).
\end{equation}
and 
$$
\lim_{\leftarrow}{}\!^1 KK^{1\otimes \pi_1,u}_\epsilon(X,SB)\to \overline{\{0\}},
$$ 
where the limits are taken over the directed set $\mathcal{X}$ of Definition \ref{seps} and $\overline{\{0\}}$ is the closure of $0$ in $KK(A,B)$.  Moreover, the isomorphism in line \eqref{kl uni iso} is a homeomorphism when the right hand side is equipped with the inverse limit topology. 

Finally there is a short exact sequence 
$$
0\to \lim_{\leftarrow}{\!}^1 KK^{1\otimes \pi_1,u}_\epsilon(X,SB)\to KK(A,B)\to \lim_{\leftarrow} KK^{\pi_1,u}_\epsilon(X,B)\to 0.
$$
\end{proposition}

\begin{proof}
Thanks to Theorems \ref{id with KL} and \ref{c0l1} respectively, it will suffice to show that 
\begin{equation}\label{lim iso}
\lim_{\leftarrow} KK^{\pi_1,u}_\epsilon(X,B)\cong \lim_{\leftarrow} KK^{\pi}_\epsilon(X,B)
\end{equation}
and 
\begin{equation}\label{lim1 iso}
\lim_{\leftarrow}{\!}^1 KK^{\pi_1,u}_\epsilon(X,SB)\cong \lim_{\leftarrow}{\!}^1 KK^{\pi}_\epsilon(X,SB).
\end{equation}
Using Lemmas \ref{phi hom}, \ref{psi hom}, and \ref{compare unital} we can construct an increasing sequence $(X_n)$ of finite subsets of $A_1$ with dense union and that all contain the unit, a sequence $(\epsilon_n)$ in $(0,1/8)$ that tends to zero as $n\to\infty$, and a diagram
\begin{equation}\label{ladder}
\xymatrix{ \cdots \ar[r] & KK^{\pi}_{\epsilon_n}(X_n,B) \ar[dr]_{\psi_*^{(n)}} \ar[r] & KK^{\pi}_{\epsilon_{n-1}}(X_{n-1},B) \ar[dr]_{\psi_*^{(n-1)}} \ar[r] & \cdots \ar[r] \ar[dr]_{\psi_*^{(2)}} & KK^{\pi}_{\epsilon_1}(X_1,B) \\
\cdots \ar[r] & KK^{\pi_1,u}_{\epsilon_n}(X_n,B) \ar[u]^-{\phi_*^{(n)}} \ar[r] & KK^{\pi_1,u}_{\epsilon_{n-1}}(X_{n-1},B) \ar[u]^-{\phi_*^{(n-1)}} \ar[r] & \cdots \ar[r] & KK^{\pi_1,u}_{\epsilon_1}(X_1,B)\ar[u]^-{\phi_*^{(1)}} }
\end{equation}
where: the horizontal maps are forget control maps; the maps labeled $\phi_*^{(n)}$ are from Lemma \ref{phi hom}; the maps labeled $\psi_*^{(n)}$ are from Lemma \ref{psi hom}; each triangle of the form 
$$
\xymatrix{ KK^{\pi}_{\epsilon_n}(X_n,B) \ar[dr]^{\psi_*^{(n)}}& \\
KK^{\pi_1,u}_{\epsilon_n}(X_n,B) \ar[u]^-{\phi_*^{(n)}} \ar[r] & KK^{\pi_1,u}_{\epsilon_{n-1}}(X_{n-1},B) }
$$
and each triangle of the form 
$$
\xymatrix{ KK^{\pi}_{\epsilon_n}(X_n,B) \ar[dr]_{\psi_*^{(n)}} \ar[r] & KK^{\pi}_{\epsilon_{n-1}}(X_{n-1},B) \\
& KK^{\pi_1,u}_{\epsilon_{n-1}}(X_{n-1},B)  \ar[u]_-{\phi_*^{(n-1)}} }
$$
commutes.   Now, by assumption that $(X_n)$ is increasing and has dense union in $A_1$, and by assumption that $\epsilon_n\to 0$, the sequence $(X_n,\epsilon_n)$ is cofinal in the directed set of Definition \ref{seps} (compare Remark \ref{seps rem} part \eqref{cof seq}).  Hence by Remark \ref{gen inv lim rem}, part \eqref{cof iso} (respectively, Lemma \ref{cofin nml1}), the top (respectively, bottom) row of diagram \eqref{ladder} computes ${\displaystyle \lim_{\leftarrow} KK^{\pi}_\epsilon(X,B)}$ (respectively, ${\displaystyle \lim_{\leftarrow} KK^{\pi_1,u}_\epsilon(X,B)}$).  The isomorphism in line \eqref{lim iso} follows directly from this.  The isomorphism in line \eqref{lim1 iso} also follows from the commuting diagram of line \eqref{ladder}, but with $B$ replaced by $SB$: we leave the direct algebraic checks involved to the reader.
\end{proof}

\subsection{Unitally absorbing representations}

In Proposition \ref{kkl kku} above we established isomorphisms 
$$
KL(A,B)\to \lim_{\leftarrow} KK^{\pi_1,u}_\epsilon(X,B).
$$
and 
$$
\lim_{\leftarrow}{}\!^1 KK^{1\otimes \pi_1,u}_\epsilon(X,SB)\to \overline{\{0\}},
$$ 
where $\pi$ is a graded, balanced, and strongly absorbing representation, and $\pi_1$ is the associated corestriction to a unital representation.  This is a little unnatural, however: it would be better to establish these isomorphisms with $\pi_1$ replaced by a more general unital representation satisfying appropriate assumptions.  Our goal in this section is to make that precise.

First, we recall a definition, which is essentially \cite[Definition 2.2]{Thomsen:2000aa} (compare also condition (2) from \cite[Theorem 2.1]{Thomsen:2000aa}).  It should be compared to Definition \ref{ab def} above.

\begin{definition}\label{ab unital def}
Let $A$ and $B$ be separable $C^*$-algebras with $A$ unital, and let $F$ be a Hilbert $B$-module.  A  unital representation $\pi:A \to \LL(F)$ is \emph{unitally absorbing} (for the pair $(A,B)$) if for any Hilbert $B$-module $E$ and ucp map $\sigma:A \to \LL(E)$, there is a sequence $(v_n)$ of isometries in $\LL(E,F)$ such that:
\begin{enumerate}[(i)]
\item $\sigma(a)-v_n^*\pi(a)v_n\in \K(E)$ for all $a\in A$ and $n\in\N$;
\item $\|\sigma(a)-v_n^*\pi(a)v_n\|\to 0$ as $n\to\infty$ for all $a\in A$.
\end{enumerate}
The representation $\pi$ is \emph{strongly unitally absorbing} if it is an infinite amplification of an absorbing representation.
\end{definition}

The following lemma, which follows ideas of Kasparov \cite{Kasparov:1980sp} (compare also \cite[Theorem 2.1]{Thomsen:2000aa}), says that unitally absorbing representations are essentially unique.  It is well-known; we could not find exactly what we needed in the literature, so give a proof.

\begin{lemma}\label{sau unique}
Let $A$ and $B$ be separable $C^*$-algebras with $A$ unital, and let $\pi:A\to \LL(F)$ and $\sigma:A\to \LL(E)$ be  unitally absorbing representations.  Then there is a sequence $(u_n)$ of unitaries in $\LL(E,F)$ such that 
\begin{enumerate}[(i)]
\item $\sigma(a)-u_n^*\pi(a)u_n\in \K(E)$ for all $a\in A$ and $n\in\N$;
\item $\|\sigma(a)-u_n^*\pi(a)u_n\|\to 0$ as $n\to\infty$ for all $a\in A$.
\end{enumerate}
%If moreover $E$ and $F$ are strongly unitally absorbing, then there is a unitary $u\in C_{ub}([1,\infty),\LL(F,E))$ such that $u^*\pi(a)u-\sigma(a)\in C_0([1,\infty),\K(F))$
\end{lemma}

\begin{proof}
Let $(\sigma^\infty,E^\infty)$ be the infinite amplification of $(\sigma,E)$, and let $(v_n)$ be a sequence of isometries in $\LL(E^\infty,F)$ such that $v_n^*\pi(a)v_n-\sigma^\infty(a)\to 0$ for all $a\in A$, and such that $v_n^*\pi(a)v_n-\sigma^\infty(a)\in \K(E^\infty)$ for all $a\in A$ and all $n$.  Using that
$$
(\pi(a)v_n-v_n\sigma^\infty(a))^*(\pi(a)v_n-v_n\sigma(a)^\infty)
$$
equals 
$$
v_n^*\pi(a^*a)v_n-\sigma^\infty(a^*a) -(v_n^*\pi(a^*)v_n-\sigma^\infty(a^*))\sigma^\infty(a)-\sigma^\infty(a^*)(v^*_n\pi(a)v_n-\sigma^\infty(a))
$$
for all $n$ and all $a\in A$, we see that we also have $\pi(a)v_n-v_n\sigma^\infty(a)\in \K(E^\infty,F)$ for all $n$ and all $a\in A$, and that $\|\pi(a)v_n-v_n\sigma^\infty(a)\|\to 0$ as $n\to\infty$ for all $a\in A$.  

Now, for representations $\phi:A\to \LL(G)$ and $\psi:A\to \LL(H)$ on Hilbert $B$-modules, let us write $\phi\sim \psi$ if there is a sequence of unitaries $(u_n)$ in $\LL(G,H)$ such that $\phi(a)-u_n^*\psi(a)u_n\in \K(G)$ for all $a\in A$ and $n\in\N$ and $\|\phi(a)-u_n^*\psi(a)u_n\|\to 0$ as $n\to\infty$ for all $a\in A$.   Let $u_n^F\in \LL(F,E\oplus F)$ be the unitary built from $v_n$ as in Lemma \ref{de2.3}.  Then the sequence $(u^F_n)$ in $\LL(F,E\oplus F)$ shows that $\pi\sim \sigma\oplus \pi$.  As the situation is symmetric in $\sigma$ and $\pi$, we also see that $\sigma \sim \sigma \oplus \pi$.  As $\sim$ is transitive, we see that $\pi\sim \sigma$ and are done.
%The proof in the strongly unitally absorbing case is similar, and essentially the same as that of Proposition \ref{cov isom}: we leave the details to the reader.
\end{proof}

\begin{corollary}\label{usa unique}
Let $A$ and $B$ be separable $C^*$-algebras with $A$ unital, and let $\pi:A\to \LL(E)$ be a unitally absorbing representation.  Then $E\cong \ell^2\otimes B$.
\end{corollary}

\begin{proof}
Using \cite[Theorem 2.4]{Thomsen:2000aa}, if $A$ and $B$ are separable with $A$ unital, there always exists a unitally absorbing representation $\pi:A\to \LL(\ell^2\otimes B)$.  Hence if $\sigma:A\to \LL(E)$ is any unitally absorbing representation, we must have that $E$ is isomorphic as a Hilbert $B$-module to $\ell^2\otimes B$.
\end{proof}

\begin{remark}\label{usa unique 2}
The same conclusion as in Corollary \ref{usa unique} holds if $A$ is not necessarily unital, and $\pi:A\to \LL(E)$ is absorbing; essentially the same argument works.
\end{remark}

We will need a lemma relating unitally absorbing representations to absorbing representations.

\begin{lemma}\label{sa vs usa}
Let $A$ and $B$ be separable $C^*$-algebras with $A$ unital, and let $\pi:A\to \LL(F)$ be an absorbing representation.  Then the corestriction of $\pi_1$ of $\pi$ to a unital representation $\pi_1:A\to \LL(\pi(1_A) F)$ is a unitally absorbing representation.

Conversely, if $\pi:A\to \LL(F)$ is a unitally absorbing, then the representation $\pi\oplus 0:A\to \LL(F\oplus F)$ is absorbing.
\end{lemma}

\begin{proof}
Let $\sigma:A\to \LL(E)$ be a ucp map with $E$ a Hilbert $B$-module.  As $\pi$ is absorbing, there is a sequence $(v_n)$ of isometries in $\LL(E,F)$ such that 
$$
\sigma(a)-v_n^*\pi(a)v_n\in \K(E)
$$
for all $a\in A$ and $n\in\N$, and such that 
$$
\|\sigma(a)-v_n^*\pi(a)v_n\|\to 0
$$
as $n\to\infty$ for all $a\in A$.  As $\sigma$ is unital we in particular have that $\|1_E-v_n^*\pi(1_A)v_n\|\to 0$ as $n\to\infty$.  Set $w_n:=\pi(1_A)v_n\in \LL(E,\pi(1_A)F)\subseteq \LL(E,F)$.  We compute that 
$$
w_n^*w_n-1_E=v_n^*\pi(1_A)v_n-1_E
$$
so $w_n^*w_n$ is a compact perturbation of $1_E$, and $\|w_n^*w_n-1_E\|\to 0$ as $n\to\infty$.  Passing to a subsequence, we may assume in particular that $w_n^*w_n$ is invertible for all $n$.  Note then that for all $n$
\begin{equation}\label{w props}
(w_n^*w_n)^{-1/2}-1_E\in \K(E), \quad \text{and} \quad (w_n^*w_n)^{-1/2}-1_E\to 0~\text{as}~n\to\infty.
\end{equation}
Define $u_n:=w_n(w_n^*w_n)^{-1/2}$.  Then $(u_n)$ is a sequence of isometries in $\LL(E,\pi(1_A)F)$ such that 
$$
\sigma(a)-u_n^*\pi(a)u_n=\sigma(a)-(w_n^*w_n)^{-1/2}v_n\pi(a)v_n(w_n^*w_n)^{-1/2}
$$
for all $a\in A$.  This computation combined with line \eqref{w props} shows that $(u_n)$ has the properties needed to show that $\pi_1$ is unitally absorbing.

Conversely, say $\pi:A\to \LL(F)$ is unitally absorbing, and let $\sigma:A\to \LL(E)$ be a ccp map.  As in \cite[Proposition 2.2.1]{Brown:2008qy}, $\sigma$ extends uniquely to a ucp map $\sigma^+\:A^+\to \LL(E)$ from the unitization $A^+$ of $A$.  Let $(\pi\oplus 0)^+:A^+\to \LL(F\oplus F)$ be the usual unitization of $\pi\oplus 0$, so $\pi^+(1_{A})$ maps to the unit of the first copy of $F$, and $\pi(1_{A^+}-1_A)$ maps to the unit of the second copy.  We claim that $(\pi\oplus 0)^+$ is unitally absorbing as a representation of $A^+$.  Indeed, with respect to the usual isomorphism $A^+\cong A\oplus \C$ for a unital $C^*$-algebra $A$ (the copy of $\C$ is generated by $1_{A^+}-1_A$), $(\pi\oplus 0)^+$ splits as a direct sum of representations $\pi\oplus \tau:A\oplus \C\to \LL(F)\oplus \LL(F)$, where $\tau$ is the unital representation on $\LL(F)$.  The representation $\tau$ is unitally absorbing by Kasparov's stabilization theorem \cite[Theorem 2]{Kasparov:1980sp}, and a direct sum of unitally absorbing representations in unitally absorbing by the equivalence of (2) and (3) from \cite[Theorem 2.1]{Thomsen:2000aa}, completing the proof of the claim.  It follows that there is a sequence of isometries $v_n:E\to F\oplus F$ such that $\|\sigma^+(a)-v_n^*(\pi\oplus 0)^+(a)v_n\|\to 0$ and $\sigma^+(a)-v_n^*(\pi\oplus 0)^+(a)v_n\in \K(E)$ for all $a\in A^+$.  The same holds if we remove the superscripts ``$^+$'' and quantify over $a\in A$, so we are done.
\end{proof}

We will need a unital variant of Definition \ref{apt}.

\begin{definition}\label{sub uni}
A representation $\pi:A\to \LL(E)$ is \emph{graded, balanced, and strongly unitally absorbing} if it comes with a fixed grading $(\pi,E)=(\pi_0\oplus \pi_0,E_0\oplus E_0)$ such that $(\pi_0,E_0)$ is strongly unitally absorbing.
\end{definition}

Given this, the following corollary of Lemma \ref{sa vs usa} is immediate.

\begin{corollary}\label{sa vs usa cor}
Let $A$ and $B$ be separable $C^*$-algebras with $A$ unital, and let $\pi:A\to \LL(F)$ be a strongly absorbing representation on a Hilbert $B$-module.  Then the corestriction of $\pi_1$ of $\pi$ to a unital representation $\pi:A\to \LL(\pi(1_A)F)$ is a strongly unitally absorbing representation.

Conversely, if $\pi:A\to \LL(F)$ is a strongly unitally absorbing, then the representation $\pi\oplus 0:A\to \LL(F\oplus F)$ is strongly absorbing. \qed
\end{corollary}

Our main goal in this section is the following result, which says essentially that any strongly unitally absorbing representation can be used to compute $KL(A,B)$ as an inverse limit.  %For the statement, let us say that a unitally absorbing representation $(\pi,E)$ is \emph{balanced graded} if comes with a fixed grading of the form $(\pi_0\oplus \pi_0,E_0\oplus E_0)$, with $(\pi_0,E_0)$ unitally absorbing.

\begin{proposition}\label{usa rep gives kl}
Let $A$ and $B$ be separable $C^*$-algebras with $A$ unital.  Then for any graded, balanced, and strongly unitally absorbing representation $\pi$ of $A$ there are isomorphisms
\begin{equation}\label{kl usa}
KL(A,B)\to \lim_{\leftarrow} KK^{\pi,u}_\epsilon(X,B).
\end{equation}
and 
$$
\lim_{\leftarrow}{}\!^1 KK^{{1\otimes }\pi,u}_\epsilon(X,SB)\to \overline{\{0\}},
$$ 
where the limits are taken over the directed set $\mathcal{X}$ of Definition \ref{seps} and $\overline{\{0\}}$ is the closure of $0$ in $KK(A,B)$.  Moreover, the isomorphism in line \eqref{kl usa} is a homeomorphism when the right hand side is equipped with the inverse limit topology. 

Finally there is a short exact sequence 
$$
0\to \lim_{\leftarrow}{\!}^1 KK^{1\otimes \pi,u}_\epsilon(X,SB)\to KK(A,B)\to \lim_{\leftarrow} KK^{\pi,u}_\epsilon(X,B)\to 0.
$$
\end{proposition}

\begin{proof}
Proposition \ref{kkl kku} establishes this in the special case that $\pi_1$ is the unital corestriction of a graded, balanced, strongly absorbing representation (it is a special case by Corollary \ref{sa vs usa cor}).  Lemma \ref{sa vs usa cor} implies that any graded, balanced, strongly unitally absorbing representation is the unital corestriction of a graded, balanced, strongly absorbing representation, however, so we are done.
\end{proof}

\begin{remark}\label{kas rem}
Let $A$ and $B$ be separable, with $A$ unital.  Assume also that at least one of $A$ and $B$ is nuclear.  It follows from \cite[Theorem 5]{Kasparov:1980sp} that if $\pi:A\to \mathcal{B}(\ell^2)$ is a faithful unital representation such that $\pi^{-1}(\K(\ell^2))=\{0\}$, then the amplification $\pi\otimes 1:A\to \LL(\ell^2\otimes B)$ is unitally absorbing.  We will not use this remark directly in the paper, but it is needed to justify the simplified picture of controlled $KK$-theory that we use in the introduction.
\end{remark}

\subsection{Matricial representations of controlled $KK$-groups}

In this subsection we give a formulation of controlled $KK$-theory in terms of matrices, which is perhaps closer to standard formulations of elementary $C^*$-algebra $K$-theory.  Although the definitions in the main body of the paper are more convenient for establishing the theory (particularly with regard to the topology on $KK$), this definition will make computations easier in some subsequent applications \cite{Willett:2021te}.

For a representation $\pi:A\to \LL(E)$ we use the amplifications $1_{M_n}\otimes \pi:A\to M_n(\LL(E))$ to identify $A$ with a (diagonal) $C^*$-subalgebra of $M_n(\LL(E))$ for all $n$.

\begin{definition}\label{alm com 0}
Let $A$ be unital, and let $\pi:A\to \LL(E)$ be a unital representation.  Let $\K(E)^+$ be the unitization of $\K(E)$.

Let $X$ be a finite subset of $A_1$, let $\epsilon>0$, and let $n\in \N$.  Define $\mathcal{P}^{\pi,\text{mx}}_{n,\epsilon}(X,B)$\footnote{The ``$mx$'' is for ``matrix''.} to be the collection of pairs $(p,q)$ of projections in $M_n(\K(E)^+)$ satisfying the following conditions:
\begin{enumerate}[(i)]
\item $\|[p,a]\|< \epsilon$ and $\|[q,a]\|< \epsilon$ for all $a\in X$;
\item the classes $[p],[q]\in K_0(\C)$ formed by taking the images of $p$ and $q$ under the canonical quotient map $M_n(\K(E)^+)\to M_n(\C)$ are the same.
\end{enumerate}
If $(p_1,q_1)$ is an element of $\mathcal{P}^{\pi,\text{mx}}_{n_1,\epsilon}(X,B)$ and $(p_2,q_2)$ is an element of $\mathcal{P}^{\pi,\text{mx}}_{n_2,\epsilon}(X,B)$, define 
$$
(p_1\oplus p_2,q_1\oplus q_2):=\Bigg(\begin{pmatrix} p_1 & 0 \\ 0 & p_2 \end{pmatrix} ~,~\begin{pmatrix} q_1 & 0 \\ 0 & q_2 \end{pmatrix}\Bigg) \in \mathcal{P}^{\pi,\text{mx}}_{n_1+n_2,\epsilon}(X,B).
$$
Define  
$$
\mathcal{P}^{\pi,\text{mx}}_{\infty,\epsilon}(X,B):=\bigsqcup_{n=1}^\infty \mathcal{P}^{\pi,\text{mx}}_{n,\epsilon}(X,B),
$$
i.e.\ $\mathcal{P}^{\pi,\text{mx}}_{\infty,\epsilon}(X,B)$ is the \emph{disjoint} union of all the sets $ \mathcal{P}^{\pi,\text{mx}}_{n,\epsilon}(X,B)$.

Equip each $\mathcal{P}^{\pi,\text{mx}}_{n,\epsilon}(X,B)$ with the norm topology it inherits from $M_n(\LL(E))\oplus M_n(\LL(E))$, and equip $\mathcal{P}^{\pi,\text{mx}}_{\infty,\epsilon}(X,B)$ with the disjoint union topology.  Let $\sim$ be the equivalence relation on $\mathcal{P}^{\pi,\text{mx}}_{\infty,\epsilon}(X,B)$ generated by the following relations:
\begin{enumerate}[(i)]
\item \label{kkm sim2} $(p,q)\sim (p\oplus r,q\oplus r)$ for any element $(r,r)\in \mathcal{P}^{\pi,\text{mx}}_{\infty,\epsilon}(X,B)$ with both components the same;
\item \label{kkm sim1} $(p_1,q_1)\sim (p_2,q_2)$ whenever these elements are in the same path component of $\mathcal{P}^{\pi,\text{mx}}_{\infty,\epsilon}(X,B)$.\footnote{Equivalently, both are in the same $\mathcal{P}^{\pi,\text{mx}}_{n,\epsilon}(X,B)$, and are in the same path component of this set.}
\end{enumerate}
Finally, define $KK^{\pi,\text{mx}}_{\epsilon}(X,B)$ to be $\mathcal{P}^{\pi,\text{mx}}_{\infty,\epsilon}(X,B)/\sim$.
\end{definition}

\begin{lemma}\label{kkm group}
Let $A$ and $B$ be separable $C^*$-algebras with $A$ unital.  Let $X\subseteq A_1$ be a finite set, and let $\epsilon>0$.  If $\pi:A\to \LL(E)$ is any unital representation, then $KK^{\pi,\text{mx}}_{\epsilon}(X,B)$ is an abelian group.
\end{lemma}

\begin{proof}
It follows directly from the definition that $KK^{\pi,\text{mx}}_{\epsilon}(X,B)$ is a monoid with identity element the class $[0,0]$.  A standard rotation homotopy shows that $KK^{\pi,\text{mx}}_{\epsilon}(X,B)$ is commutative.  To complete the proof, we claim that $[q,p]$ is the inverse of $[p,q]$.  Indeed, applying a rotation homotopy to the second variable shows that $(p\oplus q,q\oplus p)\sim (p\oplus q,p\oplus q)$, and the element $(p\oplus q,p\oplus q)$ is trivial by definition of the equivalence relation.
\end{proof}

Now, let $\pi:A\to \LL(E)$ be a unitally strongly absorbing representation as in Definition \ref{sub uni}.% i.e.\ there is a decomposition $(\pi,E)=(\pi_0\oplus \pi_0,E_0\oplus E_0)$, where $(\pi_0,E_0)$ is a strongly unitally absorbing (ungraded) representation.  Under this identification, we have a canonical identification $\LL(E)=M_2(\LL(E_0))$ under which the neutral projection $e_\pi$ on $E$ corresponds to the matrix $\begin{psmallmatrix} 1 & 0 \\ 0 & 0 \end{psmallmatrix}$.  
Our goal in this section is to establish isomorphisms 
$$
KL(A,B)\to \lim_{\leftarrow} KK^{\pi,m}_{\epsilon}(X,B)
$$
and 
$$
\lim_{\leftarrow}{}\!^1 KK^{1\otimes \pi,m}_{\epsilon}(X,SB)\to \overline{\{0\}}
$$
analogously to Propositions \ref{kkl kku} and \ref{usa rep gives kl} above; here the limits are (as usual) taken over the directed set $\mathcal{X}$ of Definition \ref{seps}.

Let then $(\pi,E)=(\pi_0\oplus \pi_0,E_0\oplus E_0)$ be a graded, balanced, unitally strongly absorbing representation, so $\pi_0$ is any graded balanced representation.  First, we provisionally define 
$$
\phi:\mathcal{P}^{\pi,u}_\epsilon(X,B) \to \mathcal{P}^{\pi_0,\text{mx}}_{2,\epsilon}(X),\quad p\mapsto \Big(\,p~,\,\begin{psmallmatrix} 1 & 0 \\ 0 & 0 \end{psmallmatrix}\,\Big),
$$
where we have used the identification $\LL(E)=M_2(\LL(E_0))$ to make sense of the right hand side.  

\begin{lemma}\label{phi map 2}
The map $\phi$ above is well-defined, and descends to a group homomorphism 
$$
\phi_*:KK^{\pi,u}_\epsilon(X,B)\to KK^{\pi_0,\text{mx}}_{\epsilon}(X,B).
$$
\end{lemma}

\begin{proof}
Using the correspondence $e_\pi\leftrightarrow \begin{psmallmatrix} 1 & 0 \\ 0 & 0\end{psmallmatrix}$ one sees that the image of $\phi$ is indeed contained in $\mathcal{P}^{\pi_0,\text{mx}}_{2,\epsilon}(X,B)$.  Moreover, $\phi$ takes homotopies to homotopies, so descends to a well-defined map of sets $\phi_*:KK^{\pi,u}_\epsilon(X,B)\to KK^{\pi_0,\text{mx}}_{\epsilon}(X,B)$.  It remains to show that this set map is a homomorphism.  

For this, let $s_1,s_2\in \mathcal{B}(\ell^2)\subseteq \LL(E)$ be a pair of Cuntz isometries inducing the operation on $KK^{\pi,u}_\epsilon(X,B)$ as in Definition \ref{kku op}.  For simplicity of notation, let us write $e=\begin{psmallmatrix} 1 & 0 \\ 0 & 0 \end{psmallmatrix} \in M_2(\K(E_0)^+)$.  Then for $[p],[q]\in KK^{\pi,u}_{\epsilon}(X,B)$, we see that 
$$
\phi_*([p]+[q])=[s_1ps_1^*+s_2qs_2^*,e]
$$
(the entries on the right should be considered as matrices in $M_2(\K(E_0)^+)$).  According to the definition of the equivalence relation defining $KK^{\pi_0,\text{mx}}_\epsilon(X,B)$, this is the same element as 
$$
[s_1ps_1^*+s_2qs_2^*\oplus e,e\oplus e].
$$
For $t\in [0,\pi/2]$, write 
$$
u_t:=s_1s_1^*\otimes 1_2+(s_2s_2^*\otimes 1_2)\Bigg(1_{M_2(\LL(E_0))}\otimes \begin{pmatrix} \cos(t) & -\sin(t) \\ \sin(t) & \cos(t) \end{pmatrix} \Bigg)
$$
(here $1_2\in M_2(\C)$, so we are considering each $u_t$ as an element of $\LL(E)\otimes M_2(\C)=M_2(\LL(E_0))\otimes M_2(\C)$).  Consider now the path 
\begin{equation}\label{ut hom}
(u_t(s_1ps_1^*+s_2qs_2^*\oplus e)u_t^*,u_t(e\oplus e)u_t^*),\quad t\in [0,\pi/2].
\end{equation}
We have that $u_t(e\oplus e)u_t^*=e\oplus e$ for all $t$.  As 
$$
((s_1ps_1^*+s_2qs_2^*)\oplus e)-(e\oplus e)
$$
is in $M_4(\K(E_0))$, we thus see that 
$$
u_t((s_1ps_1^*+s_2qs_2^*)\oplus e)u_t^*-u_t(e\oplus e)u_t^*=u_t(s_1ps_1^*+s_2qs_2^*\oplus e)u_t^*-e\oplus e.
$$
is also in $M_4(\K(E_0))$ 
It follows from this that $u_t(s_1ps_1^*+s_2qs_2^*\oplus e)u_t^*$ is in $M_4(\K(E_0)^+)$ for all $t\in [0,\pi/2]$, and therefore the path in line \eqref{ut hom} passes through $\mathcal{P}_{4,\epsilon}^{\pi_0,\text{mx}}(X,B)$.  As such, it shows that in $KK_\epsilon^{\pi_0,\text{mx}}(X,B)$ we have the identity
$$
[s_1ps_1^*+s_2qs_2^*\oplus e,e\oplus e]=[s_1ps_1^*+s_2es_2^*\oplus s_2qs_2^*+s_1es_1^*,e\oplus e].
$$
As the left hand side above is $\phi_*([p]+[q])$ we thus get 
$$
\phi_*([p]+[q])=[s_1ps_1^*+s_2es_2^*,e]+[s_2qs_2^*+s_1es_1^*,e].
$$
To complete the proof, it this suffices to show that $[s_1ps_1^*+s_2es_2^*,e]=\phi_*[p]$ and $[s_2qs_2^*+s_1es_1^*,e]=\phi_*[q]$, i.e.\ that $[s_1ps_1^*+s_2es_2^*,e]=[p,e]$ and $[s_2qs_2^*+s_1es_1^*,e]=[q,e]$.  These identities follow from Lemma \ref{id lem 3} (the first with $v=s_1$ on using the identity $(1-s_1s_1^*)e=s_2s_2^*e=s_2es_2^*$, and the second similarly with $v=s_2$), which completes the proof.
\end{proof}

We now define a map going in the other direction to $\phi$; this is more complicated.  Recall first that (as throughout the paper) ``$\ell^2$'' is shorthand for $\ell^2(\N)$.  We write $(\ell^2)^{\oplus n}$ for the direct sum of $\ell^2$ with itself $n$ times; of course, this is isomorphic to $\ell^2$, but the distinction can help keep track of notation.

To start, for each $n$, fix a unitary isomorphism $v_n\in \mathcal{B}(\C^2\otimes \ell^2,(\ell^2)^{\oplus 2n})$ such that if $p_n:(\ell^2)^{\oplus 2n}\to (\ell^2)^{\oplus 2n}$ is the projection onto the first $n$ components, then $v_np_nv_n^*=e$, where $e$ is (as usual) the projection of $\C^2\otimes \ell^2$ onto $\ell^2$ arising by projecting $\C^2$ onto its first coordinate.  Use the usual (compatible) identifications of $E$ with $\C^2\otimes \ell^2\otimes F$ and $E_0$ with $\ell^2\otimes F$ for some Hilbert module $F$, identify $\mathcal{B}(\C^2\otimes \ell^2,(\ell^2)^{\oplus 2n})$ with a subspace of $\LL(E,E_0^{\oplus 2n})$ and consider $v_n$ as an element here.  Up to the canonical identification $\LL(E_0^{\oplus 2n})=M_{2n}(\LL(E_0))$, we thus see that $v_nM_{2n}(\LL(E_0))v_n^*=\LL(E)$, and that $v_n\begin{psmallmatrix} 1 & 0 \\ 0 & 0 \end{psmallmatrix} v_n^*=e$,
where the entries of the matrix on the left are understood as $n\times n$ blocks.  

Now, let $(p,q)$ be an element of $\mathcal{P}^{\pi,\text{mx}}_{n,\epsilon}(X,B)$ for some $n$.  As the images of $p$ and $q$ under the canonical quotient map $\sigma:M_n(\K^+)\to M_n(\C)$ are the same in $K_0(M_n(\C))$, there is a unitary $u\in M_n(\C)$ such that $\sigma(p)=u\sigma(q)u^*$.  Define 
$$
v:=\begin{pmatrix} uqu^* & 1-uqu^* \\ 1-uqu^* & uqu^* \end{pmatrix}v_n \in \LL(E,E_0^{\oplus 2n}).
$$
Provisionally define a map 
$$
\psi: \mathcal{P}^{\pi_0,\text{mx}}_{\infty,\epsilon}(X,B)\to \mathcal{P}^{\pi,u}_{5\epsilon}(X,B),\quad (p,q)\mapsto v^*\begin{pmatrix} p & 0 \\ 0 & 1-uqu^* \end{pmatrix} v.
$$

\begin{lemma}\label{psi map 2}
The map $\psi$ above is well-defined, and descends to a group homomorphism 
$$
\psi_*:KK^{\pi_0,\text{mx}}_{\epsilon}(X,B)\to KK^{\pi,u}_{5\epsilon}(X,B)
$$
that does not depend on the choice of $u$ or $v_n$.
\end{lemma}

\begin{proof}
We first have to see that $\psi$ takes image in $\mathcal{P}^{\pi,u}_{5\epsilon}(X,B)$.  For simplicity of notation, let us replace $q$ with $uqu^*$, so we have that $p-q$ is in $M_{n}(\K)$.  Hence 
\begin{align*}
\begin{pmatrix} q & 1-q \\ 1-q & q \end{pmatrix}& \begin{pmatrix} p & 0 \\ 0 & 1-q \end{pmatrix} \begin{pmatrix} q & 1-q \\ 1-q & q \end{pmatrix} \\ & - \begin{pmatrix} q & 1-q \\ 1-q & q \end{pmatrix} \begin{pmatrix} q & 0 \\ 0 & 1-q \end{pmatrix} \begin{pmatrix} q & 1-q \\ 1-q & q \end{pmatrix} 
\end{align*}
is in $M_{2n}(\K)$, or in other words 
$$
\begin{pmatrix} q & 1-q \\ 1-q & q \end{pmatrix} \begin{pmatrix} p & 0 \\ 0 & 1-q \end{pmatrix} \begin{pmatrix} q & 1-q \\ 1-q & q \end{pmatrix}-\begin{pmatrix} 1 & 0 \\ 0 & 0 \end{pmatrix} 
$$
is in $M_{2n}(\K)$.  Conjugating by $v_n$, and identifying $\LL(E)=M_2(\LL(E_0))$ with the top left corner of $M_{2n}(\LL)$, and also recalling the correspondence $e_\pi\leftrightarrow \begin{psmallmatrix} 1 & 0 \\ 0 & 0\end{psmallmatrix}$, we see that 
$$
v_n^*\begin{pmatrix} q & 1-q \\ 1-q & q \end{pmatrix}\begin{pmatrix} p & 0 \\ 0 & 1-q \end{pmatrix} \begin{pmatrix} q & 1-q \\ 1-q & q \end{pmatrix}v_n-e_\pi
$$
is in $M_2(\K(E_0))$.  Direct checks show that the projection 
$$
\begin{pmatrix} q & 1-q \\ 1-q & q \end{pmatrix} \begin{pmatrix} p & 0 \\ 0 & 1-q \end{pmatrix} \begin{pmatrix} q & 1-q \\ 1-q & q \end{pmatrix}
$$ 
commutes with elements of $X$ up to error $5\epsilon$.  At this point, we have that $\psi$ does indeed define a function $\psi: \mathcal{P}^{\pi_0,\text{mx}}_{\infty,\epsilon}(X,B)\to \mathcal{P}^{\pi,u}_{5\epsilon}(X,B)$.

We now pass to the quotient on the right hand side, so getting a map $\psi_\flat: \mathcal{P}^{\pi_0,\text{mx}}_{\infty,\epsilon}(X,B)\to KK^{\pi,u}_{5\epsilon}(X,B)$.  We will show that this map does not depend on the choice of $u$ or $v_n$, which will certainly imply the same thing for $\psi_*$ once we show the latter exists.  To see that $\psi_\flat$ does not depend on the choices of $u$ such that $\sigma(p)=u\sigma(q)u^*$, note that if $U_n(\C)$ is the unitary group of $M_n(\C)$, then the collection of all such unitaries is homeomorphic to $\sigma(p)U_n(\C)\sigma(q)\times (1-\sigma(p))U_n(\C)(1-\sigma(q))$, so path connected.  Hence any two such choices give rise to homotopic elements of $\mathcal{P}^{\pi,u}_{5\epsilon}(X,B)$.  One can argue that $\psi_\flat$ does not depend on the choice of $v_n$ similarly: any two such choices are connected by a path that passes through such elements.

We now show that $\psi$ descends to a well-defined map $\psi_*:KK^{\pi_0,\text{mx}}_{\epsilon}(X,B)\to KK^{\pi,u}_{5\epsilon}(X,B)$.  First we look at part \eqref{kkm sim1} of the definition of the equivalence relation defining $KK^{\pi_0,\text{mx}}_{\epsilon}(X,B)$ from Definition \ref{alm com 0}.  Let $(p_t,q_t)_{t\in [0,1]}$ be a homotopy in some $\mathcal{P}^{\pi_0,\text{mx}}_{n,\epsilon}(X,B)$.  Using for example Lemma \ref{pp lem} we may choose a continuous path of unitaries $(u_t)_{t\in [0,1]}$ in $M_n(\C)$ such that $\sigma(p_t)=u_t\sigma(q)u_t^*$ for all $t\in [0,1]$, and use these to define $\phi_*[p_t,q_t]$ for each $t$.  From here, it is straightforward to see that $\psi$ takes homotopies to homotopies, so we are done with this part of the equivalence relation.

We now look at part \eqref{kkm sim2} of the equivalence relation from Definition \ref{alm com 0}.  We compute the image of $(p\oplus r,q\oplus r)$ under $\psi$ as follows, where $p,q\in M_n(\K(E_0)^+)$ and $r\in M_k(\K(E_0)^+)$ for some $n,k\in \N$.  Let $u\in M_n(\C)$ be a unitary such that $\sigma(p)=u\sigma(q)u^*$ in $M_n(\C)$, and set $q'=uqu^*$.  Then one computes that $\psi$ sends $(p\oplus r,q\oplus r)$ to
\begin{equation}\label{psi p,q plus r}
v_{n+k}^*\begin{pmatrix} q'pq' + 1-q' & 0 & q'p(1-q') & 0 \\
0 & 1 & 0 & 0 \\
(1-q')pq' & 0 & (1-q')p(1-q') & 0 \\
0 & 0 & 0 & 0 \end{pmatrix}v_{n+k}
\end{equation}
(the odd rows (respectively, columns) have height (resp. width) $n$, and the even rows (resp. columns) have height (resp. width) $k$).  On the other hand, $\psi$ sends $(p,q)$ to 
\begin{equation}\label{psi p,q}
v_{n}^*\begin{pmatrix} q'pq' + 1-q' &  q'p(1-q')  \\
(1-q')pq' &  (1-q')p(1-q')  \end{pmatrix}v_{n},
\end{equation}
so we must show that the elements in lines \eqref{psi p,q plus r} and \eqref{psi p,q} define the same class in $KK^{\pi,u}_{5\epsilon}(X,B)$.  Let now $i:E_0^{\oplus 2n}\to E_0^{\oplus 2(n+k)}$ be the canonical inclusion, and let $w_n:=i\circ v_n\in \mathcal{B}(\C^2\otimes \ell^2,(\ell^2)^{\oplus 2(n+k)})\subseteq \LL(E,E_0^{2(n+k)})$.  Set $v:=v_{n+k}^*w_n$, which is an isometry in $\mathcal{B}(\ell^2)\subseteq \LL(E)$.  Looking back at line \eqref{psi p,q plus r}, we have that 
\begin{align*}
&v_{n+k}^*\begin{pmatrix} q'pq' + 1-q' & 0 & q'p(1-q') & 0 \\
0 & 1 & 0 & 0 \\
(1-q')pq' & 0 & (1-q')p(1-q') & 0 \\
0 & 0 & 0 & 0 \end{pmatrix}v_{n+k}  \\ &= 
v_{n+k}^*\begin{pmatrix} q'pq' + 1-q' & 0 & q'p(1-q') & 0 \\
0 & 0 & 0 & 0 \\
(1-q')pq' & 0 & (1-q')p(1-q') & 0 \\
0 & 0 & 0 & 0 \end{pmatrix}v_{n+k} +v_{n+k}^*\begin{pmatrix} 0 & 0 & 0 & 0 \\
0 & 1 & 0 & 0 \\
0 & 0 & 0 & 0 \\
0 & 0 & 0 & 0 \end{pmatrix}v_{n+k}.
\end{align*}
The terms on the left and right above are equal to 
$$
vv_n^*\begin{pmatrix} q'pq' + 1-q' &  q'p(1-q')  \\
(1-q')pq' &  (1-q')p(1-q')  \end{pmatrix}v_nv^* \quad \text{and} \quad (1-vv^*)e
$$
respectively.  Putting all this together, we see that 
\begin{equation}\label{psi*}
\psi_*[p\oplus r,q\oplus r]=\Bigg[vv_n^*\begin{pmatrix} q'pq' + 1-q' &  q'p(1-q')  \\
(1-q')pq' &  (1-q')p(1-q')  \end{pmatrix}v_nv^*+(1-vv^*)e\Bigg].
\end{equation}
Lemma \ref{id lem 3} implies that the class on the right hand side of line \eqref{psi*} equals the class of the element in line \eqref{psi p,q}.  Hence we are done with this case of the equivalence relation too.  

At this point, we know that $\psi_*:KK^{\pi_0,\text{mx}}_{\epsilon}(X,B)\to KK^{\pi,u}_{5\epsilon}(X,B)$ is a well-defined set map.  It  remains to show that $\psi_*$ is a group homomorphism.  Let then $(p_1,q_1)$ and $(p_2,q_2)$ be elements of $\mathcal{P}_{n_1,\epsilon}^{\pi_0,\text{mx}}(X,B)$ and $\mathcal{P}_{n_2,\epsilon}^{\pi_0,\text{mx}}(X,B)$ respectively.  For notational simplicity, assume $p_i-q_i\in M_{n_i}(\K(E_0))$ for $i\in \{1,2\}$ by conjugating by an appropriate unitary as in the definition of $\psi$; this makes no real difference to the computations below.  The sum $[p_1,q_1]+[p_2,q_2]$ is represented by $[p_1\oplus p_2,q_1\oplus q_2]$, and this is mapped by $\psi_*$ to the class of the product
\begin{align}\label{big prod}
v_{n_1+n_2}^* & \begin{pmatrix} q_1 & 0 & 1-q_1 & 0 \\ 
0 & q_2 & 0 & 1-q_2 \\ 
1-q_1 & 0 & q_1 & 0 \\ 
0 & 1-q_2 & 0 & q_2 \end{pmatrix} 
\begin{pmatrix} p_1 & 0 & 0 & 0 \\ 
0 & p_2 & 0 & 0 \\ 
0 & 0 & 1-q_1 & 0 \\ 
0 & 0 & 0 & 1-q_2 \end{pmatrix} \nonumber  \\ &\quad\quad\quad\quad\quad\quad\cdot
\begin{pmatrix} q_1 & 0 & 1-q_1 & 0 \\ 
0 & q_2 & 0 & 1-q_2 \\ 
1-q_1 & 0 & q_1 & 0 \\ 
0 & 1-q_2 & 0 & q_2 \end{pmatrix} v_{n_1+n_2}.
\end{align}
Let now $s$ be the permutation unitary in $\mathcal{B}((\ell^2)^{\oplus 2(n_1+n_2)})\subseteq \LL(E_0^{\oplus 2(n_1+n_2)})$ such that conjugation by $s$ flips the second and third rows and columns in the matrices above.  Let $w_{1}:=i_{n_1}v_{n_1}$, where $i_{n_1}:E_0^{\oplus 2n_1}\to E_0^{\oplus 2(n_1+n_2)}$ is the natural inclusion, and similarly for $w_{2}$.  Set $s_1:=v_{n_1+n_2}^*sw_1$ and $s_2:=v_{n_1+n_2}^*sw_2$, so $s_1,s_2\in \mathcal{B}(\ell^2)\subseteq \LL(E)$, and they satisfy the Cuntz relation $s_1s_1^*+s_2s_2^*$ (this follows as $w_1w_1^*+w_2w_2^*=1$).  According to Lemma \ref{kku group}, we may use $s_1$ and $s_2$ to define the group operation on $KK^{\pi,u}_{5\epsilon}(X,B)$, and so 
\begin{align*}
\psi_*&[p_1,q_1]  +\psi_*[p_2,q_2] \\
& =\Bigg[~s_1v_{n_1}^*\begin{pmatrix} q_1 & 1-q_1 \\ 1-q_1 & q_1 \end{pmatrix}\begin{pmatrix} p_1 & 0 \\ 0 & 1-q_1\end{pmatrix} \begin{pmatrix} q_1 & 1-q_1 \\ 1-q_1 & q_1 \end{pmatrix}v_{n_1}s_1^* \\ & \quad \quad + s_1v_{n_2}^*\begin{pmatrix} q_2 & 1-q_2 \\ 1-q_2 & q_2 \end{pmatrix}\begin{pmatrix} p_2 & 0 \\ 0 & 1-q_2\end{pmatrix} \begin{pmatrix} q_2 & 1-q_2 \\ 1-q_2 & q_2 \end{pmatrix} v_{n_2}s_2^*~\Bigg].
\end{align*}
A direct computation shows that this equals the element in line \eqref{big prod} above, however, so we are done.
\end{proof}

We need one more technical lemma before we get to the main point.

\begin{lemma}\label{comp matrix}
Let $A$ and $B$ be separable $C^*$-algebras with $A$ unital.  Let $\pi:A\to \LL(E)$ be a graded, balanced, strongly unitally absorbing representation of $A$ on a Hilbert $B$-module.  Let $X\subseteq A_1$ be finite, and let $\epsilon>0$.  Consider the diagrams 
\begin{equation}\label{kkm1}
\xymatrix{ KK^{\pi,u}_{\epsilon}(X,B) \ar[d]^-{\phi_*} \ar[r] & KK^{\pi,u}_{5\epsilon}(X,B)\\
 KK^{\pi_0,\text{mx}}_{\epsilon}(X,B)  \ar@{=}[r] &  KK^{\pi_0,\text{mx}}_{\epsilon}(X,B) \ar[u]^-{\psi_*}} 
\end{equation}
and 
\begin{equation}\label{kkm2}
\xymatrix{ KK^{\pi,u}_{5\epsilon}(X,B)  \ar@{=}[r] & KK^{\pi_0,p}_{5\epsilon}(X,B) \ar[d]^-{\phi_*} \\
 KK^{\pi_0,\text{mx}}_{\epsilon}(X,B) \ar[u]^-{\psi_*} \ar[r] &  KK^{\pi_0,\text{mx}}_{5\epsilon}(X,B) } 
\end{equation}
where the horizontal arrows are the canonical forget control maps.  These commute.
\end{lemma}

\begin{proof}
We first look at diagram \eqref{kkm1}.  We compute that for $p\in \mathcal{P}_{\epsilon}^{\pi,u}(X,B)$, 
$$
\psi\phi(p)=v_2^*\begin{pmatrix} e & 1-e \\ 1-e & e \end{pmatrix} \begin{pmatrix} p & 0 \\ 0 & 1-e \end{pmatrix}\begin{pmatrix} e & 1-e \\ 1-e & e \end{pmatrix}v_2.
$$
The entries appearing above can be identified with $2\times 2$ matrices, with diagonal matrix units corresponding to $e$ and $1-e$, and $p$ corresponding to the matrix $\begin{pmatrix} epe & ep(1-e) \\ (1-e)pe & (1-e)p(1-e)\end{pmatrix}$.  With respect to this picture, one computes that the above equals 
$$
v_2^*\begin{pmatrix} epe & 0 & 0 & ep(1-e) \\ 
0 & 1 & 0 & 0 \\
0 & 0 & 0 & 0 \\
(1-e)pe & 0 & 0 & (1-e)p(1-e) \end{pmatrix} v_2.
$$
Define 
$$
i:E_0^{\oplus 2}\to E_0^{\oplus 4},\quad (v,w)\mapsto (v,0,0,w),
$$
and define $v:=v_2^*i$, which is an isometry inside the copy $\mathcal{B}(\ell^2)\subseteq \LL(E)$ from Lemma \ref{reps}.  One computes using the above that 
$$
\psi\phi(p)=vpv^*+(1-vv^*)e,
$$
whence $[\psi\phi(p)]=[p]$ by Lemma \ref{id lem 3}, as required.

Now let us look at diagram \eqref{kkm2}.  Let $(p,q)$ be an element of $\mathcal{P}^{\pi_0,\text{mx}}_{n,\epsilon}(X,B)$ for some $n$.  For notational convenience, assume that $p-q\in M_n(\K(E_0))$; this can be achieved by conjugating by a unitary in $M_n(\C)$, and helps streamline notation below, while making no real difference to the argument.  Again adopting the notation $e$ for $\begin{psmallmatrix} 1 & 0 \\ 0 & 0 \end{psmallmatrix}$, we compute that 
$$
\phi\psi(p,q)=\Bigg(\,v_n^*\begin{pmatrix} q & 1-q \\ 1-q & q \end{pmatrix} \begin{pmatrix} p & 0 \\ 0 & 1-q \end{pmatrix} \begin{pmatrix} q & 1-q \\ 1-q & q \end{pmatrix}v_n\, ,\,e\,\Bigg).
$$
Let $0_{2n-2}$ be the zero element in $M_{2n-2}(\K(E_0)^+)$.  Then the element above has the same class as
\begin{equation}\label{first phipsi}
\Bigg(\,v_n^*\begin{pmatrix} q & 1-q \\ 1-q & q \end{pmatrix} \begin{pmatrix} p & 0 \\ 0 & 1-q \end{pmatrix} \begin{pmatrix} q & 1-q \\ 1-q & q \end{pmatrix}v_n\oplus 0_{2n-2}\, ,\,e\oplus 0_{2n-2}\,\Big).
\end{equation}
Now, let $p_n:E_0^{\oplus 2n}\to E_0^{\oplus 2}$ be defined by projecting onto the first two coordinates, and define $w_n:=v_np_n$, which is a co-isometry in $\mathcal{B}((\ell^2)^{\oplus 2n})\subseteq \LL(E_0^{\oplus n})$ with source projection the projection onto the first two coordinates in $E_0^{\oplus 2n}$.  The space of all co-isometries in $\mathcal{B}((\ell^2)^{\oplus 2n})$ with source projection dominating the projection onto the first two coordinates is path connected\footnote{If $n=1$, such co-isometries are automatically unitary, but not in general.  Either way, it is a path connected space.} (in the norm topology).  Hence  we may connect $w_n$ through such co-isometries to one that acts as the identity on the first two coordinates, from which it follows that the element in line \eqref{first phipsi} represents the same class as
\begin{align*}
\Bigg(\,w_n\Big(v_n^*\begin{pmatrix} q & 1-q \\ 1-q & q \end{pmatrix} \begin{pmatrix} p & 0 \\ 0 & 1-q \end{pmatrix} & \begin{pmatrix} q & 1-q \\ 1-q & q \end{pmatrix}v_n\oplus 0_{2n-2}\Big)w_n^*\, ,\\&\,w_n(e\oplus 0_{2n-2})w_n^*\,\Bigg).
\end{align*}
Computing, this equals 
\begin{equation}\label{second phi psi}
\Bigg(\,\begin{pmatrix} q & 1-q \\ 1-q & q \end{pmatrix} \begin{pmatrix} p & 0 \\ 0 & 1-q \end{pmatrix}  \begin{pmatrix} q & 1-q \\ 1-q & q \end{pmatrix}\, ,\, \begin{pmatrix} 1 & 0 \\ 0 & 0 \end{pmatrix}\,\Bigg),
\end{equation}
where all blocks in the matrices appearing above are $n\times n$.  Note now if we write 
$$
r:=\frac{1}{2}\Bigg(\begin{pmatrix} q & 1-q \\ 1-q & q \end{pmatrix}+\begin{pmatrix} 1 & 0 \\ 0 & 1 \end{pmatrix}\Bigg),
$$
then $r$ is a projection such that $\|[r,a]\|<\epsilon$ for all $a\in X$.  For $t\in [0,\pi]$ define $u_t:=r+\exp(it)(1-r)$, so $(u_t)$ is a path of unitaries connecting $\begin{psmallmatrix} q & 1-q \\ 1-q & q \end{psmallmatrix}$ to the identity, and all $(u_t)$ satisfy $\|[u_t,a]\|<\epsilon$ for all $a\in X$.  Hence the path 
$$
\Bigg(\,u_t\begin{pmatrix} q & 1-q \\ 1-q & q \end{pmatrix} \begin{pmatrix} p & 0 \\ 0 & 1-q \end{pmatrix} \begin{pmatrix} q & 1-q \\ 1-q & q \end{pmatrix}u_t^*\, ,\, u_t\begin{pmatrix} 1 & 0 \\ 0 & 0 \end{pmatrix}u_t^*\,\Bigg)
$$
shows that the element in line \eqref{second phi psi} defines the same class as 
$$
\Bigg(\,\begin{pmatrix} p & 0 \\ 0 & 1-q \end{pmatrix} \, ,\, \begin{pmatrix} q & 1-q \\ 1-q & q \end{pmatrix}\begin{pmatrix} 1 & 0 \\ 0 & 0 \end{pmatrix}\begin{pmatrix} q & 1-q \\ 1-q & q \end{pmatrix}\,\Bigg),
$$
which equals $(p\oplus 1-q,q\oplus 1-q)$.  This last element defines the same class as $(p,q)$ by definition, however, so we are done.
\end{proof}

We are finally ready for the main result of this subsection.

\begin{proposition}\label{kku kkm}
Let $A$ and $B$ be separable $C^*$-algebras with $A$ unital.  Let $\pi:A\to \LL(E)$ be a graded, balanced, and strongly unitally absorbing representation of $A$ on a Hilbert $B$-module.  Write $(\pi,E)=(\pi_0\oplus \pi_0,E_0\oplus E_0)$ with $(\pi_0,E_0)$ unitally strongly absorbing.  

Then there are isomorphisms
\begin{equation}\label{kl mx iso 0}
KL(A,B)\to \lim_{\leftarrow} KK^{\pi_0,\text{mx}}_\epsilon(X,B).
\end{equation}
and 
$$
\lim_{\leftarrow}{}\!^1 KK^{{1\otimes \pi_0},\text{mx}}_\epsilon(X,SB)\to \overline{\{0\}},
$$ 
where the limits are taken over the directed set $\mathcal{X}$ of Definition \ref{seps} and $\overline{\{0\}}$ is the closure of $0$ in $KK(A,B)$.  Moreover, the isomorphism in line \eqref{kl mx iso 0} is a homeomorphism when the right hand side is equipped with the inverse limit topology. 

Finally, there is a short exact sequence 
$$
0\to \lim_{\leftarrow}{\!}^1 KK^{1\otimes\pi_0,\text{mx}}_\epsilon(X,SB)\to KK(A,B)\to \lim_{\leftarrow} KK^{\pi_0,\text{mx}}_\epsilon(X,B)\to 0.
$$
\end{proposition}

\begin{proof}
The proof follows from Lemma \ref{comp matrix}, quite analogously to that of Proposition \ref{kkl kku}.  We leave the details to the reader.
\end{proof}

Let us conclude with a final lemma on representation-independence, which is an analogue of Proposition \ref{usa rep gives kl} above.

\begin{corollary}\label{ua rep gives kl}
Let $A$ and $B$ be separable $C^*$-algebras with $A$ unital..  Then for any unitally absorbing representation $\pi:A\to \LL(E)$ we have that 
\begin{equation}\label{kl mx iso}
KL(A,B)\to \lim_{\leftarrow} KK^{\pi,\text{mx}}_\epsilon(X,B).
\end{equation}
and 
$$
\lim_{\leftarrow}{}\!^1 KK^{\pi,\text{mx}}_\epsilon(X,SB)\to \overline{\{0\}},
$$ 
where the limits are taken over the directed set $\mathcal{X}$ of Definition \ref{seps}, and $\overline{\{0\}}$ is the closure of $0$ in $KK(A,B)$.  Moreover, the isomorphism in line \eqref{kl mx iso} is a homeomorphism when the right hand side is equipped with the inverse limit topology. 

Finally, there is a short exact sequence 
$$
0\to \lim_{\leftarrow}{\!}^1 KK^{\pi,\text{mx}}_\epsilon(X,SB)\to KK(A,B)\to \lim_{\leftarrow} KK^{\pi,\text{mx}}_\epsilon(X,B)\to 0.
$$
\end{corollary}

\begin{proof}
Proposition \ref{kku kkm} tells us that the result is true whenever there is a balanced, graded, unitally strongly absorbing representation $(\sigma,F)$ with $(\sigma,F)=(\pi\oplus\pi,E\oplus E)$.  However, this is always true, so we are done.
\end{proof}

\bibliography{Generalbib}

\end{document}